\documentclass[a4paper,12pt]{article}
\usepackage[top=1in, left=1in]{geometry}
\usepackage{tikz}
\usepackage{pgf}
\usepackage{relsize}
\usepackage{subfig}
\usetikzlibrary{arrows,automata}
\usetikzlibrary{fit}
\usepackage{amsmath,amssymb,amsthm,url,amsfonts}
\usepackage{graphics}
\usepackage{float}
\usepackage[numbers]{natbib}
\usepackage{textcomp}
\usetikzlibrary{arrows,%
                petri,%
                topaths}%

\theoremstyle{definition}

\newtheorem{assumption}{Assumption}

\newtheorem{definition}{Definition}

\usepackage{enumerate}
\usepackage{verbatim}
\usepackage{pgf}

\newcommand{\be}{\begin{equation}}
\newcommand{\ee}{\end{equation}}

\newcommand{\als}[1]{\begin{align*}#1\end{align*}}

\usepackage{setspace}

\newtheorem{theorem}{Theorem}[section]
\newtheorem{lemma}[theorem]{Lemma}

\newcommand{\norm}[1]{\left|\left|#1\right|\right|}
\newcommand{\alignShort}[1]{\begin{align*}#1\end{align*}}

\makeatother
\begin{document}
\begin{titlepage}
\title{A Distributed Newton Method for Network Utility Maximization \footnote{This work was supported by National Science Foundation under Career grant DMI-0545910, the
DARPA ITMANET program, ONR MURI N000140810747 and AFOSR Complex Networks Program.}
}

\author{Ermin Wei\thanks{Department of Electrical Engineering and Computer Science,
Massachusetts Institute of Technology},
Asuman Ozdaglar\footnotemark[2], and
Ali Jadbabaie \thanks{Department of Electrical
and Systems Engineering and GRASP Laboratory, University of Pennsylvania}
}
\date{\today}
\maketitle

\begin{abstract}
Most existing work uses dual decomposition and first-order methods
to solve Network Utility Maximization (NUM) problems in a
distributed manner, which suffer from slow rate of convergence
properties. This paper develops an alternative distributed
Newton-type fast converging algorithm for solving NUM problems with
self-concordant utility functions. By using novel matrix splitting
techniques, both primal and dual updates for the Newton step can be
computed using iterative schemes in a decentralized manner. We propose a stepsize rule and provide a distributed procedure to compute it in finitely many iterations. The key feature of our direction and stepsize computation schemes is that both are implemented using the same distributed information exchange mechanism employed by first order methods. We show that even when the
Newton direction and the stepsize in our method are computed within
some error (due to finite truncation of the iterative schemes), the
resulting objective function value still converges superlinearly in terms of primal iterations to
an explicitly characterized error neighborhood. Simulation results
demonstrate significant convergence rate improvement of our
algorithm relative to the existing first-order methods based on dual
decomposition.
\end{abstract}
\thispagestyle{empty}
\end{titlepage}

\section{Introduction}\label{sec:intro}

Most of today's communication networks are large-scale and comprise
of agents with heterogeneous preferences. Lack of access to
centralized information in such networks necessitate design of
distributed control algorithms that can operate based on locally
available information. Some applications include routing and
congestion control in the Internet, data collection and processing
in sensor networks, and cross-layer design in wireless networks.
This work focuses on the rate control problem in wireline networks,
which can be formulated in the {\it Network Utility Maximization
(NUM)} framework proposed in  \cite{KellyMaTan} (see also
\cite{LowLapsley}, \cite{Srikant}, and \cite{mung}). NUM problems
are characterized by a fixed network and a set of sources, which
send information over the network along predetermined routes. Each
source has a local utility function over the rate at which it sends
information. The goal is to determine the source rates that maximize
the sum of utilities subject to link capacity constraints. The
standard approach for solving NUM problems relies on using dual
decomposition and subgradient (or first-order) methods, which
through a price exchange mechanism among the sources and the links
yields algorithms that can operate on the basis of local
information.\footnote{The price exchange mechanism involves
destinations (end nodes of a route) sending route prices (aggregated
over the links along the route) to sources, sources updating their
rates based on these prices and finally links updating prices based
on new rates sent over the network.} One major shortcoming of this
approach is the slow rate of convergence.

In this paper, we propose a novel Newton-type second-order method
for solving the NUM problem in a distributed manner, which leads to
significantly faster convergence. Our approach involves transforming
the inequality constrained NUM problem to an equality-constrained
one through introducing slack variables and logarithmic
barrier functions, and using an equality-constrained Newton method
for the reformulated problem. There are two challenges in
implementing this method in a distributed manner. First challenge is
the computation of the Newton direction. This computation involves a
matrix inversion, which is costly and requires global information.
We solve this problem by using an iterative scheme based on a novel
matrix splitting technique. Since the objective function of the
(equality-constrained) NUM problem is {\it separable}, i.e., it is
the sum of functions over each of the variables, this splitting
enables computation of the Newton direction using decentralized
algorithms based on limited ``scalar" information exchange between
sources and links. This exchange involves destinations iteratively
sending route prices (aggregated link prices or dual variables along
a route) to the sources, and sources sending the route price scaled
by the Hessian to the links along its route. Therefore, our
algorithm has \textit{comparable level of information exchange} with the
first-order methods applied to the NUM problem.

The second challenge is related to the computation of a stepsize
rule that can guarantee local superlinear convergence of the primal iterations. Instead of
the iterative backtracking rules typically used with Newton methods,
we propose a stepsize rule which is inversely proportional to the
inexact Newton decrement (where the inexactness arises due to errors
in the computation of the Newton direction) if this decrement is
above a certain threshold and takes the form of a pure Newton step
otherwise. Computation of the inexact Newton decrement involves aggregating local information from the sources and links in the network. We propose a novel distributed procedure for computing the inexact Newton decrement in finite number of steps using again the same information exchange mechanism employed by first order methods.

Since our method uses iterative schemes to compute the Newton
direction, exact computation is not feasible.
Another major contribution of our work is to consider a truncated
version of this scheme, allow error in stepsize computation and present convergence rate analysis of
the constrained Newton method when the stepsize and the Newton
direction are estimated with some error. We show that when these
errors are sufficiently small, the value of the objective function
converges superlinearly in terms of primal iterations to a neighborhood of the optimal objective
function value, whose size is explicitly quantified as a function of
the errors and bounds on them.

Our work contributes to the growing literature on distributed
optimization and control of multi-agent networked systems. There are
two standard approaches for designing distributed algorithms for
such problems. The first approach, as mentioned above, uses dual
decomposition and subgradient methods, which for some problems
including NUM problems lead to iterative distributed algorithms (see
\cite{KellyMaTan}, \cite{LowLapsley}). Subsequent work by Athuraliya
and Low in \cite{Low} use diagonal scaling to approximate Newton
steps to speed up the subgradient algorithm while maintaining their
distributed nature. Despite improvements in speed over the
first-order methods, as we shall see, the performance of this
modified algorithm does not achieve the rate gains obtained by
second-order methods.

The second approach involves considering \textit{consensus-based}
schemes, in which agents exchange local estimates with their
neighbors with the goal of aggregating information over an exogenous
(fixed or time-varying) network topology (see \cite{johnthes},
\cite{Blondel}, \cite{ConsensusDelay}, \cite{distasyn}, \cite{ali},
\cite{murray}, \cite{AsuNewton} and \cite{Consensus2}). It has been
shown that under some mild assumption on the connectivity of the
graph and updating rules, the distance from the vector formed by current estimates to consensus diminishes linearly. Consensus schemes can be used to
compute the average of local values or more generally as a building
block for developing distributed optimization algorithms with
linear/sublinear rate of convergence (\cite{NedicSubgradientConsensus}). The stepsize for the distributed Newton method can be computed using consensus type of algorithms. However, the distributed Newton method achieves quadratic rate of convergence for the primal iterations, using consensus results in prohibitively slow stepsize computation at each iteration, and is hence avoided in our method.

Other than the papers cited above, our paper is also related to
\cite{BetGaf83}, \cite{Klincewicz83}, \cite{BoydBP} and
\cite{AsuNewton}. In \cite{BetGaf83}, Bertsekas and Gafni studied a
projected Newton method for optimization problems with twice
differentiable objective functions and simplex constraints. They
proposed finding the Newton direction (exactly or approximately)
using a conjugate gradient method. This work showed that when
applied to multi-commodity network flow problems, the conjugate
gradient iterations can be obtained using simple graph operations,
however did not investigate distributed implementations. Similarly,
in \cite{Klincewicz83}, Klincewicz proposed a Newton method for
network flow problems that computes the dual variables at each step
using an iterative conjugate gradient algorithm. He showed that
conjugate gradient iterations can be implemented using a
``distributed" scheme that involves simple operations and information
exchange along a spanning tree. Spanning tree based computations
involve passing all information to a centralized node and may
therefore be restrictive for NUM problems which are characterized by
decentralized (potentially autonomous) sources.

In \cite{BoydBP}, the authors have developed a distributed
Newton-type method for the NUM problem using a belief propagation
algorithm. Belief propagation algorithms, while performing well in
practice, lack systematic convergence guarantees. Another recent
paper \cite{AsuNewton} studied a Newton method for
equality-constrained network optimization problems and presented a
convergence analysis under Lipschitz assumptions. In this paper, we
focus on an inequality-constrained problem, which is reformulated as
an equality-constrained problem using barrier functions. Therefore,
this problem does not satisfy Lipschitz assumptions. Instead, we
assume that the utility functions are self-concordant and present a
novel convergence analysis using properties of self-concordant
functions.

Our analysis for the convergence of the algorithm also relates to
work on convergence rate analysis of inexact Newton methods (see
\cite{dembo}, \cite{inexactKelley}). These works focus on providing
conditions on the amount of error at each iteration relative to the
norm of the gradient of the current iterate that ensures superlinear
convergence to the {\it exact optimal solution} (essentially
requiring the error to vanish in the limit). Even though these
analyses can provide superlinear rate of convergence, the vanishing
error requirement can be too restrictive for practical
implementations. Another novel feature of our analysis is the
consideration of \textit{convergence to an approximate neighborhood of the
optimal solution}. In particular, we allow a fixed error level to be
maintained at each step of the Newton direction computation and show
that superlinear convergence is achieved by the primal iterates to
an error neighborhood, whose size can be controlled by tuning the
parameters of the algorithm. Hence, our work also contributes to the
literature on error analysis for inexact Newton methods.

The rest of the paper is organized as follows: Section
\ref{sec:model} defines the problem formulation and related
transformations. Section \ref{sec:alg} describes the exact
constrained primal-dual Newton method for this problem. Section
\ref{sec:inex} presents a distributed iterative scheme for computing
the dual Newton step and the distributed inexact Newton-type
algorithm. Section \ref{sec:analysis} contains the rate of
convergence analysis for our algorithm. Section \ref{sec:sim}
presents simulation results to demonstrate convergence speed
improvement of our algorithm to the existing methods with linear
convergence rates. Section \ref{sec:conclu} contains our concluding
remarks.

\vskip 0.2pc

\noindent\textbf{Basic Notation and Notions:}

A vector is viewed as a column vector, unless clearly stated
otherwise. We write $\mathbb{R}_+$ to denote the set of nonnegative
real numbers, i.e., $\mathbb{R}_+ = [0,\infty)$. We use subscripts
to denote the components of a vector and superscripts to index a
sequence, i.e., $x_i$ is the $i^{th}$ component of vector $x$ and
$x^k$ is the $k$th element of a sequence. When $x_{i}\geq 0$ for all
components $i$ of a vector $x$, we write $x\geq 0$.

For a matrix $A$, we write $A_{ij}$ to denote the matrix entry in
the $i^{th}$ row and $j^{th}$ column, and $[A]_i$ to denote the
$i^{th}$ column of the matrix $A$, and $[A]^j$ to denote the
$j^{th}$ row of the matrix $A$. We write $I(n)$ to denote the
identity matrix of dimension $n\times n$. We use $x'$ and $A'$ to
denote the transpose of a vector $x$ and a matrix $A$ respectively.
For a real-valued function $f:X\rightarrow \mathbb{R}$, where $X$ is
a subset of $\mathbb{R}^n$, the gradient vector and the Hessian
matrix of $f$ at $x$ in $X$ are denoted by $\nabla f(x)$ and
$\nabla^2 f({x})$ respectively. We use the vector $e$ to denote the
vector of all ones.

A real-valued convex function $ g: X\rightarrow \mathbb{R}$, where
$X$ is a subset of $\mathbb{R}$, is \textit{self-concordant} if it is three times continuously differentiable and
$|{g}'''(x)| \leq 2{g}''(x)^\frac{ 3}{2}$ for all $x$ in its
domain.\footnote{Self-concordant functions are defined through the
following more general definition: a real-valued three times continuously differentiable convex function $
g: X\rightarrow \mathbb{R}$, where $X$ is a subset of $\mathbb{R}$,
is \textit{self-concordant}, if there exists a constant $a>0$, such
that $|{g}'''(x)| \leq 2a^{-\frac{1}{2}}{g}''(x)^\frac{3}{2}$ for
all $x$ in its domain \cite{InteriorBook}, \cite{Jarre}. Here we
focus on the case $a=1$ for notational simplification in the
analysis.} For real-valued functions in
$\mathbb{R}^n$, a convex function $g: X\rightarrow \mathbb{R}$,
where $X$ is a subset of $\mathbb{R}^n$, is self-concordant if it is
self-concordant along every direction in its domain, i.e., if the
function $\tilde g(t) = g(x+tv)$ is self-concordant in $t$ for all
$x$ and $v$. Operations that preserve self-concordance property
include summing, scaling by a factor $\alpha\geq 1$, and composition
with affine transformation (see \cite{Boyd} Chapter $9$ for more
details).

\section{Network Utility Maximization Problem}\label{sec:model}
We consider a network represented by a set $\mathcal{L}= \{1, ...,
L\}$ of (directed) links of finite nonzero capacity given by
$c=[c_l]_{l\in\mathcal{L}}$ with $c>0$. The network is shared by a set
$\mathcal{S} = \{1, ..., S\}$ of sources, each of which transmits
information along a predetermined route. For each link $l$, let
$S(l)$ denote the set of sources use it. For each source $i$, let
$L(i)$ denote the set of links it uses. We also denote the
nonnegative source rate vector by $s =[s_i]_{i\in\mathcal{S}}$. The
capacity constraint at the links can be compactly expressed as
\alignShort{Rs \leq c,} where $R$ is the \textit{routing
matrix}\footnote{This is also referred to as the \textit{link-source
incidence matrix} in the literature. Without loss of generality, we assume that each
source flow traverses at least one link, each link is used by at least one source and the links form a connected graph.} of dimension  $L\times S$,
i.e.,
 \begin{equation}\label{eq:Rmatrix}
 R_{ij} = \left\{ \begin{array}{rl}
    1 & \mbox{if link $i$ is on the route of source $j$},\\
    0 & \mbox{otherwise}.
    \end{array}
    \right.
 \end{equation}

We associate a utility function $U_i:\mathbb{R}_+\to\mathbb{R}$ with
each source $i$, i.e., $U_i(s_i)$ denotes the utility of source $i$
as a function of the source rate $s_i$. We assume the utility
functions are additive, such that the overall utility of the network
is given by $\sum_{i =1}^{S} U_i(s_i)$. Thus the Network Utility
Maximization(NUM) problem can be formulated as
\begin{align}\label{ineqFormulation}
 \mbox{maximize} \quad &\sum_{i =1}^{S} U_i(s_i)\\\nonumber
 \mbox{subject to} \quad &Rs \leq c,\\\nonumber
 & s \geq 0.
\end{align}
We adopt the following assumption.
\begin{assumption}\label{asmp:utility}
The utility functions $U_i:\mathbb{R}_+\to\mathbb{R}$ are strictly concave, monotonically nondecreasing
 on $(0, \infty)$. The functions
$-U_i:\mathbb{R}_+\to\mathbb{R}$ are self-concordant on
$(0,\infty)$.
\end{assumption}

The self-concordance assumption is satisfied by standard utility functions considered in the literature, for instance logarithmic, i.e., weighted proportional fair utility functions \cite{Srikant},  and concave quadratic utility functions, and is adopted here to allow a self-concordant analysis in establishing local quadratic convergence. We use $h(x)$ to denote the (negative of the)
objective function of problem (\ref{ineqFormulation}), i.e., $h(x) =
-\sum_{i =1}^{S} U_i({x_i})$, and $h^*$ to denote the (negative of
the) optimal value of this problem.\footnote{We consider the
negative of the objective function value to work with a minimization
problem.} Since $h(x)$ is continuous and the feasible set of problem
(\ref{ineqFormulation}) is compact, it follows that problem
(\ref{ineqFormulation}) has an optimal solution, and therefore $h^*$
is finite. Moreover, the interior of the feasible set is nonempty,
i.e., there exists a feasible solution $x$ with $x_i =
\frac{\underline{c}}{S+1}$ for all $i\in\mathcal{S}$ with
$\underline{c}>0$.\footnote{One possible value for $\underline{c}$
is $\underline{c}= \min_l\{c_l\}$.}

To facilitate the development of a distributed Newton-type method,
we consider a related equality-constrained problem by introducing
nonnegative slack variables $[y_l]_{l\in\mathcal{L}}$ for the
capacity constraints, defined by
\be\label{eq:slack}
\sum_{j=1}^{S}R_{lj}s_{j}+{y}_{l} = c_l \quad \textrm {for }l = 1, 2
\ldots L, \ee and logarithmic barrier functions for the
nonnegativity constraints (which can be done since the feasible set
of (\ref{ineqFormulation}) has a nonempty interior).\footnote{We
adopt the convention that $\log(x)=-\infty$ for $x\leq 0$.} We
denote the new decision vector by $x =
([s_i]'_{i\in\mathcal{S}},[y_l]'_{l\in\mathcal{L}})'$. This problem
can be written as
\begin{align}\label{eqFormulation}
 \mbox{minimize} \quad &-\sum_{i =1}^{S} U_i({x_i})-\mu\sum_{i =1}^{S+L} \log{({x_i})}\\\nonumber
 \mbox{subject to} \quad &A{x} = c,
\end{align}
where $A$ is the $L\times (S+L)$-dimensional matrix given by
\be\label{defA} A = [R \quad I(L)],\ee and $\mu$ is a nonnegative
barrier function coefficient. We use $f(x)$ to denote the objective
function of problem (\ref{eqFormulation}), i.e., \be\label{objfunc-eqformulation} f( x) =
-\sum_{i =1}^{S} U_i({x_i})-\mu\sum_{i =1}^{S+L}
\log{({x}_i)},\ee and $f^*$ to denote
the optimal value of this problem, which is finite for positive $\mu$.\footnote{This problem has a feasible solution, hence $f^*$ is upper bounded. Each of the variable $x_i$ is upper bounded by $\bar{c}$, where $\bar{c}=\max_l\{c_l\}$, hence by monotonicity of utility and logarithm functions, the optimal objective function value is lower bounded. Note that in the optimal solution of problem (\ref{eqFormulation}) $x_i\neq 0$ for all $i$, due to the logarithmic barrier functions.}

By Assumption \ref{asmp:utility}, the function $f(x)$ is separable,
strictly convex, and has a
positive definite diagonal Hessian matrix on the positive orthant. The function $f(x)$ is also self-concordant for $\mu\geq1$, since both summing and
scaling by a factor $\mu\geq 1$ preserve self-concordance property.

We write the optimal solution of problem (\ref{eqFormulation}) for a
fixed barrier function coefficient $\mu$ as $x(\mu)$. One can show
that as the barrier function coefficient $\mu$ approaches 0, the
optimal solution of problem (\ref{eqFormulation}) approaches that of
problem (\ref{ineqFormulation}), when the constraint set in
(\ref{ineqFormulation}) has nonempty interior and is convex
\cite{ConvexBertsekas}, \cite{oldInt}. Hence by continuity from
Assumption \ref{asmp:utility}, $h(x(\mu))$ approaches $h^*$.
Therefore, in the rest of this paper, unless clearly stated
otherwise, we study \textit{iterative distributed methods} for
solving problem (\ref{eqFormulation}) for a given $\mu$. In order to
preserve the self-concordance property of the function $f$, which
will be used in our convergence analysis, we first develop a
Newton-type algorithm for $\mu\geq 1$. In Section \ref{subsec:mu},
we show that problem (\ref{eqFormulation}) for any $\mu>0$ can be
tackled by solving two instances of problem (\ref{eqFormulation})
with different coefficients $\mu\ge 1$, leading to a solution
$x(\mu)$ that satisfies $\frac{h(x(\mu))-h^*}{h^*}\leq a$ for any
positive scalar $a$.

\section{Exact Newton Method}\label{sec:alg}

For each fixed $\mu$, problem (\ref{eqFormulation}) is feasible and has a convex
objective function, affine constraints, and a finite optimal value
$f^*$. Therefore, we can use a strong duality theorem to show that,
for problem (\ref{eqFormulation}), there is no duality gap  and
there exists a dual optimal solution (see \cite{Asubook}). Moreover,
since matrix $A$ has full row rank,

we can use a (feasible start)
equality-constrained Newton method to solve problem
(\ref{eqFormulation})(see \cite{Boyd} Chapter 10).

\subsection{Feasible Initialization}\label{sec:init}

We initialize the algorithm with some feasible and strictly positive
vector $x^0$. For example, one such initial vector is given by
\begin{align}\label{init}
x_i^0 &= \frac{\underline{c}}{S+1}\quad \textrm{for } i = 1, 2
\ldots S, \\ \nonumber
     \quad    x_{l+S}^0 &= c_l - \sum_{j=1}^{S}R_{lj} \frac{\underline{c}}{S+1}
     \quad\textrm{for } l = 1, 2 \ldots L,
\end{align}
where $c_l$ is the finite capacity for link $l$, $\underline{c}$ is
the minimum (nonzero) link capacity, $S$ is the total number of
sources in the network, and $R$ is routing matrix [cf. Eq.
(\ref{eq:Rmatrix})].

\subsection{Iterative Update Rule}
Given an initial feasible vector $x^0$, the algorithm generates the iterates by
\begin{align}
 x^{k+1} =  x^k + d^k \Delta   x^k,\label{eq:primalUpdate}
\end{align}
where $d^k$ is a positive stepsize,  $\Delta x^k$ is the (primal)
Newton direction given as the solution of the following system of
linear equations:\footnote{This is a primal-dual method with the
vectors $\Delta x^k$ and $w^k$ acting as primal direction and dual
variables respectively}.
\begin{align}\label{eq:newtonUpdate}
\left( \begin{array}{cc}
\nabla^2 f({x}^k)& A' \\
A & 0
\end{array}
\right)
\left(\begin{array}{c}
\Delta  x^k\\
w^k
\end{array}
\right) =
-\left(\begin{array}{c}
\nabla f({x}^k)\\
0
\end{array}
\right).
\end{align}
We will refer to $x^k$ as the {\it primal vector} and $w^k$ as the
{\it dual vector} (and their components as primal and dual variables
respectively). We also refer to $w^k$ as the {\it price vector}
since the dual variables $[w_l^k]_{l\in\mathcal{L}}$ associated with
the link capacity constraints can be viewed as prices for using
links. For notational convenience, we will use $H_k = \nabla^2
f({x}^k)$ to denote the Hessian matrix in the rest of the paper.

Solving for $\Delta x^k$ and $w^k$ in the preceding system yields
\begin{align}
&\Delta   x^k = -H_k^{-1}(\nabla f(  x^k) + A'w^k), \label{eq:primalDirection}\\
&(AH_k^{-1}A')w^k = -AH_k^{-1}\nabla f(  x^k).\label{eq:dualUpdateBackground}
\end{align}

This system has a unique solution for all $k$. To see this, note
that the matrix $H_k$ is a diagonal matrix with entries
\begin{align}\label{eq:Hdiag}
(H_k)_{ii}= \left\{
\begin{array}{ccc}
-\frac{\partial^2 U_i( x_i^k)}{\partial  x_i^2}+\frac{\mu}{ (x_i^k)^2} \quad&  1 \leq i \leq S,\\
\frac{\mu}{ (x_i^k)^2}  \quad&  S+1 \leq i \leq S+L.
\end{array}
\right.
\end{align}
By Assumption \ref{asmp:utility}, the functions $U_i$ are strictly
concave, which implies $\frac{\partial^2 U_i( x_i^k)}{\partial
x_i^2}\leq 0$. Moreover, the primal vector $x^k$ is bounded (since
the method maintains feasibility) and, as we shall see in Section
\ref{sec:stepsize}, can be guaranteed to remain strictly positive by
proper choice of stepsize. Therefore, the entries $(H_k)_{ii}>0$ and
are well-defined for all $i$, implying that the Hessian matrix $H_k$ is
invertible. Due to the structure of $A$ [cf.\ Eq.\ (\ref{defA})],
the column span of $A$ is the entire space $\mathbb{R}^L$, and hence
the matrix $AH_k^{-1}A'$ is also invertible.\footnote{If for some
$x\in \mathbb{R}^L$, we have $AH_k^{-1}A'x=0$, then
$x'AH_k^{-1}A'x=\norm{H_k^{-\frac{1}{2}}A'x}_2=0$, which implies
$\norm{A'x}_2=0$, because the matrix $H$ is invertible. The rows of
the matrix $A'$ span $\mathbb{R}^L$, therefore we have $x=0$. This
shows that the matrix $AH_k^{-1}A'$ is invertible.} This shows that
the preceding system of linear equations can be solved uniquely for
all $k$.

The objective function $f$ is separable in $x_i$, therefore given
the vector $w^k_l$ for $l$ in $L(i)$, the Newton direction $\Delta
x^k_i$ can be computed by each source $i$  using local information
available to that source. However, the computation of the vector
$w^k$ at a given primal solution $x^k$ cannot be implemented in a
decentralized manner since the evaluation of the matrix inverse
$(AH_k^{-1}A')^{-1}$ requires global information. The following
section provides a distributed inexact Newton method, based on
computing the vector $w^k$ using a decentralized iterative scheme.

\section{Distributed Inexact Newton Method}\label{sec:inex}

In this section, we introduce a distributed Newton method using
ideas from matrix splitting in order to compute the dual vector
$w^k$ at each $k$ using an iterative scheme. Before proceeding to
present the details of the algorithm, we first introduce some
preliminaries on matrix splitting.

\subsection{Preliminaries on Matrix Splitting}\label{sec:matrixSplitting}
Matrix splitting can be used to solve a system of linear equations
given by \alignShort{Gy=a,} where $G$ is an $n\times n$ matrix and
$a$ is an $n$-dimensional vector. Suppose that the matrix $G$ can be
expressed as the sum of an invertible matrix $M$ and a matrix $N$,
i.e., \be\label{eq:matrixsum}G = M+N.\ee Let $y_0$ be an arbitrary
$n$-dimensional vector. A sequence $\{y^k\}$ can be generated by the
following iteration: \be\label{eq:matrixSplitting} y^{k+1} =
-M^{-1}Ny^k+M^{-1}a . \ee It can be seen that the sequence $\{y^k\}$
converges as $k\to\infty$ if and only if the spectral radius of the
matrix $M^{-1}N$ is strictly bounded above by 1. When the sequence
$\{y^k\}$ converges, its limit $y^*$ solves the original linear
system, i.e., $Gy^* =a$ (see \cite{Berman} and \cite{Cottle} for
more details). Hence, the key to solving the linear equation via
matrix splitting is the bound on the spectral radius of the matrix
$M^{-1}N$. Such a bound can be obtained using the following result
(see Theorem 2.5.3 from \cite{Cottle}).
\begin{theorem}\label{thm:splitting} Let $G$ be a real symmetric matrix.
Let $M$ and $N$ be matrices such that $G=M+N$ and assume that $M$ is
invertible and both matrices $M+N$ and $M-N$ are positive definite.
Then the spectral radius of $M^{-1}N$, denoted by $\rho(M^{-1}N)$,
satisfies $\rho(M^{-1}N)<1$.
\end{theorem}

By the above theorem, if $G$ is a real, symmetric, positive
definite matrix and $M$ is a nonsingular matrix, then one sufficient condition for the iteration
(\ref{eq:matrixSplitting}) to converge is that the matrix $M-N$ is
positive definite. This can be guaranteed using Gershgorin Circle
Theorem, which we introduce next (see \cite{Gershgorin} for more
details).

\begin{theorem}(Gershgorin Circle Theorem)\label{thm:gershgorin}
Let $G$ be an $n\times n$ matrix, and define $r_i(G) = \sum_{j \neq
i} |G_{ij}|$. Then, each eigenvalue of $G$ lies in one of the
Gershgorin sets $\{\Gamma_i\}$, with $\Gamma_i$ defined as disks in
the complex plane, i.e., \be \Gamma_i = \{z\in \mathbb{C} \mid |z-G_{ii}|
\leq r_i(G)\}. \nonumber \ee
\end{theorem}
One corollary of the above theorem is that if a matrix G is strictly
diagonally dominant, i.e., $|G_{ii}|> \sum_{j \neq i} |G_{ij}|$, and
$G_{ii}>0$ for all $i$, then the real parts of all the eigenvalues
lie in the positive half of the real line, and thus the matrix is
positive definite. Hence a sufficient condition for the matrix $M-N$
to be positive definite is that $M-N$ is strictly diagonally
dominant with strictly positive diagonal entries.

\subsection{Distributed Computation of the Dual Vector}\label{sec:dual}
We use the matrix splitting scheme introduced in the preceding
section to compute the dual vector $w^k$ in Eq.\
(\ref{eq:dualUpdateBackground}) in a distributed manner. Let $D_k$
be a diagonal matrix, with diagonal entries \be\label{eq:defD}
(D_k)_{ll}=(AH_k^{-1}A')_{ll},\ee and matrix $B_k$ be given by
\be\label{eq:defB} B_k = AH_k^{-1}A' - D_k.\ee Let matrix
$\bar{B}_k$ be a diagonal matrix, with diagonal entries
\be\label{eq:defBbar} (\bar{B}_k)_{ii} = \sum_{j=1}^{L}
(B_k)_{ij}.\ee By splitting the matrix $AH_k^{-1}A'$ as the sum of
$D_k+\bar{B}_k$ and $B_k-\bar{B}_k$, we obtain the following result.

\begin{theorem}\label{thm:dualSplitting}
For a given $k>0$, let $D_k$, $B_k$, $\bar B_k$ be the
matrices defined in Eqs.\ (\ref{eq:defD}), (\ref{eq:defB}) and
(\ref{eq:defBbar}). Let $w(0)$ be an arbitrary initial vector and
consider the sequence $\{w(t)\}$ generated by the iteration
\be\label{eq:dualInnerIteration} w(t+1) =
(D_k+\bar{B}_k)^{-1}(\bar{B}_k-B_k)w(t)+(D_k+\bar{B}_k)^{-1}(-AH_k^{-1}\nabla
f(x^k)), \ee for all $t\geq0$. Then the spectral radius of the
matrix $(D_k+\bar{B}_k)^{-1}(B_k-\bar{B}_k)$ is strictly bounded
above by $1$ and the sequence $\{w(t)\}$ converges as $t\to\infty$,
and its limit is the solution to Eq.\
(\ref{eq:dualUpdateBackground}).
\end{theorem}

\begin{proof}
We split the matrix $AH_k^{-1}A'$ as \be\label{eq:splitting}
(AH_k^{-1}A') = (D_k+\bar{B}_k) + (B_k-\bar{B}_k) \ee and use the
iterative scheme presented in Eqs.\ (\ref{eq:matrixsum}) and
(\ref{eq:matrixSplitting}) to solve Eq.\
(\ref{eq:dualUpdateBackground}). For all $k$, both the real matrix
$H_k$ and its inverse, $H_k^{-1}$, are positive definite and
diagonal. The matrix $A$ has full row rank and is element-wise
nonnegative. Therefore the product $AH_k^{-1}A'$ is real, symmetric,
element-wise nonnegative and positive definite. We let
\be\label{eq:Qmatrix}Q_k= (D_k+\bar{B}_k) - (B_k-\bar{B}_k) =
D_k+2\bar{B}_k - B_k\ee denote the difference matrix. By definition
of $\bar B_k$ [cf.\ Eq.\ (\ref{eq:defBbar})], the matrix
$2\bar{B}_k-B_k$ is diagonally dominant, with nonnegative diagonal
entries. Moreover, due to strict positivity of the second
derivatives of the logarithmic barrier functions, we have
$(D_k)_{ii}>0$ for all $i$. Therefore the matrix $Q_k$ is strictly
diagonally dominant. By Theorem \ref{thm:gershgorin}, such matrices
are positive definite. Therefore, by Theorem \ref{thm:splitting},
the spectral radius of the matrix
$(D_k+\bar{B}_k)^{-1}(B_k-\bar{B}_k)$ is strictly bounded above by
$1$. Hence the splitting scheme (\ref{eq:splitting}) guarantees the
sequence $\{w(t)\}$ generated by iteration
(\ref{eq:dualInnerIteration}) to converge to the solution of Eq.\
(\ref{eq:dualUpdateBackground}).
\end{proof}

This provides an iterative scheme to compute the dual vector $w^k$
at each primal iteration $k$ using an iterative scheme. We will
refer to the iterative scheme defined in Eq.\
(\ref{eq:dualInnerIteration}) as the {\it dual iteration}.

There are many ways to split the matrix $AH_k^{-1}A'$. The
particular one in Eq.\ (\ref{eq:splitting}) is chosen here due to
three desirable features. First it guarantees that the difference
matrix $Q_k$ [cf. Eq.\ (\ref{eq:Qmatrix})] is strictly diagonally
dominant, and hence ensures convergence of the sequence $\{w(t)\}$.
Second, with this splitting scheme, the matrix $D_k+\bar{B}_k$ is
diagonal, which eliminates the need for global information when
calculating its inverse. The third feature enables us to study
convergence rate of iteration (\ref{eq:dualInnerIteration}) in terms
of a dual (routing) graph which we introduce next.

\begin{definition}\label{def:dualGraph}
Consider a network $\mathcal{G}=\{\mathcal{L},\mathcal{S}\}$,
represented by a set $\mathcal{L}= \{1, ..., L\}$ of (directed)
links, and a set $\mathcal{S} = \{1, ..., S\}$ of sources. The links
form a strongly connected graph, and each source sends information
along a predetermined route. The {\it weighted dual (routing) graph}
$\tilde{\mathcal{G}} = \{\tilde{\mathcal{N}},\tilde{\mathcal{L}}\}$,
where $\tilde{\mathcal{N}}$ is the set of nodes, and
$\tilde{\mathcal{L}}$ is the set of (directed) links defined by:\\
\textrm{A.} $\tilde{\mathcal{N}} = \mathcal{L}$;\\
\textrm{B.} A link is present between node $L_i$ to $L_j$ in $\tilde{\mathcal{G}}$ if and only if
there is some common flow between $L_i$ and $L_j$ in $\mathcal{G}$.\\
\textrm{C.} The weight $\tilde W_{ij}$ on the link from node $L_i$
to $L_j$ is given by
\[\tilde W_{ij}=(D_k+\bar{B}_k)^{-1}_{ii}(B_k)_{ij}=(D_k+\bar{B}_k)^{-1}_{ii}(AH_k^{-1}A')_{ij}=(D_k+\bar{B}_k)^{-1}_{ii}\sum_{s\in
S(i)\cap S(j)}H_{ss}^{-1},\] where the matrices $D_k$, $B_k$, and
$\bar{B}_k$ are defined in Eqs.\ (\ref{eq:defD}), (\ref{eq:defB})
and (\ref{eq:defBbar}).
\end{definition}

One example of a network and its dual graph are presented in Figures
\ref{fig:withFlow} and \ref{fig:dual}. Note that the unweighted
indegree and outdegree of a node are the same in the dual graph,
however the weights are different depending on the direction of the
links. The splitting scheme in Eq.\ (\ref{eq:splitting}) involves the
matrix $(D_k+\bar{B}_k)^{-1}(\bar{B}_k-B_k)$, which is the weighted
Laplacian matrix of the dual graph.\footnote{We adopt the following
definition for the weighted Laplacian matrix of a graph. Consider a
weighted directed graph $\mathcal{G}$ with weight $W_{ij}$
associated with the link from node $i$ to $j$. We let $W_{ij}= 0$
whenever the link is not present. These weights form a
\textit{weighted adjacency matrix }$W$. The \textit{weighted
out-degree matrix} $D$ is defined as a diagonal matrix with
$D_{ii}=\sum_{j}W_{ij}$ and the \textit{weighted Laplacian matrix}
$L$ is defined as $L = D-W$. See \cite{graphTheory},
\cite{ChungGraphBook} for more details on graph Laplacian matrices.}
The weighted out-degree of node $i$ in the dual graph, i.e., the
diagonal entry $(D_k+\bar{B}_k)^{-1}_{ii}\bar B_{ii}$ of the
Laplacian matrix, can be viewed as a measure of the {\it congestion
level} of a link in the original network since the neighbors in the
dual graph represent links that share flows in the original network.
We show in Section \ref{sec:ConvDual1} that the spectral properties
of the Laplacian matrix of the dual graph dictate the convergence
speed of dual iteration (\ref{eq:dualInnerIteration}).

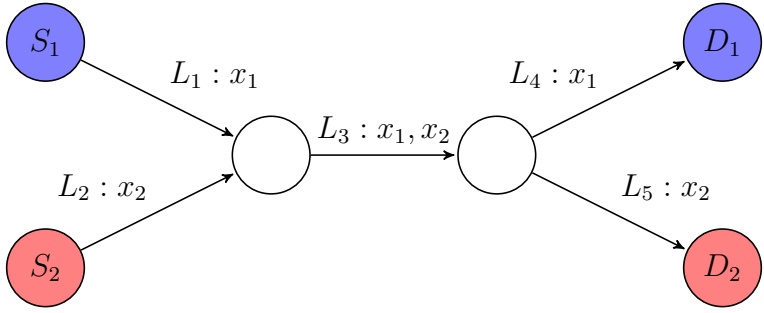
\begin{figure}
\begin{center}\vspace{-1cm}
\begin{tikzpicture}[-,>=stealth',shorten >=1pt,auto,node distance=2.0cm, semithick, scale = 0.75]
    \tikzstyle{every state}=[draw=black,text=black]
    \tikzstyle{LabelStyle}=[fill=white,sloped]
    \tikzstyle{EdgeStyle}=[bend left]
    \node[state][fill=blue!50]      at(0,0) (1) {$S_1$};
    \node[state][fill=red!50] at(0,-4) (2) {$S_2$};
    \node[state][fill=white] at(4,-2)(3){} ;
    \node[state][fill=white] at (8,-2) (4) {};
    \node[state][fill=blue!50] at (12,0) (5) {$D_1$};
    \node[state][fill = red!50] at (12,-4) (6) {$D_2$};
    \path
    (1) edge [->] node {$L_1: x_1$} (3)
    (2) edge [->] node {$L_2: x_2$} (3)
    (3) edge [->] node {$L_3: x_1,x_2$} (4)
    (4) edge[->] node  {$L_4: x_1$} (5)
    (4) edge [->] node {$L_5: x_2$} (6)
    ;
\end{tikzpicture}
\caption{A sample network. Each source-destination pair is displayed with the same color. We use $x_i$ to denote the flow corresponding to the $i^{th}$ source-destination pair and $L_i$ to denote the $i^{th}$ link.}\label{fig:withFlow}
\end{center}
\end{figure}

\begin{figure}\begin{center}\vspace{-1cm}
\begin{tikzpicture}[-,>=stealth',shorten >=1pt,auto,node distance=2.0cm, semithick,scale = 0.65]
    \tikzstyle{every state}=[draw=black,text=black]
    \tikzstyle{LabelStyle}=[fill=white,sloped]
    \node[state][fill=white] at(-8,0) (1) {$L_1$};
    \node[state][fill=white] at(-8,-8) (2) {$L_2$};
    \node[state][fill=white] at(0,-4)(3){$L_3$} ;
    \node[state][fill=white] at (8,0) (4) {$L_4$};
    \node[state][fill=white] at (8,-8) (5) {$L_5$};
    \path
    (1) edge [->,bend left=10] node {$x_1$} (3)
    (2) edge [->,bend left=10] node {$x_2$} (3)
    (3) edge [->,bend left=10] node{$x_1$} (4)
    (3) edge[->,bend left=10] node {$x_2$} (5)
    (1) edge [->,bend left=10] node {$x_1$} (4)
    (2) edge[->,bend left=10]node {$x_2$}(5)
    (1) edge [<-,bend right=10] node {$x_1$} (3)
    (2) edge [<-,bend right=10] node {$x_2$} (3)
    (3) edge [<-,bend right=10] node{$x_1$} (4)
    (3) edge[<-,bend right=10] node {$x_2$} (5)
    (1) edge [<-,bend right=10] node {$x_1$} (4)
    (2) edge[<-,bend right=10]node {$x_2$}(5)

    ;
\end{tikzpicture}\end{center}
\caption{Dual graph for the network in Figure \ref{fig:withFlow}, each link in this graph corresponds to the flows shared between the links in the original network. }\label{fig:dual}
\end{figure}
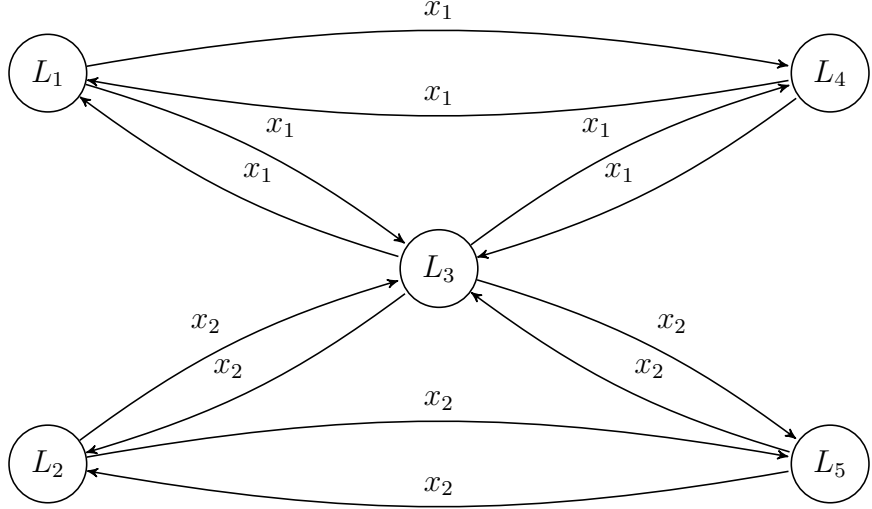

We next rewrite iteration (\ref{eq:dualInnerIteration}), analyze the
information exchange required to implement it and develop a
distributed computation procedure to calculate the dual vector. For
notational convenience, we define the \textit{price of the route}
for source $i$, $\pi_i(t)$, as the sum of the dual variables
associated with links used by source $i$ at the $t^{th}$ dual
iteration, i.e., $\pi_i (t) = \sum _{l\in L(i)} w_l(t)$. Similarly,
we define the \textit{weighted price of the route} for source $i$,
$\Pi_i(t)$, as the price of the route for source $i$ weighted by the
$i$th diagonal element of the inverse Hessian matrix, i.e.,
$\Pi_i(t) = (H_k^{-1})_{ii}\sum _{l\in L(i)} w_l(t)$.

\begin{lemma}\label{lemma:dualCorrect}
For each primal iteration $k$, the dual iteration
(\ref{eq:dualInnerIteration}) can be written as
\begin{align}  \label{eq:wIteAlg}
w_l(t+1) =& \frac{1}{  (H_k)^{-1}_{(S+l)(S+l)}+\sum_{i \in S(l)}
\Pi_i(0)}\Big(\Big(\sum_{i \in S(l)} \Pi_i(0) - \sum_{i\in S(l)}
(H_k)_{ii}^{-1}\Big)w_l(t)-\sum_{i \in S(l)} \Pi_i(t)\\ \nonumber &
+ \sum_{i\in S(l)} (H_k)_{ii}^{-1}w_l(t)
    -\sum_{i\in S(l)} (H_k^{-1})_{ii}\nabla_i f(x^k) - (H_k^{-1})_{(S+l)(S+l)}\nabla_{S+l} f(x^k)\Big),\end{align}
    where $\Pi_i(0)$ is the weighted price of the route for source $i$ when $w(0)=[1,1 \ldots, 1]'$.
\end{lemma}

\begin{proof}
Recall the definition of matrix $A$, i.e., $A_{li}=1$ for $i=1, 2
\ldots S$ if source $i$ uses link $l$, i.e., $i \in S(l)$, and
$A_{li}=0$ otherwise. Therefore, we can write the price of the route
for source $i$ as, $\pi_i(t) = \sum_{l=1}^L A_{li}w(t)_l =
[A']^iw(t)$. Similarly, since the Hessian matrix $H_k$ is diagonal,
the weighted price can be written as \be\label{eq:bigPiDef}{\Pi_i(t)
= (H_k)_{ii}^{-1}[A']^iw(t) = [H_k^{-1}A']^iw(t).} \ee On the other
hand, since $A = [R$ $I(L)]$, where $R$ is the routing matrix, we
have
\begin{align*}
(AH^{-1}_kA'w(t))_l&= \sum_{i=1}^{S} ([A]_i [H_k^{-1}A']^iw(t))_l + (H_k^{-1})_{(S+l)(S+l)}w_l(t)\\
&=  \sum_{i=1}^{S} A_{li} ( [H_k^{-1}A']^iw(t))+ (H_k^{-1})_{(S+l)(S+l)}w_l(t).
\end{align*}
Using the definition of the matrix $A$ one more time, this implies
\begin{align}\label{eq:usePi}(AH^{-1}_kA'w(t))_l&=\sum_{i \in S(l)}[H_k^{-1}A']^iw(t)+(H_k^{-1})_{(S+l)(S+l)}w_l(t)\\ \nonumber
&= \sum_{i \in S(l)} \Pi_i(t)
+(H_k^{-1})_{(S+l)(S+l)}w_l(t),\end{align} where the last equality
follows from Eq.\ (\ref{eq:bigPiDef}).

Using Eq.\ (\ref{eq:defB}), the above relation implies that
$((B_k+D_k)w(t))_l=\sum_{i \in S(l)} \Pi_i(t)
+(H_k^{-1})_{(S+l)(S+l)}w_l(t)$. We next rewrite $(\bar B_k)_{ll}$.
Using the fact that $w(0)=[1,1 \ldots, 1]'$, we have
\alignShort{(AH^{-1}_kA'w(0))_l=((B_k+D_k)w(0))_l = \sum_{j=1}^L
(B_k)_{lj}+(D_k)_{ll}.} Using the definition of $\bar B_k$ [cf.\
Eq.\ (\ref{eq:defBbar})], this implies \alignShort{ (\bar B_k)_{ll}
&= \sum_{j=1}^{L} (B_k)_{lj} = (AH^{-1}_kA'w(0))_l -(D_k)_{ll} \\&=
\sum_{i \in S(l)} \Pi_i(0) +(H_k^{-1})_{(S+l)(S+l)}- (D_k)_{ll}.}
This calculation can further be simplified using
\be\label{eq:defDDis} (D_k)_{ll}=(AH_k^{-1}A')_{ll} = \sum_{i\in
S(l)} (H_k)_{ii}^{-1} + (H_k)^{-1}_{(S+l)(S+l)}, \ee [cf.\ Eq.\
(\ref{eq:defD})], yielding \be\label{eq:bbarDist}(\bar
B_k)_{ll}=\sum_{i \in S(l)} \Pi_i(0) - \sum_{i\in S(l)}
(H_k)_{ii}^{-1}.\ee

Following the same argument, the value $(B_kw(t))_l$ for all $t$ can
be written as
\begin{align*}(B_kw(t))_l &= (AH^{-1}_kA'w(t))_l-(D_kw(t))_l \\ \nonumber
&= \sum_{i=1}^{S} \Pi_i(t) +(H_k^{-1})_{(S+l)(S+l)}w_l(t)- (D_k)_{ll}w_l(t) \\ \nonumber
&=  \sum_{i=1}^{S} \Pi_i(t) -\sum_{i\in S(l)} (H_k)_{ii}^{-1}w_l(t),\end{align*}
where the first equality follows from Eq.\ (\ref{eq:defBbar}), the second equality follows from Eq.\ (\ref{eq:usePi}), and the last equality follows from Eq.\ (\ref{eq:defDDis}).

Finally, we can write $(AH_k^{-1}\nabla f(x^k))_l$ as
\als{(AH_k^{-1}\nabla f(x^k))_l = \sum_{i\in S(l)}
(H_k^{-1})_{ii}\nabla_i f(x^k) + (H_k^{-1})_{(S+l)(S+l)}\nabla_{S+l}
f(x^k).} Substituting the preceding into
(\ref{eq:dualInnerIteration}), we obtain the desired iteration
(\ref{eq:wIteAlg}).
\end{proof}

We next analyze the information exchange required to implement
iteration (\ref{eq:wIteAlg}) among sources and links in the network.
We first observe the local information available to sources and
links. Each source $i$ knows the $i$th diagonal entry of the Hessian
$(H_k)_{ii}$ and the $i$th component of the gradient $\nabla_i
f(x^k)$. Similarly, each link $l$ knows the ($S+l$)th diagonal entry
of the Hessian $(H_k)_{S+l,S+l}$ and the ($S+l$)th component of the
gradient $\nabla_{S+l} f(x^k)$. In addition to the locally available
information, each link $l$, when executing iteration
(\ref{eq:wIteAlg}), needs to compute the terms:
\[\sum_{i\in S(l)} (H_k)^{-1}_{ii},
\qquad \sum_{i\in S(l)} (H_k^{-1})_{ii}\nabla_i f(x^k),\qquad
\sum_{i\in S(l)} \Pi_i(0),\qquad \sum_{i\in S(l)} \Pi_i(t).\] The
first two terms can be computed by link $l$ if each source sends its
local information to the links along its route ``once" in primal
iteration $k$. The third term can be computed by link $l$ again once
for every $k$ if the route price $\pi_i(0)$ (aggregated along the
links of a route when link prices are all equal to 1) are sent by
the destination to source $i$, which then evaluates and sends the
weighted price $\Pi_i(0)$ to the links along its route. The fourth
term can be computed with a similar feedback mechanism, however the
computation of this term needs to be repeated for every dual
iteration $t$.

The preceding information exchange suggests the following
distributed implementation of (\ref{eq:wIteAlg}) (at each primal
iteration $k$) among the sources and the links, where each source or link is viewed as a processor, information available at source $i$ can be passed to the links it traverses, i.e., $l\in L(i)$, and information about the links along a route can be aggregated and sent back to the corresponding source using a feedback mechanism:

\begin{itemize}
\item[1.] {\bf Initialization.}
 \begin{itemize}
\item[1.a] Each source $i$ sends its local information $(H_k)_{ii}$ and $\nabla_i f(x^k)$
to the links along its route, $l \in L(i)$. Each link $l$ computes
$(H_k)^{-1}_{(S+l)(S+l)}$, $\sum_{i\in S(l)}
(H_k)_{ii}^{-1}$,
$(H_k^{-1})_{(S+l)(S+l)}\nabla_{S+l}f(x^k) $ and $\sum_{i\in S(l)} (H_k^{-1})_{ii}\nabla_i
f(x^k)$.
\item[1.b] Each link $l$ starts with price $w_l(0)=1$. The link
prices $w_l(0)$ are aggregated along route $i$ to compute
$\pi(0)=\sum _{l\in L(i)} w_l(0)$ at the destination. This
information is sent back to source $i$.
\item[1.c] Each source computes the weighted price
$\Pi_i(0) = (H_k^{-1})_{ii}\sum _{l\in L(i)} w_l(0)$ and sends it to
the links along its route, $l\in L(i)$.
\item[1.d] Each link $l$ then initializes with arbitrary price $w_l(1)$.
\end{itemize}
\item[2.] {\bf Dual Iteration.}
\begin{itemize}
\item[2.a] The link prices $w_l(t)$ are updated using (\ref{eq:wIteAlg}) and aggregated along route $i$ to compute
$\pi(t)$ at the destination. This information is sent back to source
$i$.
\item[2.b] Each source computes the weighted price
$\Pi_i(t) $ and sends it to the links along its route, $l\in L(i)$.
\end{itemize}
\end{itemize}

The direction of information flow can be seen in Figures
\ref{fig:step1} and \ref{fig:step2}.

\begin{figure}
\begin{minipage}[b]{0.5\linewidth} 
\centering
\includegraphics[width=6cm]{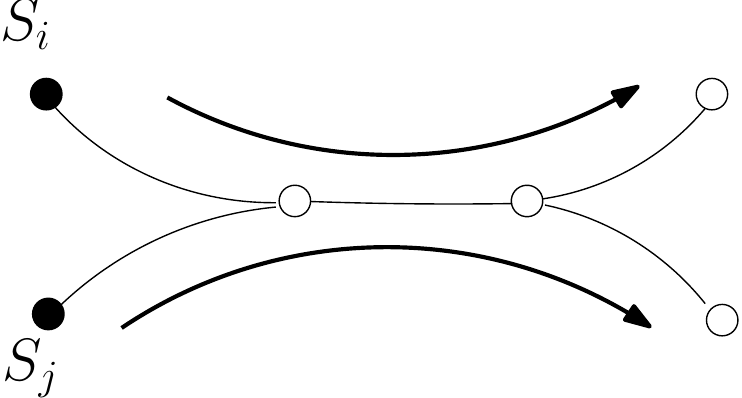}
\caption{Direction of information flow for the steps 1.a, 1.c and 2.b, from sources to the links they use.}\label{fig:step1}
\end{minipage}
\hspace{0.5cm} 
\begin{minipage}[b]{0.5\linewidth}
\centering
\includegraphics[width=6cm]{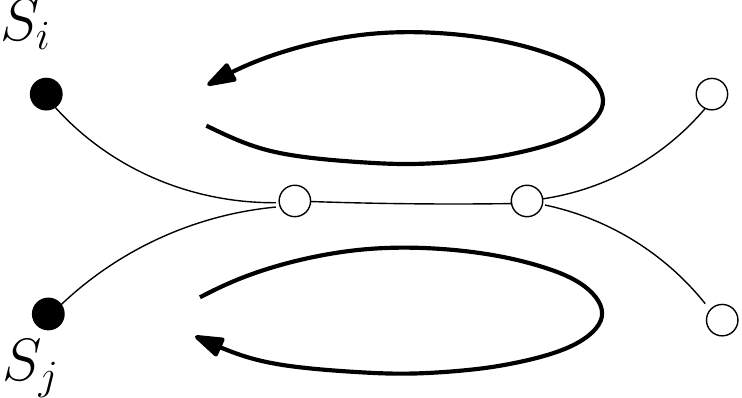}
\caption{Direction of flow for the steps 1.b and 2.a, from links to the sources using them.}\label{fig:step2}
\end{minipage}
\end{figure}

Note that the sources need to send their Hessian and gradient
information once per primal iteration since these values do not
change in the dual iterations. Moreover, this algorithm has
comparable level of information exchange with the subgradient based
algorithms applied to the NUM problem (\ref{ineqFormulation}) (see
\cite{Low}, \cite{Kelly}, \cite{LowLapsley}, \cite{AsuChapter} for
more details). In both types of algorithms, only the sum of prices
of links along a route is fed back to the source, and the links
update prices based on scalar information sent from sources using
that link. The computation here is slightly more involved since it
requires scaling by Hessian matrix entries, however all operations
are scalar-based, hence does not impose degradation on the
performance of the algorithm.

\subsection{Distributed Computation of the Primal Newton Direction}\label{sec:primal}
Once the dual variables are computed, the primal Newton direction
can be obtained according to Eq.\ (\ref{eq:primalDirection}) as
\be\label{eq:deltax} (\Delta  x^k)_i = -(H_k)_{ii}^{-1}(\nabla_i f(
x^k) + (A'w^k)_i) =  -(H_k)_{ii}^{-1}\nabla_i f(  x^k) + \Pi_i,\ee
where $\Pi_i$ is the weighted price of the route for source $i$
computed at termination of the dual iteration. Hence, the primal
Newton direction can be computed using local information by each
source. However, because the dual variable computation involves an
iterative scheme, the exact value for $w^k$ is not available.
Therefore, the direction $\Delta x^k$ computed using Eq.\
(\ref{eq:deltax}) may violate the equality constraints in problems
(\ref{eqFormulation}). To maintain feasibility of the generated
primal vectors, the calculation of the {\it inexact Newton
direction} at a primal vector $x^k$, which we denote by $\Delta
\tilde x^k$, is separated into two stages.

In the first stage, the first $S$ components of $\Delta \tilde x^k$, denoted by $\Delta \tilde s^k$,
is computed via Eq.\ (\ref{eq:deltax}) using the dual variables
obtained via the iterative scheme, i.e., \be\label{eq:ds}\Delta \tilde s_i^k =  -(H_k)_{ii}^{-1}(\nabla_i f(
x^k)+[R']^iw^k).\ee In the second stage, the last $L$
components of $\Delta \tilde x^k$ (corresponding to the slack
variables) are computed to ensure that the condition
$A\Delta \tilde x^k = 0$ is satisfied, i.e. \be\label{eq:dxPrimal}\Delta \tilde x^k = \left(\begin{array}{c}\Delta\tilde s^k\\-R\Delta\tilde s^k\end{array}\right).\ee This calculation
involves each link computing the slack introduced by the first $S$
components of $\Delta \tilde x^k$.

The algorithm presented generates the primal vectors as follows: Let
$x^0$ be an initial strictly positive feasible primal vector (see
Eq.\ (\ref{init}) for one possible choice). For any $k\ge 0$, we
have \be\label{eq:inexaAlgo}x^{k+1} = x^k + d^k \Delta \tilde
x^k,\ee where $d^k$ is a positive stepsize and $\Delta \tilde x^k$
is the inexact Newton direction at primal vector $x^k$ (obtained
through an iterative dual variable computation scheme and a
two-stage primal direction computation that maintains feasibility).
We will refer to this algorithm as the {\it (distributed) inexact
Newton method}.

\subsection{Stepsize Rule}\label{sec:stepsize}

We next describe a stepsize rule that can be computed in a
distributed manner while achieving local superlinear convergence
rate (to an error neighborhood) for the primal iterations. This rule will further guarantee
that the primal vectors $x^k$ generated by the algorithm remain
strictly positive for all $k$, hence ensuring that the Hessian
matrix is well-defined at all iterates (see Eq.\ (\ref{eq:Hdiag})
and Theorem \ref{theorem:tildeFeasible}).

Our stepsize rule will be based on an inexact version of the Newton
decrement. At a given primal vector $x^k$ (with Hessian matrix
$H_k$), we define the \textit{exact Newton direction}, denoted by
$\Delta x^k$, as the exact solution of the system of equations
(\ref{eq:newtonUpdate}). The \textit{exact Newton decrement}
${\lambda}( x^k)$ is defined as
\be\label{eq:lambdaExactAl}{\lambda}(x^k) = \sqrt{(\Delta {
x}^k)'H_k\Delta {x}^k}.\ee Similarly, the \textit{inexact Newton
decrement} ${\tilde\lambda}(x^k)$ is given by
\be\label{eq:lambdaInexact}\tilde{\lambda}(x^k) = \sqrt{(\Delta
{\tilde x}^k)'H_k\Delta {\tilde x}^k},\ee where  $\Delta {\tilde
x}^k$ is the inexact Newton direction at primal vector $x^k$. Note
that both $\lambda (x^k)$ and $\tilde{\lambda}(x^k)$ are nonnegative
and well-defined due to the fact that the matrix $\nabla^2 f(x^k)$
is positive definite.

Our stepsize rule involves the inexact Newton decrement
$\tilde \lambda (x^k)$, we use $\theta^k$ to denote the approximate value of
$\tilde \lambda (x^k)$ obtained through some distributed computation procedure.
One possible such procedure with finite termination yielding $\theta^k = \tilde \lambda (x^k)$ exactly is described in Appendix \ref{app:stepsizeComp}. However, other estimates $\theta^k$ can be used, which can potentially be obtained by exploiting the diagonal structure of the
Hessian matrix, writing the inexact Newton decrement as
\[\tilde \lambda (x^k) = \sqrt{\sum_{i\in \mathcal{L}\bigcup\mathcal{S}}
(\Delta \tilde x^k)_i^2(H_k)_{ii}} = \sqrt{(L+S)\bar y},\] where
$\bar y=\frac{1}{S+L}\sum_{i\in \mathcal{S}\bigcup\mathcal{L}}
(\Delta \tilde x^k)_i^2(H_k)_{ii}$ and using consensus type of algorithms.

Given the scalar $\theta^k$, an approximation to the inexact Newton decrement $\tilde \lambda (x^k)$, at each iteration $k$, we choose the stepsize $d^k$ as follows: Let $V$
be some positive scalar with $0<V<0.267$. We have
\begin{align}\label{eq:stepsize}
d^k = \left\{\begin{array}{ccc}
& \frac{b}{\theta^k+1} \quad &\textrm{if} \quad\theta^k \geq V \textrm{ for all previous $k$}, \\
& 1 \quad & \textrm{otherwise},
\end{array}
\right.
\end{align}
where $b\in (0,1)$.  The upper bound on $V$ will be used in analysis
of the quadratic convergence phase of our algorithm [cf.\ Assumption
\ref{ass:01}]. This bound will also ensure the strict positivity of
the generated primal vectors [cf.\ Theorem
\ref{theorem:tildeFeasible}].

There can be two sources of error in the execution of the algorithm.
The first is in the computation of the inexact Newton direction,
which arises due to iterative computation of the dual vector $w^k$
and the modification we use to maintain feasibility. Second source
of error is in the stepsize rule, which is a function
of $\theta^k$, an approximation to the inexact Newton decrement
$\tilde \lambda(x^k)$. We next
state two assumptions that quantify the bounds on these errors.

\begin{assumption}\label{ass:errorBoundEps}
Let $\{x^k\}$ denote the sequence of primal vectors generated by the
distributed inexact Newton method. Let $\Delta x^k$ and $\Delta
{\tilde x}^k$  denote the exact and inexact Newton directions at
$x^k$, and $\gamma^k$ denote the error in the Newton direction
computation, i.e.,
\begin{equation}\label{eq:defGamma}
    \Delta x^k = \Delta {\tilde x}^k+ \gamma^k.
\end{equation}
For all $k$, $\gamma^k$ satisfies
\begin{equation}\label{ineq:errorEps}
   |(\gamma^k)' \nabla^2 f( x^k) \gamma^k|  \leq    p^2(\Delta {\tilde x}^k)'\nabla^  2 f( x^k) \Delta {\tilde x}^k +\epsilon.
\end{equation}
for some positive scalars $p<1$ and $\epsilon$.
\end{assumption}

This assumption imposes a bound on the weighted norm of the Newton
direction error $\gamma^k$ as a function of the weighted norm of
$\Delta \tilde x^k$ and a constant $\epsilon$. Note that without the
constant $\epsilon$, we would require this error to vanish when
$x^k$ is close to the optimal solution, $i.e.$, when $\Delta \tilde
x^k$ is small, which is impractical for implementation purposes.
Given $p$ and $\epsilon$, one can devise distributed schemes for
determining the number of dual iterations needed so that the
resulting error $\gamma^k$ satisfies this Assumption (see Appendix
\ref{app:errorbounds}).

We bound the error in the inexact Newton decrement calculation as follows.

\begin{assumption}\label{ass:stepsize}
Let $\tau^k$ denote the error in the Newton decrement calculation,
i.e., \be\label{eq:defTau}\tau^k = \tilde \lambda (x^k) -
\theta^k.\ee For all $k$, $\tau^k$ satisfies
\alignShort{|\tau^k|\leq \left(\frac{1}{b}-1\right)(1+V).}
\end{assumption}

This assumption will be used in establishing the strict positivity
of the generated primal vectors $x^k$. Given $b$ and $V$, using
convergence rate results for average consensus schemes (see
\cite{Consensus2},\cite{NedicSubgradientConsensus}), one can provide
a lower bound on the number of average consensus steps needed so
that the error $\tau^k$ satisfies this assumption. When the method presented in Appendix \ref{app:stepsizeComp} is used to compute $\theta^k$, then we have $\tau^k=0$ for all $k$ and the preceding assumption is satisfied clearly. Throughout the
rest of the paper, we assume the conditions in Assumptions
\ref{asmp:utility}-\ref{ass:stepsize} hold.

We next show that the stepsize choice in (\ref{eq:stepsize}) will
guarantee strict positivity of the primal vector $x^k$ generated by
our algorithm. This is important since it ensures that the Hessian
$H^k$ and therefore the (inexact) Newton direction is well-defined
at each iteration. We proceed by first establishing a bound on the
error in the stepsize calculation.

\begin{lemma}\label{lemma:stepSize}
Let $\theta^k$ be an approximation of the inexact Newton decrement
$\tilde{\lambda}(x^k)$ defined in (\ref{eq:lambdaInexact}). For
$\theta^k \geq V$, we have \be \label{eq:stepsizeBound}
(2b-1)/(\tilde{\lambda}( x^k)+1)\leq \frac{b}{\theta^k+1}\leq
1/(\tilde{\lambda}( x^k)+1), \ee where $b\in (0,1)$ is the constant
used in stepsize choice (\ref{eq:stepsize}).
\end{lemma}
\begin{proof}
By Assumption \ref{ass:stepsize} and the fact $\theta^k\geq V$, we
have \be\label{ineq:abs}{|\tilde{\lambda}(
x^k)-\theta^k|\leq\left(\frac{1}{b}-1\right)(1+V)
\leq\left(\frac{1}{b}-1\right)(1+\theta^k).}\ee By multiplying both
sides by the positive scalar $b$, the above relation implies
\alignShort{b\theta^k-b\tilde{\lambda}( x^k)\leq (1-b)(1+\theta^k),}
which  yields\alignShort{(2b-1)\theta^k+(2b-1)\leq b\tilde{\lambda}(
x^k)+b.} By dividing both sides of the above relation by the
positive scalar $(\theta^k+1)(\tilde\lambda (x^k)+1)$, we obtain the
first inequality in Eq.\ (\ref{eq:stepsizeBound}).

Similarly, using Eq.\ (\ref{ineq:abs}) we can  establish
\alignShort{b\tilde \lambda (x^k)-b\theta^k\leq (1-b)(1+\theta^k),}
which can be rewritten as \alignShort{b\tilde \lambda (x^k)+b\leq
\theta^k+1.} After dividing both sides of the preceding relation by
the positive scalar $(\theta^k+1)(\tilde\lambda (x^k)+1)$, we obtain
the second inequality in Eq.\ (\ref{eq:stepsizeBound}).
\end{proof}

With this bound on the stepsize error, we can show that starting
with a strictly positive feasible solution, the primal vectors $x^k$
generated by our algorithm remain positive for all $k$.

\begin{theorem}\label{theorem:tildeFeasible}
Given a strictly positive feasible primal vector $x^0$, let
$\{x^k\}$ be the sequence generated by the inexact distributed
Newton method (\ref{eq:inexaAlgo}). Assume that the stepsize $d^k$
is selected according to Eq.\ (\ref{eq:stepsize}) and the constant
$b$ satisfies $\frac{V+1}{2V+1}<b<1$. Then, the primal vector $x^k$
is strictly positive for all $k$.
\end{theorem}

\begin{proof}
We will prove this claim by induction. The base case of  $x^0>0$
holds by the assumption of the theorem. Since the $U_i$ are strictly
concave [cf.\ Assumption \ref{asmp:utility}], for any $x^k$, we have
$-\frac{\partial^2 U_i}{\partial x_i^2}( x_i^k)\geq 0$. Given the
form of the Hessian matrix [cf.\ Eq.\ (\ref{eq:Hdiag})], this
implies $(H_k)_{ii}\geq \frac{\mu}{ (x_i^k)^2}$ for all $i$, and
therefore \alignShort{\tilde \lambda(x^k) =
\left(\sum_{i=1}^{S+L}(\Delta \tilde
x^k_i)^2(H_k)_{ii}\right)^{\frac{1}{2}} \geq
\left(\sum_{i=1}^{S+L}\mu\left(\frac{\Delta \tilde
x^k_i}{x_i^k}\right)^2\right)^{\frac{1}{2}} \geq
\mbox{max}_i\left|\frac{\sqrt{\mu}\Delta\tilde x^k_i}{x_i^k}\right|,
} where the last inequality follows from the nonnegativity of the
terms $\mu\left(\frac{\Delta\tilde x^k_i}{x_i^k}\right)^2$. By
taking the reciprocal on both sides, the above relation implies
\begin{align} \label{ineq:stepFeasible}
\frac{1}{\tilde \lambda(x^k)} &\leq \frac{1}{\mbox{max}_i\left|\frac{\sqrt{\mu}\Delta\tilde x^k_i}{x_i^k}\right|} =\frac{1}{\sqrt{\mu}}\mbox{min}_i \left|\frac{x_i^k}{\Delta\tilde x^k_i}\right|
\leq \mbox{min}_i \left|\frac{x_i^k}{\Delta\tilde x^k_i}\right|,
\end{align}
where the last inequality follows from the fact that $\mu\geq 1$.

We show the inductive step by considering two cases.
\begin{itemize}
\item{Case i: $\theta^k\geq V$}\\
By Lemma \ref{lemma:stepSize}, the stepsize $d^k$ satisfies
\[d^k\leq {1}/(1+\tilde \lambda(x^k))< {1}/\tilde \lambda(x^k).\]
Using Eq.\ (\ref{ineq:stepFeasible}), this implies $d^k<
\mbox{min}_i \left|\frac{x_i^k}{\Delta\tilde x^k_i}\right|$. Hence
if $x^k>0$, then $x^{k+1} = x^k + d^k\Delta\tilde x^k>0$.
\item{Case ii: $\theta^k<V$}\\
By Assumption \ref{ass:stepsize}, we have $\tilde
\lambda(x^k)<V+\left(\frac{1}{b}-1\right)(1+V)$. Using the fact that
$b>\frac{V+1}{2V+1}$, we obtain \alignShort{\tilde \lambda(x^k)<
V+\left(\frac{1}{b}-1\right)(1+V)<
V+\left(\frac{2V+1}{V+1}-1\right)(1+V)=2V\leq 1,} where the last
inequality follows from the fact that $V<0.267$. Hence we have $d^k
= 1< \frac{1}{\tilde \lambda(x^k)}\leq \mbox{min}_i
|\frac{x_i^k}{\Delta \tilde x^k_i}|$, where the last inequality
follows from Eq.\ (\ref{ineq:stepFeasible}). Once again, if $x^k>0$,
then $x^{k+1} = x^k + d^k\Delta \tilde x^k>0$.
\end{itemize}
In both cases we have $x^{k+1} = x^k + d^k\Delta \tilde x^k>0$,
which completes the induction proof.
\end{proof}

In the rest of the paper, we will assume that the constant $b$ used
in the definition of the stepsize satisfies $\frac{V+1}{2V+1}<b<1$.

\section{Convergence Analysis}\label{sec:analysis}
We next present our convergence analysis for both primal and dual
iterations. We first establish convergence for dual iterations.

\subsection{Convergence in Dual Iterations}\label{sec:ConvDual1}
We characterize the rate of convergence of the dual iteration
(\ref{eq:dualInnerIteration}). We will use the following lemma
\cite{VargaMatrixIte}.

\begin{lemma}\label{lemma:eig}
Let $M$ be an $n\times n$ matrix, and assume that its spectral
radius, denoted by $\rho(M)$, satisfies $\rho(M)<1$. Let
$\{\lambda_i\}_{i=1,\ldots,n}$ denote the set of eigenvalues of $M$,
with $1>|\lambda_1|\geq|\lambda_2|\geq\ldots\geq|\lambda_n|$ and let
$v_i$ denote the set of corresponding unit length right
eigenvectors. Assume the matrix has $n$ linearly independent
eigenvectors.\footnote{An alternative assumption is that the
algebraic multiplicity of each $\lambda_i$ is equal to its
corresponding geometric multiplicity, since eigenvectors associated
with different eigenvalues are independent \cite{Lay}.} Then for the
sequence $w(t)$ generated by the following iteration
\be\label{iteLemma} w(t+1)=Mw(t),\ee we have
\be\label{speedConvLemma} \norm{w(t)-w^*}_2 \leq |\lambda_1|^t
\alpha ,\ee for some positive scalar $\alpha$, where $w^*$ is the
limit of iteration (\ref{iteLemma}) as $t\to \infty$.
\end{lemma}

We use $M$ to denote the $L\times L$ matrix, $M =
(D_k+\bar{B}_k)^{-1}(\bar{B}_k-B_k)$, and $z$ to denote the vector
$z = (D_k+\bar{B}_k)^{-1}(-AH_k^{-1}\nabla f(x^k))$. We can rewrite
iteration (\ref{eq:dualInnerIteration})  as $w(t+1) =Mw(t)+z$, which
implies \als{w(t+q)&=M^qw(t)+\sum_{i=0}^{q-1}M^iz =
M^qw(t)+(I-M^{q})(I-M)^{-1}z.} This alternative representation is
possible since $\rho(M)<1$, which follows from Theorem
\ref{thm:dualSplitting}. After rearranging the terms, we obtain
\als{w(t+q)&= M^q(w(t)-(I-M)^{-1}z)+(I-M)^{-1}z.} Therefore starting
from some arbitrary initial vector $w(0)$, the convergence speed of
the sequence $w(t)$ coincides with the sequence $u(t)$, generated by
$u(t+q)= M^qu(0)$, where $u(0) =w(0)-M(I-M)^{-1}z$.

We next show that the matrix $M$ has $L$ linearly independent
eigenvectors in order to apply the preceding lemma. We first note that since the nonnegative matrix $A$
has full row rank and the Hessian matrix $H$ has positive diagonal
elements, the product matrix $AH_k^{-1}A'$ has positive diagonal
elements and nonnegative entries. This shows that  the matrix $D_k$
[cf.\ Eq.\ (\ref{eq:defD})] has positive diagonal elements and the
matrix $\bar{B}$ [cf.\ Eq.\ (\ref{eq:defBbar})] has nonnegative
entries. Therefore the matrix $(D_k+\bar{B}_k)^{-\frac{1}{2}}$ is
diagonal and nonsingular. Hence, using the relation $\tilde M =
(D_k+\bar{B}_k)^{\frac{1}{2}}M(D_k+\bar{B}_k)^{-\frac{1}{2}}$, we
see that the matrix $M=(D_k+\bar{B}_k)^{-1}(\bar{B}_k-B_k)$ is
similar to the matrix $\tilde M
=(D_k+\bar{B}_k)^{-\frac{1}{2}}(\bar{B}_k-B_k)(D_k+\bar{B}_k)^{-\frac{1}{2}}$.
From the definition of $B_k$ [cf.\ Eq.\ (\ref{eq:defB})] and the
symmetry of the matrix $AH_k^{-1}A'$, we conclude that the matrix
$B$ is symmetric. This shows that the matrix $\tilde M$ is symmetric
and hence diagonalizable, which implies that the matrix $M$ is also
diagonalizable, and therefore it has $L$ linearly independent
eigenvectors.\footnote{If a square matrix $A$ of size $n\times n$ is
symmetric, then $A$ has $n$ linearly independent eigenvectors. If a
square matrix $B$ of size $n\times n$ is similar to a symmetric
matrix, then $B$ has $n$ linearly independent eigenvectors
\cite{Horn}.} We can use Lemma \ref{lemma:eig} to infer that
\als{\norm{w(t)-w^*}_2= \norm{u(t)-u^*}_2 \leq |\lambda_1|^t\alpha,}
where $\lambda_1$ is the eigenvalue of $M$ with largest magnitude,
and $\alpha$ is a constant that depends on the initial vector $u(0)
=w(0)-(I-M)^{-1}z$. Hence $\lambda_1$ determines the speed of convergence of the dual iteration.

We next analyze the relationship between $\lambda_1$ and the dual graph topology.
First note that the matrix $M=
(D_k+\bar{B}_k)^{-1}(\bar{B}_k-B_k)$ is the weighted Laplacian
matrix of the dual graph [cf.\ Section \ref{sec:dual}], and is therefore positive semidefinite \cite{ChungGraphBook}. We then have $\rho(M) = |\lambda_1| = \lambda_1\geq 0$. From graph theory \cite{graphLaplacian}, Theorem \ref{thm:dualSplitting} and the above analysis, we have
\be\label{ineq:eigBound}\frac{4\mbox{mc}(M)}{L}\leq \lambda_1 \leq \min\left\{2\max_{l \in L} \left[(D_k+\bar{B}_k)^{-1}\bar{B}_k\right]_{ll}, 1\right\} ,\ee
where mc$(M)$ is the weighted maximum cut of the dual graph, i.e., \[\mbox{mc}(M)=\max_{S\subset \tilde{\mathcal{N}}}\left\{\sum_{i\in S, j \not\in S}\tilde W_{ij}+\sum_{i\in S, j \not\in S} \tilde W_{ji}\right\},\]  where $\tilde W_{ij}$ is the weight associated with the link from node $i$ to $j$. The above relation suggests that a large maximal cut of the dual graph provides a large lower bound on $\lambda_1$, implying the dual iteration cannot finish with very few iterates. When the maximum weighted out-degree, i.e., $\max_{l \in L} \left[(D_k+\bar{B}_k)^{-1}\bar{B}_k\right]_{ll}$, in the dual graph is small, the above relation provides a small upper bound on $\lambda_1$ and hence suggesting that the dual iteration converges fast.

We finally illustrate the relationship between the dual graph
topology and the underlying network properties by means of two
simple examples that highlight how different network structures can
affect the dual graph and hence the convergence rate of the dual
iteration. In particular, we show that the dual iteration converges
slower for a network with a more congested link. Consider two
networks given in Figures \ref{fig:Ex1} and \ref{fig:Ex2}, whose
corresponding dual graphs are presented in Figures \ref{fig:dualEx1}
and \ref{fig:dualEx2} respectively. Both of these networks have $3$
source-destination pairs and $7$ links. However, in Figure
\ref{fig:Ex1} all three flows use the same link, i.e., $L_4$,
whereas in Figure \ref{fig:Ex2} at most two flows share the same
link. This difference in the network topology results in different
degree distributions in the dual graphs as shown in Figures
\ref{fig:dualEx1} and \ref{fig:dualEx2}. To be more concrete, let
$U_i(s_i)=15\log(s_i)$ for all sources $i$ in both graphs and link
capacity $c_l = 35$ for all links $l$. We apply our distributed
Newton algorithm to both problems, for the primal iteration when all
the source rates are $10$, the largest weighted out-degree in the
dual graphs of the two examples are $0.46$ for Figure
\ref{fig:dualEx1} and $0.095$ for Figure \ref{fig:dualEx2}, which
implies the upper bounds for $\lambda_1$ of the corresponding dual
iterations are $0.92$ and $0.19$ respectively [cf.\ Eq.\
(\ref{ineq:eigBound})]. The weighted maximum cut for Figure
\ref{fig:dualEx1} is obtained by isolating the node corresponding to
$L_4$, with weighted maximum cut value of 0.52. The maximum cut for
Figure \ref{fig:dualEx2} is formed by isolating the set $\{L_4,
L_6\}$, with weighted maximum cut value of $0.17$. Based on
(\ref{ineq:eigBound}) these graph cuts generate lower bounds for
$\lambda_1$ of $0.30$ and $0.096$ respectively. By combining the
upper and lower bounds, we obtain intervals for $\lambda_1$ as
$[0.30, 0.92]$ and $[0.096, 0.19]$ respectively. Recall that a large
spectral radius corresponds to slow convergence in the dual
iteration [cf.\ Eq.\ (\ref{speedConvLemma})], therefore these bounds
guarantee that the dual iteration for the network in Figure
\ref{fig:Ex2}, which is less congested, converges faster than for
the one in Figure \ref{fig:Ex1}. Numerical results suggest the
actual largest eigenvalues are $0.47$ and $0.12$ respectively, which
confirm with the prediction.

\begin{figure}
\begin{center}
\hspace{-2cm}
\includegraphics[trim=0.5cm 5cm 3cm 10cm, clip, width = 14cm]{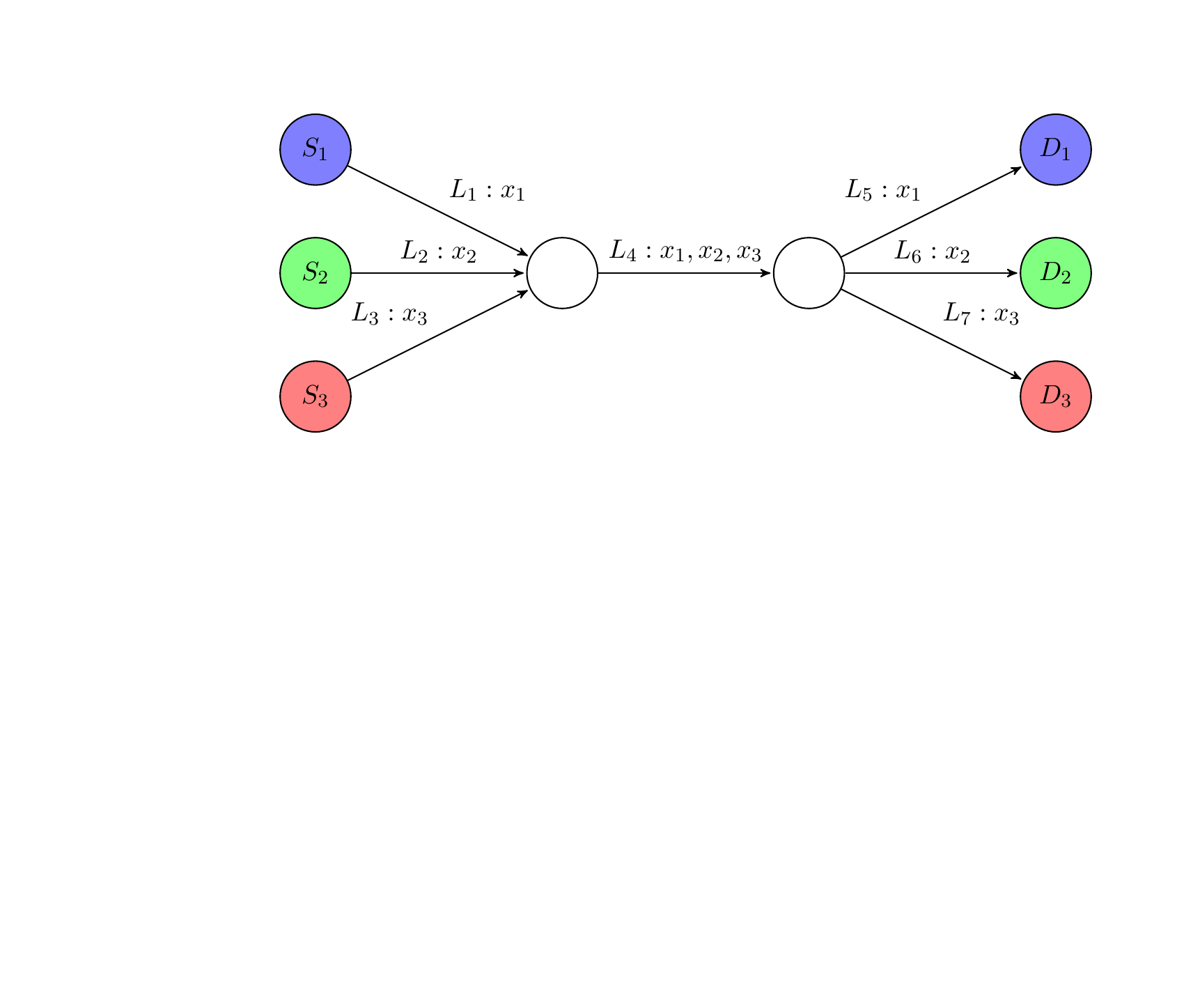}
\vspace{-7cm}
\caption{Each source-destination pair is displayed with the same color. We use $x_i$ to denote the flow corresponding to the $i^{th}$ source-destination pair and $L_i$ to denote the $i^{th}$ link. All 3 flows traverse link $L_4$.}\label{fig:Ex1}
\end{center}
\begin{center}\vspace{0cm}
\includegraphics[trim=2cm 10cm 3cm 8cm, clip, width = 14cm]{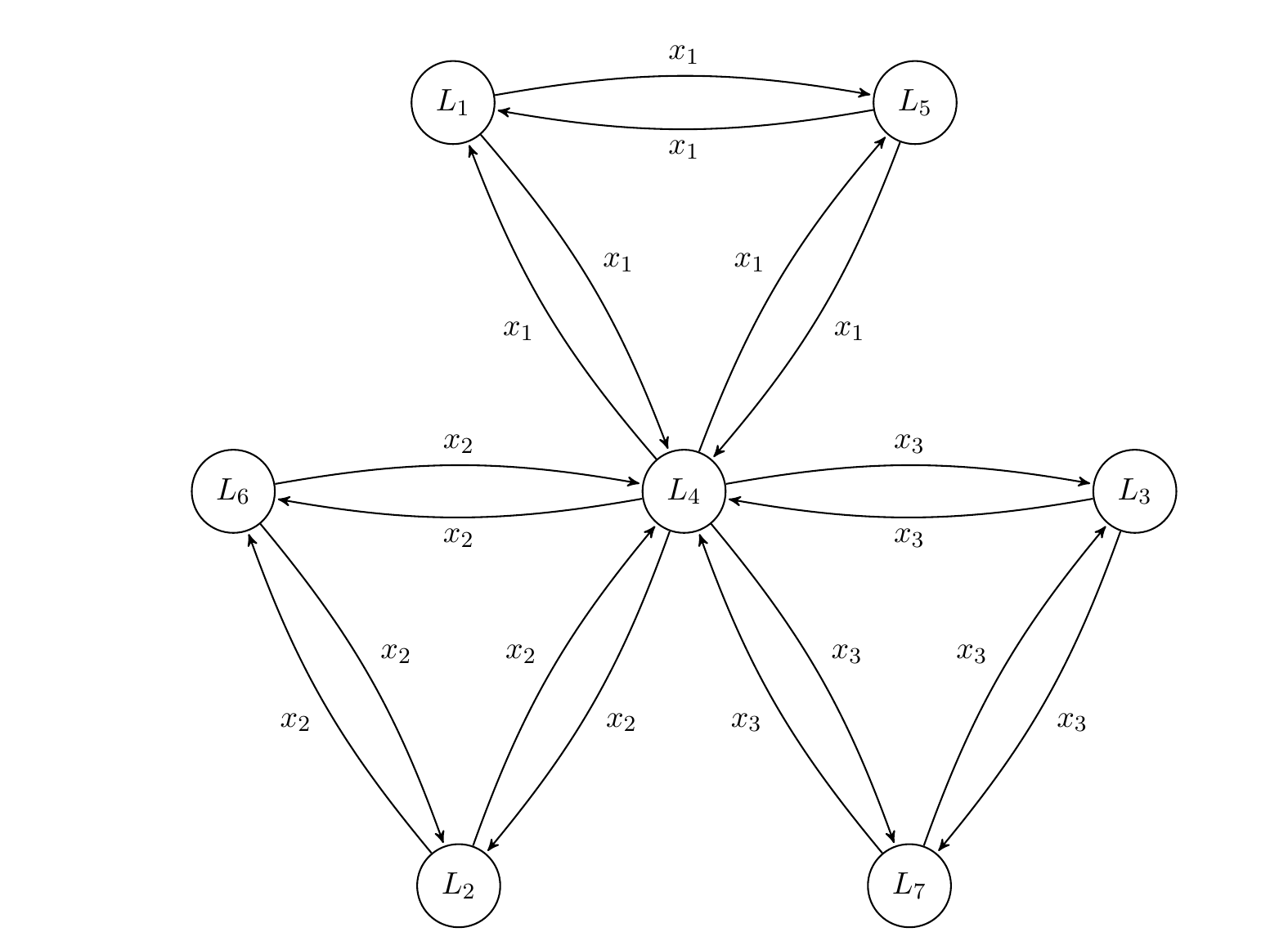}
\caption{Dual graph for the network in Figure \ref{fig:Ex1}, each link in this graph corresponds to the flows shared between the links in the original network. The node corresponding to link $L_4$ has high unweighted out-degree equal to $6$. }\label{fig:dualEx1}
\end{center}
\end{figure}

\begin{figure}
\begin{center}
\includegraphics[trim=2cm 13cm 3cm 10cm, clip, width = 14cm]{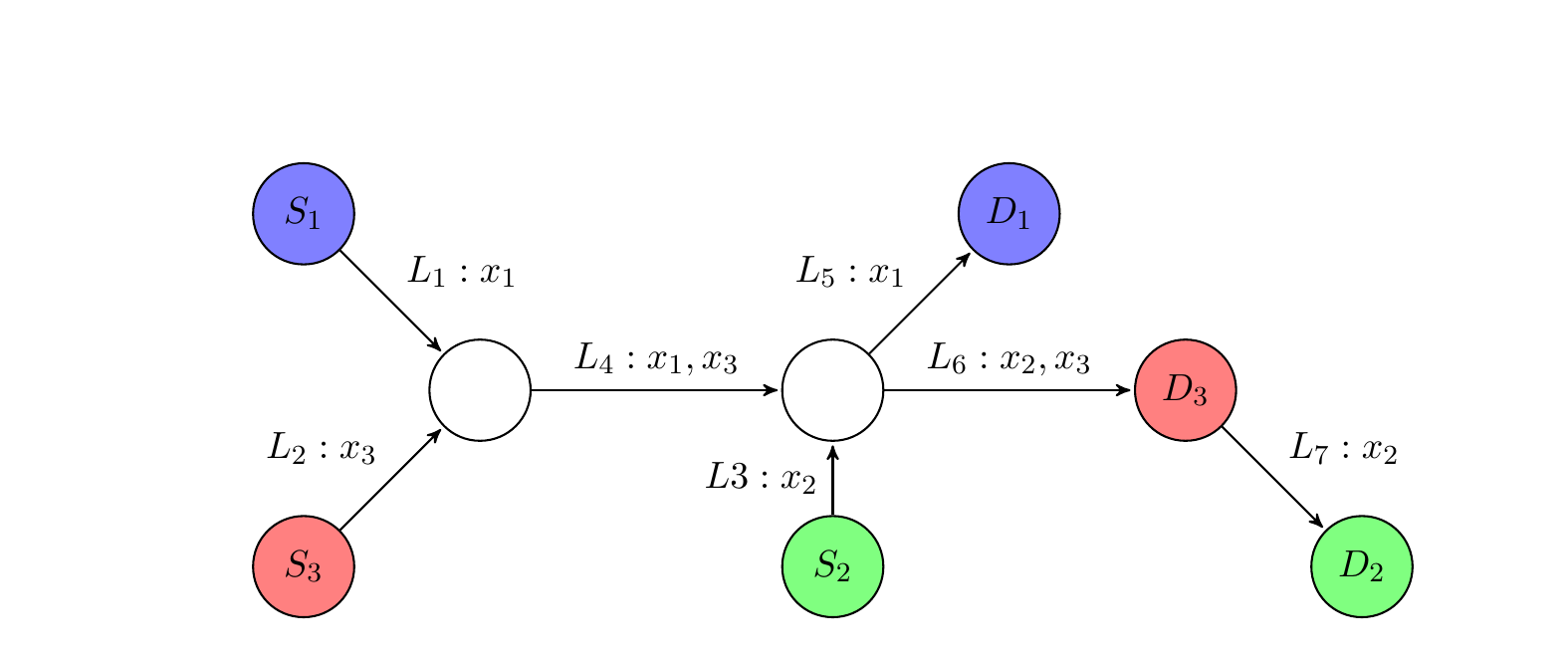}
\caption{Each source-destination pair is displayed with the same color. We use $x_i$ to denote the flow corresponding to the $i^{th}$ source-destination pair and $L_i$ to denote the $i^{th}$ link. Each link has at most $2$ flows traversing it.}\label{fig:Ex2}
\end{center}
\begin{center}
\includegraphics[trim=2cm 11cm 3cm 9cm, clip, width = 14cm]{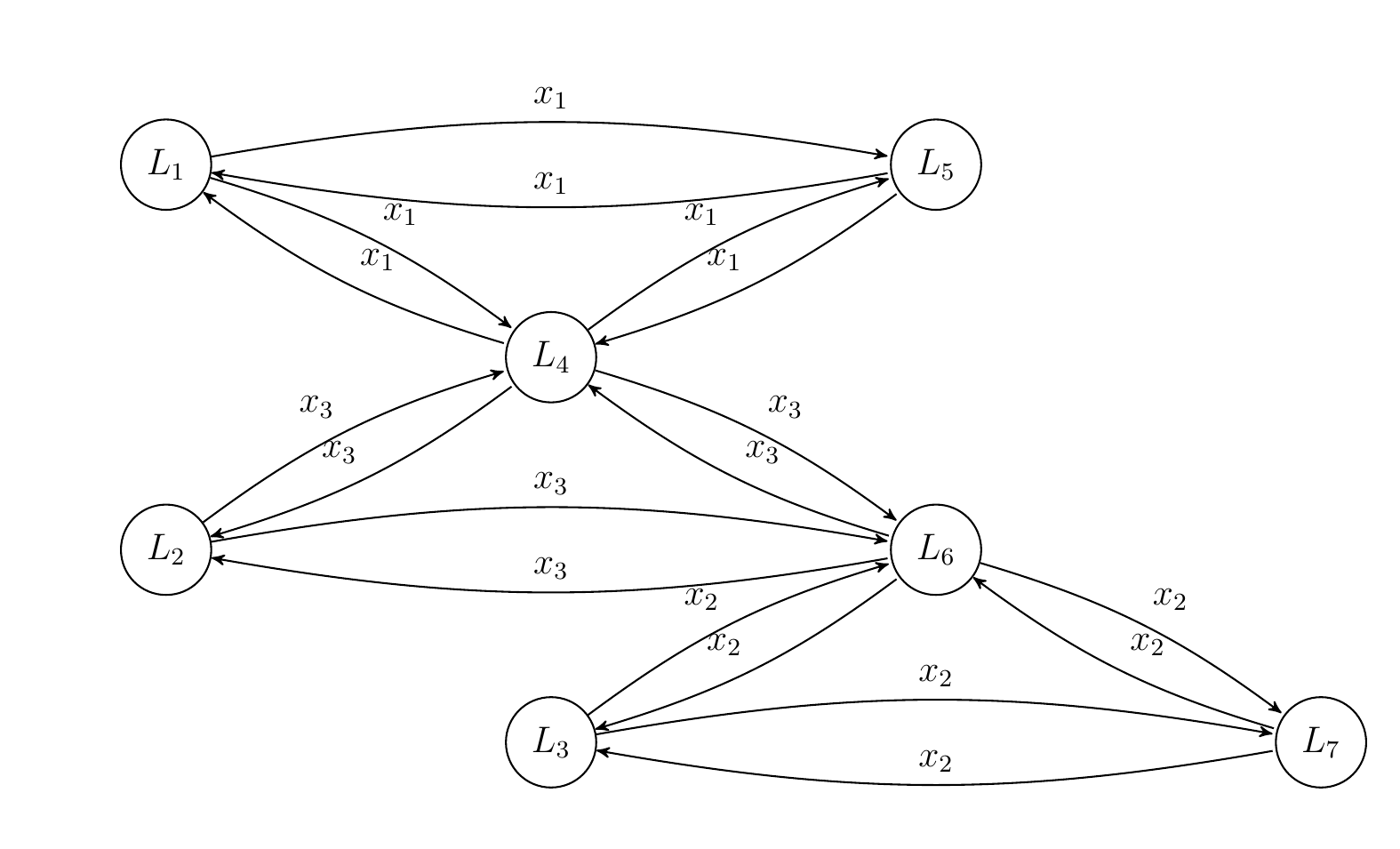}
\caption{Dual graph for the network in Figure \ref{fig:dualEx2}, each link in this graph corresponds to the flows shared between the links in the original network. Both nodes corresponding to links $L_4$ and $L_6$ has relatively high out-degree equal to $4$. }\label{fig:dualEx2}
\end{center}
\end{figure}

\subsection{Convergence in Primal Iterations}
We next present our convergence analysis for the primal sequence
$\{x^k\}$ generated by the inexact Newton method
(\ref{eq:inexaAlgo}). For the $k^{th}$ iteration, we define the
function $\tilde{f}_k:\mathbb{R}\rightarrow \mathbb{R}$ as
\be\label{eq:ftilde} \tilde{f}_k(t) = f(x^k+t\Delta \tilde x^k),\ee
which is self-concordant, because the objective function $f$ is
self-concordant. Note that the value  $\tilde f_k(0)$ and $\tilde
f_k(d^k)$ are the objective function values at $x^k$ and $x^{k+1}$
respectively. Therefore $\tilde f_k(d^k)-\tilde f_k(0)$ measures the
decrease in the objective function value at the $k^{th}$ iteration.
We will refer to the function $\tilde f_k$ as the {\it objective
function along the Newton direction}.

Before proceeding further, we first introduce some properties of
self-concordant functions and the Newton decrement, which will be
used in our convergence analysis.\footnote{We use the same notation
in these lemmas as in
(\ref{eqFormulation})-(\ref{objfunc-eqformulation}) since these
relations will be used in the convergence analysis of the inexact
Newton method applied to problem (\ref{eqFormulation}).}

\subsubsection{Preliminaries}\label{sec:preli}

Using the definition of a self-concordant function, we have the
following result (see \cite{Boyd} for the proof).
\begin{lemma}\label{lemma:selfConcord}
Let $\tilde f:\mathbb{R}\rightarrow \mathbb{R}$ be a self-concordant function. Then for all $t\geq0$ in the domain of the function $\tilde f$ with $t\tilde f''(0)^{\frac{1}{2}}<1$, the following inequality holds:
\begin{equation}\label{eq:selfConIneq}
    \tilde{f}(t) \leq \tilde{f}(0) + t\tilde{f}'(0) - t{\tilde{f}}''(0)^{\frac{1}{2}} - \log(1-t\tilde f''(0)^{\frac{1}{2}}).
\end{equation}
\end{lemma}
We will use the preceding lemma to prove a key relation in analyzing
convergence properties of our algorithm [see Lemma
\ref{lemma:tildeF}]. The next lemma will be used to relate the
weighted norms of a vector $z$, with weights $\nabla^2f(x)$ and
$\nabla^2f(y)$ for some $x$ and $y$. This lemma plays an essential
role in establishing properties for the Newton decrement (see
\cite{Jarre}, \cite{InteriorBook} for more details).

\begin{lemma}\label{lemmaMatrix}
Let $f:\mathbb{R}^n\rightarrow \mathbb{R}$ be a self-concordant
function. Suppose vectors $x$ and $y$ are in the domain of $f$ and $
\tilde\lambda = ((x-y)'\nabla^2f(x)(x-y))^\frac{1}{2} < 1$, then for
any $z\in \mathbb{R}^n$, the following inequality holds:
\begin{equation}\label{eq:lemmaMatrix}
    (1-\tilde\lambda)^2z'\nabla^2 f(x)z\leq z'\nabla^2f(y)z\leq\frac{1}{(1-\tilde\lambda)^2}z'\nabla^2f(x)z.
\end{equation}
\end{lemma}

The next two lemmas establish properties of the Newton decrement
generated by the equality-constrained Newton method. The first lemma
extends results in \cite{Jarre} and \cite{InteriorBook} to allow
inexactness in the Newton direction and reflects the effect of the
error in the current step on the Newton decrement in the next
step.\footnote{We use the same notation in the subsequent lemmas as
in problem formulation (\ref{eqFormulation}) despite the fact that
the results hold for general optimization problems with
self-concordant objective functions and linear equality
constraints.}

\begin{lemma}\label{lemma:nexStep}
Let $f:\mathbb{R}^n\to \mathbb{R}$ be a self-concordant function. Consider solving the equality
constrained optimization problem
\begin{align}\label{equalityExact}
 \mbox{minimize} \quad &f(x)\\\nonumber
 \mbox{subject to} \quad &Ax = c,
\end{align}
using an (exact) Newton method with feasible initialization, where the matrix $A$ is in $\mathbb{R}^{L\times (L+S)}$ and has full column rank, i.e.,   rank$(A) = L$.  Let $\Delta x$ be the exact Newton direction at $x$, i.e., $\Delta x$ solves the following system of
linear equations,
\begin{align}\label{eq:newtonUpdateLemma}
\left( \begin{array}{cc}
\nabla^2 f({x})& A' \\
A & 0
\end{array}
\right)
\left(\begin{array}{c}
\Delta  x\\
w
\end{array}
\right) =
-\left(\begin{array}{c}
\nabla f({x})\\
0
\end{array}
\right).
\end{align}

Let $ \Delta \tilde x$
denote any direction with $\gamma= \Delta x - \Delta \tilde x $, and
$x(t) = x + t\Delta \tilde x$ for $t \in [0,1]$. Let $z$ be the
exact Newton direction at $x+\Delta\tilde x$. If $\tilde\lambda =
\sqrt{\Delta \tilde x'\nabla^2f(x)\Delta \tilde x} <1$, then we have
\alignShort{z\nabla^2 f(x+\Delta \tilde x)'z
    & \leq \frac{\tilde\lambda^2}{1-\tilde\lambda}\sqrt{z' \nabla^2f(x) z} + |\gamma'\nabla ^2f(x)'z|.}
\end{lemma}

\begin{proof}
We first transform problem (\ref{equalityExact}) into an unconstrained one via elimination technique, establish equivalence in the Newton decrements and the Newton primal directions between the two problems following the lines in \cite{Boyd}, then derive the results for the unconstrained problem and lastly we map the result back to the original constrained problem.

Since the matrix $A$  has full column rank, i.e., rank$(A)=L$, in order to eliminate the equality constraints, we
let matrix $K \in \mathbb{R}^{(S+L)\times S}$ be any matrix whose
range is null space of A, with rank$(K)=S$, vector $\hat{x} \in
\mathbb{R}^{S+L}$ be a feasible solution for problem
(\ref{equalityExact}), i.e., $A\hat x = c$.  Then we have the
parametrization of the affine feasible set as
\[\lbrace x | Ax=c \rbrace = \lbrace Ky+\hat{x}|y \in \mathbb{R}^{S}\rbrace.\]
The eliminated equivalent optimization problem becomes
\begin{align}\label{reducedExact}
 \mbox{minimize}_{y\in\mathbb{R}^S}\quad F(y) = f(Ky+\hat{x}).
\end{align}

We next show the Newton primal direction for the constrained problem (\ref{equalityExact}) and unconstrained problem (\ref{reducedExact}) are isomorphic, where a feasible solution $x$ for problem (\ref{equalityExact}) is mapped to $y$ in problem (\ref{reducedExact}) with $Ky+\hat{x} = x$. We start by showing that each $\Delta y$ in the unconstrained problem corresponds uniquely to the Newton direction in the constrained problem.

For the unconstrained problem, the gradient and Hessian are given by
\be\label{eq:dfLemma}
    \nabla F(y) = K'\nabla f(Ky+\hat{x}), \quad \nabla^2 F(y) = K'\nabla ^2 f(Ky+\hat{x})K.
\ee
Note that the objective function $f$ is three times continuously differentiable, which implies its Hessian matrix $\nabla ^2 f(Ky+\hat{x})$ is symmetric, and therefore we have $\nabla^2 F(y)$ is symmetric, i.e., $\nabla^2 F(y)' = \nabla^2 F(y)$.

The Newton direction for problem (\ref{reducedExact}) is given by \be\label{eq:dyLemma}\Delta y = -\left(\nabla ^2 F(y)\right)^{-1} \nabla F(y)= -(K'\nabla^2 f(x)K)^{-1}K'\nabla f(x).\footnote{The matrix $K\nabla^2f(x)K$ is invertible. If for some
$y\in \mathbb{R}^S$, we have $K\nabla^2f(x)K'y=0$, then
$y'K\nabla^2f(x)K'y=\norm{(\nabla^2f(x))^{\frac{1}{2}}K'y}_2=0$, which implies
$\norm{K'x}_2=0$, because the matrix $\nabla^2f(x)$ is strictly positive for all $x$. The rows of
the matrix $K'$ span $\mathbb{R}^S$, therefore we have $y=0$. This
shows that the matrix $K\nabla^2f(x)K'$ is invertible.} \ee

We choose
\begin{equation}\label{eq:wLemma}
    w = -(AA')^{-1}A(\nabla f(x)+\nabla^2f(x)\Delta x),
\end{equation} and show that $(\Delta x, w)$ where \be\label{eq:dxLemma}\Delta x = K\Delta y\ee is the unique solution pair for the linear system (\ref{eq:newtonUpdateLemma}) for the constrained problem (\ref{equalityExact}). To establish the first equation, i.e., $\nabla^2 f(x)\Delta x+A'w=-\nabla f(x)$, we use the property that $\left( \begin{array}{c}K' \\A\end{array}\right)u=\left( \begin{array}{c}K'u \\Au\end{array}\right)=0$ for some $u\in \mathbb{R}^{S+L}$ implies $u=0$.\footnote{If $K'u=0$, then the vector $u$ is orthogonal to the row space of the matrix $K'$, and hence column space of the matrix $K$, i.e., null space of the matrix $A$. If $Au=0$, then $u$ is in the null space of the matrix $A$. Hence the vector $u$ belongs to the set nul$(A)\cap\left(\right.$nul$\left.(A)\right)^\bot$, which implies $u=0$.} We have
\begin{align*}
&\left( \begin{array}{c}
K' \\
A
\end{array}
\right)
\left(\begin{array}{c}
\nabla^2 f(x)\Delta x + A'w + \nabla f(x)
\end{array}
\right)\\
 =&
\left(\begin{array}{c}
K'\nabla^2f(x)K( -(K'\nabla^2 f(x)K)^{-1}K'\nabla f(x)) + K'A'w + K'\nabla f(x)\\
A\nabla^2f(x)\Delta x-A(\nabla f(x)+\nabla^2f(x)\Delta x)+ A\nabla f(x)
\end{array}
\right)\\
 =&
\left(\begin{array}{c}
0\\0\end{array}
\right),
\end{align*}
where the first equality follows from definition of $\Delta x$, $\Delta y$ and $w$ [cf.\ Eqs.\ (\ref{eq:dxLemma}), (\ref{eq:dyLemma}) and (\ref{eq:wLemma})] and the second equality follows the fact that $K'A'w=0$ for any $w$.\footnote{Let $K'A'w=u$, then we have $\norm{u}_2^2 = u'K'A'w = w'AKu$. Since the range of matrix $K$ is the null space of matrix $A$, we have $AKu=0$ for all $u$, hence $\norm{u}_2^2=0$, suggesting $u=0$.} Therefore we conclude that the first equation in (\ref{eq:newtonUpdateLemma}) holds. Since the range of matrix $K$ is the null space of matrix $A$, we have $AKy=0$ for all $y$, therefore the second equation in (\ref{eq:newtonUpdateLemma}) holds, i.e., $A\Delta x = 0$.

For the converse, given a Newton direction $\Delta x$ defined as solution to the system (\ref{eq:newtonUpdateLemma}) for the constrained problem (\ref{equalityExact}), we can uniquely recover a vector $\Delta y$, such that $K\Delta y = \Delta x$. This is because $A\Delta x=0$ from (\ref{eq:newtonUpdateLemma}), and hence $\Delta x$ is in the null space of the matrix $A$, i.e., column space of the matrix $K$. The matrix $K$ has full rank, thus there exists a unique $\Delta y$. Therefore the (primal) Newton directions for problems (\ref{reducedExact}) and (\ref{equalityExact}) are isomorphic under the mapping $K$. In what follows, we perform our analysis for the unconstrained problem (\ref{reducedExact}) and then use isomorphic transformations to show the result hold for the equality constrained problem (\ref{equalityExact}).

Consider the unconstrained problem (\ref{equalityExact}), let
$\Delta y$ denote the exact Newton direction at $y$ [cf.\ Eq.\
(\ref{eq:dfLemma})], vector $\Delta\tilde y$ denote any direction
in $\mathbb{R}^S$, $y(t) = y+t\Delta \tilde y$ and $\tilde\lambda =
\sqrt{\Delta \tilde y'\nabla^2 F(y)\Delta \tilde y}$. Note that with
the isomorphism established earlier, we have $\tilde\lambda =
\sqrt{\Delta \tilde y'\nabla^2 F(y)\Delta \tilde y}= \sqrt{\Delta
\tilde y'K'\nabla^2 f(Ky+\hat x)K\Delta \tilde y} =\sqrt{\Delta
\tilde x'\nabla^2 f(x)\Delta \tilde x}$, where $x=Ky+\hat x$ and
$\Delta \tilde x = K\Delta \tilde y$. From the assumption in the
theorem, we have $\tilde\lambda<1$. For any $t<1$,
$(y-y(t))'\nabla^2 F(y) (y-y(t)) = t^2\tilde\lambda^2<1$ and by
Lemma \ref{lemmaMatrix} for any $z_y$ in $\mathbb{R}^S$, we have
\begin{equation}
    (1-t\tilde\lambda)^2z_y'\nabla^2F(y)z_y\leq z_y'\nabla^2F(y(t))z_y \leq \frac{1}{(1-t\tilde\lambda)^2} z_y' \nabla^2 F(y) z_y \nonumber
\end{equation}
which implies
\begin{equation}\label{ineq:absineq1}
    z_y'(\nabla^2F(y(t))-\nabla^2F(y))z_y \leq \left(\frac{1}{(1-t\tilde\lambda)^2}-1\right)z_y'\nabla^2F(y)z_y,
\end{equation}
and
\begin{equation*}
    z_y'(\nabla^2F(y)-\nabla^2F(y(t)))z_y \leq \left(1-(1-t\tilde\lambda)^2\right)z_y'\nabla^2F(y)z_y.
\end{equation*}
Using the fact that $1-(1-t\tilde\lambda)^2 \leq \frac{1}{(1-t\tilde\lambda)^2} - 1$, the preceding relation can be rewritten as
\begin{equation}\label{ineq:absineq3}
    z_y'(\nabla^2F(y)-\nabla^2F(y(t)))z_y \leq \left(\frac{1}{(1-t\tilde\lambda)^2} - 1\right)z_y'\nabla^2F(y)z_y.
\end{equation}
Combining relations (\ref{ineq:absineq1}) and (\ref{ineq:absineq3}) yields
\begin{equation}\label{ineq:absineqMaster}
    \left|z_y'(\nabla^2F(y)-\nabla^2F(y(t)))z_y\right| \leq \left(\frac{1}{(1-t\tilde\lambda)^2} - 1\right)z_y'\nabla^2F(y)z_y.
\end{equation}

Since the function $F$ is convex, the Hessian matrix $\nabla^2 F(y)$ is positive semidefinite. We can therefore apply the generalized Cauchy-Schwarz inequality and obtain
\begin{align}\label{eq:absDerivative}
& \left|(\Delta \tilde y)'(\nabla^2F(y(t))-\nabla^2F(y))z_y\right|\\ \nonumber
&\leq \sqrt{(\Delta \tilde y)'(\nabla^2F(y(t))-\nabla^2F(y))\Delta \tilde y'} \sqrt{z_y'(\nabla^2F(y(t))-\nabla^2F(y))z_y}\\ \nonumber
&\leq \left(\frac{1}{(1-t\tilde\lambda)^2}-1\right)\sqrt{(\Delta \tilde y)' \nabla^2F(y) \Delta \tilde y}\sqrt{z_y' \nabla^2F(y) z_y} \\
&= \left(\frac{1}{(1-t\tilde\lambda)^2}-1\right)\tilde\lambda\sqrt{z_y' \nabla^2F(y) z_y},\nonumber
\end{align}
where the second inequality follows from relation (\ref{ineq:absineqMaster}), and the equality follows from definition of $\tilde\lambda$.

Define the function $\kappa:\mathbb{R}\to\mathbb{R}$, as $\kappa(t) = \nabla F(y(t))'z_y + (1-t)(\Delta \tilde  y)'\nabla ^2F(y)' z_y$, then
\begin{align*}
    \left|\frac{d}{dt}\kappa(t)\right| = \left|(\Delta\tilde y)'\nabla^2 F(y(t))'z_y-(\Delta \tilde y)'\nabla^2F(y)z_y \right|= \left|(\Delta \tilde y)'(\nabla^2F(y(t))-\nabla^2F(y))z_y\right|,
\end{align*}
which is the left hand side of (\ref{eq:absDerivative}).

Define $\gamma_y = \Delta y - \Delta \tilde y$, which by the isomorphism, implies $\gamma = \Delta x-\Delta \tilde x = K\gamma_y$. By rewriting $\Delta\tilde y = \Delta y - \gamma_y$ and observing the exact Newton direction $\Delta y$ satisfies $\Delta y = -\nabla^2 F(y)^{-1}\nabla F(y)$ [cf.\ Eq.\ (\ref{eq:dfLemma})] and hence by symmetry of the matrix $\nabla^2F(y)$, we have $\Delta y'\nabla^2F(y) = \Delta y'\nabla^2F(y)' = -\nabla F(y)'$, we obtain
\alignShort{\kappa(0)=\nabla F(y)'z_y + (\Delta \tilde  y)'\nabla ^2F(y)' z_y = \nabla F(y)'z_y - \nabla F(y)' z_y-\gamma_y'\nabla ^2F(y)z_y = -\gamma_y'\nabla ^2F(y)z_y.}

Hence by integration, we obtain the bound
\begin{align*}
    \left|\kappa(t) \right| &\leq \tilde\lambda\sqrt{z_y' \nabla^2F(y) z_y}{\mathlarger \int}_{0}^{t} \left(\frac{1}{(1-s\tilde\lambda)^2}-1\right) ds +|\gamma_y'\nabla ^2F(y)z_y|\\
    & = \frac{\tilde\lambda^2t^2}{1-\tilde\lambda t}\sqrt{z_y' \nabla^2F(y) z_y}+|\gamma_y'\nabla ^2F(y)z_y| .
\end{align*}
For $t=1$, $y(t) = y+\Delta\tilde y$, above equation implies
    \begin{eqnarray*}
    \left|\kappa(1) \right| &= \left|\nabla F(y+\Delta \tilde y)'z_y\right|
    & \leq\frac{\tilde\lambda^2}{1-\tilde\lambda}  \sqrt{z_y' \nabla^2F(y) z_y}+| \gamma_y'\nabla ^2F(y)z_y|.
\end{eqnarray*}
We now specify $z_y$ to be the exact Newton direction at $y+\Delta\tilde y$, then $z_y$ satisfies $z_y'\nabla^2 F(y+\Delta \tilde y)z_y = \left|\nabla F(y+\Delta \tilde y)'z_y\right|$, by using the definition of Newton direction at $y+\Delta\tilde y$ [cf.\ Eq.\ (\ref{eq:dyLemma})], which proves
\alignShort{z_y\nabla^2 F(y+\Delta \tilde y)z_y
    & \leq \frac{\tilde\lambda^2}{1-\tilde\lambda}\sqrt{z_y' \nabla^2F(y) z_y} + |\gamma_y'\nabla ^2F(y)'z_y|.}

We now use the isomorphism once more to transform the above relation to the equality constrained problem domain. We have $z = Kz_y$, the exact Newton direction at $x+\Delta \tilde x = \hat x + Ky + K\Delta\tilde y$. The left hand side becomes
\alignShort{z_y'\nabla^2 F(y+\Delta \tilde y)z_y = z_y'K'\nabla^2 f(x+\Delta\tilde x)Kz_y= z'\nabla^2 f(x+\Delta\tilde x)z.}
Similarly, we have the right hand sand satisfies
\alignShort{\frac{\tilde\lambda^2}{1-\tilde\lambda}\sqrt{z_y' \nabla^2F(y) z_y} + |\gamma_y'\nabla ^2F(y)'z_y| &= \frac{\tilde\lambda^2}{1-\tilde\lambda}\sqrt{z_y' K'\nabla^2f(x)K z_y} + |\gamma_y'K'\nabla ^2f(x)Kz_y| \\
&= \frac{\tilde\lambda^2}{1-\tilde\lambda}\sqrt{z' \nabla^2f(x) z} + |\gamma'\nabla ^2f(x)'z|.}
By combining the above two relations, we have established the desired relation.
\end{proof}

One possible matrix $K$ in the above proof for problem (\ref{eqFormulation}) is given by $K = \left(\begin{array}{c}I(S)\\-R\end{array}\right)$, whose corresponding unconstrained domain consists of the source rate variables. In the unconstrained domain, the source rates are updated and then the matrix $K$ adjusts the slack variables accordingly to maintain the feasibility, which coincides with our inexact distributed algorithm in the primal domain. The above lemma will be used to guarantee quadratic rate of
convergence for the distributed inexact Newton method
(\ref{eq:inexaAlgo})]. The next lemma plays a central role in
relating the suboptimality gap in the objective function value to
the exact Newton decrement (see \cite{Boyd} for more details).

\begin{lemma}
Let  $F:\mathbb{R}^n\to \mathbb{R}$ be a self-concordant function.
Consider solving the unconstrained optimization problem
\be\hbox{minimize}_{x\in \mathbb{R}^n}\
F(x),\label{uncons-newton-two}\ee using an (unconstrained) Newton
method. Let $\Delta x$ be the exact Newton direction at $x$, i.e.,
$\Delta x= -\nabla^2 F(x)^{-1}\nabla F(x)$.
 Let $\lambda(x)$ be the exact Newton decrement, i.e.,
$\lambda (x) = \sqrt{(\Delta { x})'\nabla^2 F(x)\Delta {x}}$. Let
$F^*$ denote the optimal value of problem (\ref{uncons-newton-two}).
If $\lambda(x)\leq 0.68$, then we have
\begin{equation}\label{eq:selfConOptimal}
    F^*\geq F(x) - \lambda(x)^2.
\end{equation}
\end{lemma}

Using the same elimination technique and isomorphism established for
Lemma \ref{lemma:nexStep}, the next result follows immediately.

\begin{lemma}\label{lemma:quadConverg}
Let  $f:\mathbb{R}^n\to \mathbb{R}$ be a self-concordant function. Consider solving the equality
constrained optimization problem
\begin{align}\label{equalityExact-two}
 \mbox{minimize} \quad &f(x)\\\nonumber
 \mbox{subject to} \quad &Ax = c,
\end{align}using a constrained Newton
method with feasible initialization. Let $\Delta x$ be the exact (primal) Newton direction at $x$, i.e., $\Delta x$ solves the system
\begin{align*}
\left( \begin{array}{cc}
\nabla^2 f({x})& A' \\
A & 0
\end{array}
\right)
\left(\begin{array}{c}
\Delta  x\\
w
\end{array}
\right) =
-\left(\begin{array}{c}
\nabla f({x})\\
0
\end{array}
\right).
\end{align*}
Let $\lambda(x)$ be the exact Newton decrement, i.e.,
$\lambda (x) = \sqrt{(\Delta { x})'\nabla^2 f(x)\Delta {x}}$. Let
$f^*$ denote the optimal value of problem (\ref{equalityExact-two}).
If $\lambda(x)\leq 0.68$, then we have
\begin{equation}\label{eq:selfConOptimal}
    f^*\geq f(x) - \lambda(x)^2.
\end{equation}
\end{lemma}

Note that the relation on the suboptimality gap in the preceding
lemma holds when the exact Newton decrement is sufficiently small
(provided by the numerical bound 0.68, see \cite{Boyd}). We will use
these lemmas in the subsequent sections for the convergence rate
analysis of the distributed inexact Newton method applied to problem
(\ref{eqFormulation}). Our analysis comprises of two parts: The
first part is the {\it damped convergent phase}, in which we provide
a lower bound on the improvement in the objective function value at
each step by a constant. The second part is the {\it quadratically
convergent phase}, in which the suboptimality in the objective
function value diminishes quadratically to an error level.

\subsubsection{Basic Relations}
We first introduce some key relations, which provides a bound on the
error in the Newton direction computation. This will be used for
both phases of the convergence analysis.

\begin{lemma}\label{lemma:gammaLambda}
Let $\{x^k\}$ be the primal sequence generated by the inexact Newton
method (\ref{eq:inexaAlgo}). Let $\tilde \lambda(x^k)$ be the
inexact Newton decrement at $x^k$ [cf.\ Eq.\ (\ref{eq:lambdaInexact})]. For
all $k$, we have
\begin{equation*}
|(\gamma^k)'\nabla^2 f(x^k) \Delta \tilde{x}^k|  \leq    p\tilde{\lambda}(
x^k)^2 + \tilde{\lambda}(x^k)\sqrt{\epsilon},
\end{equation*}
where $\gamma^k$, $p$, and $\epsilon$ are nonnegative scalars defined
in Assumption \ref{ass:errorBoundEps}.
\end{lemma}
\begin{proof}
By Assumption \ref{asmp:utility}, the Hessian matrix $\nabla^2 f(x^k)$ is positive definite for all $x^k$. We therefore can apply the generalized Cauchy-Schwarz inequality and obtain
\begin{align}\label{ineq:basics}
 |(\gamma^k)'\nabla^2 f(x^k) \Delta \tilde{x}^k|&\leq \sqrt{((\gamma^k)' \nabla^2 f(x^k) \gamma^k)((\Delta \tilde{x}^k)'\nabla^2 f(x^k) \Delta \tilde{x}^k)}\\ \nonumber
&\leq \sqrt{(p^2 \tilde{\lambda}(x^k)^2+\epsilon) \tilde{\lambda}(x^k)^2}\\ \nonumber
&\leq \sqrt{(p^2 \tilde{\lambda}(x^k)^2+\epsilon+2p \tilde{\lambda}(x^k)\sqrt{\epsilon})  \tilde{\lambda}(x^k)^2},
\end{align}
where the second inequality follows from Assumption \ref{ass:errorBoundEps} and definition of $\tilde\lambda(x^k)$, and the third inequality follows by adding the nonnegative term $2p \sqrt{\epsilon}\tilde{\lambda}(x^k)^3$ to the right hand side.
By the nonnegativity of the inexact Newton decrement $\tilde{\lambda}(x^k)$, it can be seen that relation (\ref{ineq:basics}) implies
\be
|(\gamma^k)' \nabla^2 f(x^k) \Delta \tilde{x}^k| \leq  \tilde{\lambda}(x^k)(p \tilde{\lambda}(x^k)+\sqrt \epsilon ) = p\tilde{\lambda}(x^k)^2+ \tilde{\lambda}(x^k)\sqrt{\epsilon},\nonumber
\ee
which proves the desired relation.
\end{proof}

Using the preceding lemma, the following basic relation can be
established, which will be used to measure the improvement in the
objective function value.

\begin{lemma}\label{lemma:tildeF}
Let $\{x^k\}$ be the primal sequence generated by the inexact Newton
method (\ref{eq:inexaAlgo}). Let $\tilde{f}_k$ be the objective
function along the Newton direction and  $\tilde \lambda (x^k)$ be
the inexact Newton decrement [cf.\ Eqs.\ (\ref{eq:ftilde}) and
(\ref{eq:lambdaInexact})] at $x^k$ respectively. For all $k$ with $0\leq t <
1/{\tilde{\lambda}(x^k)}$, we have
\begin{align}\label{eq:selfConIneqInexactEps}
    \tilde{f}_k(t) \leq  \tilde{f}_k(0) - t(1-p)\tilde{\lambda}(x^k)^2 - (1-\sqrt\epsilon) t\tilde{\lambda}(x^k)-\log(1-t\tilde{\lambda}(x^k)),
\end{align}
where $p$, and $\epsilon$ are the nonnegative scalars defined in Assumption \ref{ass:errorBoundEps}.
\end{lemma}
\begin{proof}
Recall that $\Delta x^k$ is the exact Newton direction, which solves the system (\ref{eq:newtonUpdate}). Therefore for some $w^k$, the following equation is satisfied,
\begin{equation*}
    \nabla ^2 f(x^k) \Delta x^k + A' w^k = -\nabla f(x^k).
\end{equation*}
By left multiplying the above relation by $(\Delta \tilde{x}^k)'$, we obtain
\begin{equation}
        (\Delta \tilde{x}^k)'\nabla ^2 f(x^k) \Delta x^k + (\Delta \tilde{x}^k)' A' w^k = -(\Delta\tilde{x}^k)'\nabla f(x^k).\\ \nonumber
\end{equation}
Using the facts that $\Delta x^k = \Delta \tilde{x}^k+ \gamma^k$ from Assumption \ref{ass:errorBoundEps} and $A\Delta \tilde{x}^k = 0$ by the design of our algorithm, the above relation yields
\begin{equation*}
    (\Delta \tilde{x}^k)'\nabla ^2 f(x) \Delta \tilde{x}^k + (\Delta \tilde{x}^k)'\nabla ^2 f(x^k) \gamma^k = -(\Delta\tilde{x}^k)'\nabla f(x^k).
\end{equation*}
By Lemma \ref{lemma:gammaLambda}, we can bound $(\Delta \tilde{x}^k)'\nabla ^2 f(x^k) \gamma^k$ by,
\begin{equation}
     p\tilde{\lambda}(x^k)^2 + \tilde{\lambda}(x^k)\sqrt{\epsilon}\geq (\Delta \tilde{x}^k)'\nabla^2 f(x^k) \gamma^k \geq -p\tilde{\lambda}(x^k)^2 - \tilde{\lambda}(x^k)\sqrt{\epsilon}.\nonumber
\end{equation}
Using the definition of $\tilde{\lambda}(x^k)$ [cf.\ Eq.\ (\ref{eq:lambdaInexact})] and the preceding two relations, we obtain the following bounds on $(\Delta\tilde{x}^k)'\nabla f(x^k)$:
\begin{equation*}
-(1+p)\tilde{\lambda}(x^k)^2 - \tilde{\lambda}(x^k)\sqrt{\epsilon}\leq   (\Delta\tilde{x}^k)'\nabla f(x^k) \leq -(1-p)\tilde{\lambda}(x^k)^2 + \tilde{\lambda}(x^k)\sqrt{\epsilon}.
\end{equation*}
By differentiating the function $\tilde{f_k}(t)$, and using the preceding relation, this yields,
\begin{align}\label{ineq:dfInexactEps}
    \tilde{f}_k'(0) & = \nabla f(x^k)'\Delta \tilde{x}^k\\ \nonumber
    &\leq -(1-p)\tilde{\lambda}(x^k)^2 + \tilde{\lambda}(x^k)\sqrt{\epsilon}.
\end{align}
Moreover, we have
\begin{align}\label{eq:boundf''}
    \tilde{f}_k''(0) &= (\Delta \tilde{x}^k)' \nabla^2 f(x^k) \Delta \tilde{x}^k\\\nonumber
    & = \tilde{\lambda}(x^k)^2.
\end{align}
The function $\tilde f_k(t)$ is self-concordant for all $k$, therefore by Lemma \ref{lemma:selfConcord}, for  $0\leq t < 1/{\tilde{\lambda}(x^k)}$, the following relations hold:
\begin{align*}
    \tilde{f}_k(t) &\leq \tilde{f}_k(0) + t\tilde{f}_k'(0) - t\tilde{f}_k''(0)^{\frac{1}{2}} - \log(1-tf_k''(0)^{\frac{1}{2}})\\
&\leq \tilde{f}_k(0) - t(1-p)\tilde{\lambda}(x^k)^2 + t\tilde{\lambda}(x^k)\sqrt{\epsilon}-t\tilde{\lambda}(x^k)-\log(1-t\tilde{\lambda}(x^k))\nonumber\\
 &=  \tilde{f}_k(0) - t(1-p)\tilde{\lambda}(x^k)^2 - (1-\sqrt\epsilon) t\tilde{\lambda}(x^k)-\log(1-t\tilde{\lambda}(x^k)),
\end{align*}
where the second inequality follows by Eqs.\
(\ref{ineq:dfInexactEps}) and (\ref{eq:boundf''}). This proves Eq.\
(\ref{eq:selfConIneqInexactEps}).
\end{proof}

The preceding lemma shows that a careful choice of the stepsize $t$
can guarantee a constant lower bound on the improvement in the
objective function value at each iteration. We present the
convergence properties of our algorithm in the following two
sections.

\subsubsection{Damped Convergent Phase}\label{sec:damped}
In this section, we consider the case when $\theta^k \geq  V$ and
stepsize $d^k = \frac{b}{\theta^k+1}$ [cf.\ Eq.
(\ref{eq:stepsize})]. We will provide a constant lower bound on the
improvement in the objective function value in this case. To this
end, we first establish the improvement bound for the exact stepsize
choice of $t = 1/(\tilde{\lambda}(x^k)+1)$.

\begin{theorem}\label{thm:Damped}
Let $\{x^k\}$ be the primal sequence generated by the inexact Newton
method (\ref{eq:inexaAlgo}). Let $\tilde{f}_k$ be the objective
function along the Newton direction and  $\tilde \lambda (x^k)$ be
the inexact Newton decrement at $x^k$ [cf.\ Eqs.\ (\ref{eq:ftilde}) and
(\ref{eq:lambdaInexact})]. Consider the scalars $p$ and $\epsilon$
defined in Assumption \ref{ass:errorBoundEps} and assume that
$0<p<\frac{1}{2}$ and
$0<\epsilon<\left(\frac{\left(0.5-p\right)(2Vb-V+b-1)}{b}\right)^2$,
where $b$ is the constant used in the stepsize rule [cf.\ Eq.\
(\ref{eq:stepsize})]. For $\theta^k \geq  V$ and $t =
1/\left(\tilde{\lambda}(x^k)+1\right)$, there exists a scalar
$\alpha>0$ such that
\begin{align}\label{ineq:dampedProp}
    \tilde{f}_k(t) - \tilde{f}_k(0) &\leq -\alpha\left(1+p\right) \left(\frac{2Vb-V+b-1}{b}\right)^2
    \Big/\left(1+\frac{2Vb-V+b-1}{b}\right).
\end{align}

\end{theorem}
\begin{proof}
For notational simplicity, let $y = \tilde{\lambda}\left(x^k\right)$
in this proof. We will show that for any positive scalar $\alpha$
with $0<\alpha\leq\left(\frac{1}{2}-p-\frac{\sqrt{\epsilon}
b}{(2Vb-V+b-1)}\right)/\left(p+1\right)$, Eq.\
(\ref{ineq:dampedProp}) holds. Note that such $\alpha$ exists since
$\epsilon < \left(\frac{\left(0.5-p\right)(2Vb-V+b-1)}{b}\right)^2$.

By Assumption \ref{ass:stepsize}, we have for $\theta^k\geq V$,
\be\label{ineq:yBound}y\geq\theta^k-\left(\frac{1}{b}-1\right)(1+V)\geq
V-\left(\frac{1}{b}-1\right)(1+V)=\frac{2Vb-V+b-1}{b}.\ee Using
$b>\frac{V+1}{2V+1}$, we have $y\geq
V-\left(\frac{1}{b}-1\right)(1+V)>0$, which implies $2Vb-V+b-1>0$.
Together with $0<\alpha\leq\left(\frac{1}{2}-p-\frac{\sqrt{\epsilon}
b}{2Vb-V+b-1}\right)/\left(p+1\right)$ and $b>\frac{V+1}{2V+1}$,
this shows
\begin{align*}
\sqrt{\epsilon} \leq \frac{2Vb-V+b-1}{b}\left(\frac{1}{2}-p-\alpha\left(1+p\right)\right).
\end{align*}Combining the above, we obtain
\begin{align*}
\sqrt{\epsilon} \leq
y\left(\frac{1}{2}-p-\alpha\left(1+p\right)\right),
\end{align*}
which using algebraic manipulation yields
\begin{equation*}
-\left(1-p\right)y-\left(1-\sqrt{\epsilon}\right)+\left(1+y\right)-\frac{y}{2} \leq -\alpha\left(1+p\right)y.
\end{equation*}
From Eq.\ (\ref{ineq:yBound}), we have $y>0$. We can therefore
multiply by $y$ and divide by $1+y$ both sides of the above
inequality to obtain
\begin{align}\label{ineq:dampedkey2}
-\frac{1-p}{1+y}y^2-\frac{1-\sqrt{\epsilon}}{1+y}y  +y - \frac{y^2}{2\left(1+y\right)} \leq -\alpha \frac{\left(1+p\right)y^2}{1+y}
\end{align}
Using second order Taylor expansion on $\log\left(1+y\right)$, we have for $y \geq 0$
\begin{equation*}
    \log\left(1+y\right) \leq y - \frac{y^2}{2\left(1+y\right)}.
\end{equation*}
Using this relation in Eq.\ (\ref{ineq:dampedkey2}) yields,
\begin{align*}
-\frac{1-p}{1+y}y^2-\frac{1-\sqrt{\epsilon}}{1+y}y+\log\left(1+y\right) \leq -\alpha\frac{\left(1+p\right)y^2}{1+y}.
\end{align*}
Substituting the value of $t = 1/\left(y+1\right)$, the above
relation can be rewritten as
\alignShort{-\left(1-p\right)ty^2-\left(1-\sqrt{\epsilon}\right)ty-\log\left(1-ty\right)
\leq -\alpha \frac{\left(1+p\right)y^2}{1+y}.} Using Eq.\
(\ref{eq:selfConIneqInexactEps}) from Lemma \ref{lemma:tildeF} and
definition of $y$ in the preceding, we obtain
\alignShort{\tilde{f}_k\left(t\right) - \tilde{f}_k\left(0\right)
\leq -\alpha\left(1+p\right) \frac{y^2}{y+1}.} Observe that the
function $h\left(y\right) = \frac{y^2}{y+1}$ is monotonically
increasing in $y$, and for $\theta^k\geq  V$ by relation
(\ref{ineq:yBound}) we have $y\geq \frac{2Vb-V+b-1}{b}$. Therefore
\alignShort{-\alpha\left(1+p\right)
\frac{y^2}{y+1}\leq-\alpha\left(1+p\right)
\left(\frac{2Vb-V+b-1}{b}\right)^2/\left(1+\frac{2Vb-V+b-1}{b}\right).}
Combining the preceding two relations completes the proof.
\end{proof}

Note that our algorithm uses the stepsize $d^k =
\frac{d}{\theta^k+1}$ in the damped convergent phase, which is an
approximation to the stepsize $t = 1/(\tilde{\lambda}(x^k)+1)$ used
in the previous theorem. The error between the two is bounded by
relation (\ref{eq:stepsizeBound}) as shown in Lemma
\ref{lemma:stepSize}. We next show that with this error in the
stepsize computation, the improvement in the objective function
value in the inexact algorithm is still lower bounded at each
iteration.

Let $\beta = \frac{d^k}{t}$, where $t = 1/(\tilde{\lambda}(x^k)+1)$.
By the convexity of $f$, we have
\begin{align*}
    f(x^k+\beta t\Delta x^k) = f(\beta(x^k+t\Delta x^k)+(1-\beta)(x^k))
                        \leq \beta f(x^k+t\Delta x^k) + (1-\beta)
                        f(x^k).
\end{align*}
Therefore the objective function value improvement is bounded by
\begin{align*}
    f(x+\beta t\Delta x^k)-f(x^k) &\leq \beta f(x^k+t\Delta x^k) + (1-\beta) f(x^k)- f(x^k) \\
                            & = \beta(f(x^k+t\Delta x^k)-f(x^k))\\
                            & = \beta(\tilde{f}_k(t) - \tilde{f}_k(0)),
\end{align*}
where the last equality follows from the definition of
$\tilde{f}_k(t)$. Using Lemma \ref{lemma:stepSize}, we obtain bounds
on $\beta$ as $2b-1\leq \beta\leq 1$. Hence combining this bound
with Theorem \ref{thm:Damped}, we obtain
\be f(x^{k+1}) - f(x^k)
\leq -(2b-1)\alpha\left(1+p\right)
\frac{\left(\frac{2Vb-V+b-1}{b}\right)^2}{\left(1+\frac{2Vb-V+b-1}{b}\right)}.
\ee
Hence in the damped convergent phase we can guarantee a lower
bound on the object function value improvement at each iteration.
This bound is monotone in $b$, i.e.,  the closer the scalar $b$ is
to $1$, the faster the objective function value improves, however
this also requires the error in the inexact Newton decrement
calculation, i.e., $\tilde\lambda(x^k)-\theta^k$, to diminish to $0$
[cf.\ Assumption \ref{ass:stepsize}].

\subsubsection{Quadratically Convergent Phase}\label{sec:quad}
In this phase, there exists $\bar k$ with $\theta^{\bar{k}}<V$ and
the step size choice is $d^k=1$ for all $k\ge \bar k$.\footnote{Note that once the condition $\theta^{\bar k}<V$ is satisfied,  in all the following iterations, we have stepsize $d^k=1$ and no longer need to compute $\theta^k$. }
We show that the suboptimality in the primal objective function value diminishes
quadratically to a neighborhood of optimal solution. We proceed by
first establishing the following lemma for relating the exact and
the inexact Newton decrements.

\begin{lemma}
Let $\{x^k\}$ be the primal sequence generated by the inexact Newton
method (\ref{eq:inexaAlgo}) and $\lambda(x^k)$, $\tilde
\lambda(x^k)$ be the exact and inexact Newton decrements at $x^k$
[cf.\ Eqs.\ (\ref{eq:lambdaExactAl}) and (\ref{eq:lambdaInexact})].
Let $p$ and $\epsilon$ be the nonnegative scalars defined in
Assumption \ref{ass:errorBoundEps}. We have
\begin{equation}\label{ineq:lemmaLambdaEps}
(1-p)\tilde{\lambda}(x^k)-\sqrt{\epsilon}\leq \lambda(x^k) \leq
(1+p)\tilde{\lambda}(x^k)+\sqrt{\epsilon}.
\end{equation}
\end{lemma}
\begin{proof}
By Assumption \ref{asmp:utility}, for all $k$, $\nabla^2 f(x^k)$ is positive definite. We therefore can apply the generalized Cauchy-Schwarz inequality and obtain
\begin{align}\label{ineq:firstabs}
 |(\Delta x^k) '\nabla^2 f(x^k) \Delta \tilde{x}^k|
 &\leq \sqrt{((\Delta {x^k}) '\nabla^2 f(x^k) \Delta x^k)((\Delta \tilde{x}^k) '\nabla^2 f(x^k) \Delta \tilde{x}^k)}\\ \nonumber
& = \lambda(x^k)\tilde{\lambda}(x^k),\nonumber
\end{align}
where the equality follows from definition of $\lambda(x^k)$ and
$\tilde{\lambda}(x^k)$. Note that by Assumption
\ref{ass:errorBoundEps}, we have $\Delta x^k = \Delta \tilde{x}^k
+\gamma^k$, and hence
\begin{align}\label{ineq:secondabs}
 |(\Delta x^k) '\nabla^2 f(x^k) \Delta \tilde{x}^k|&= |(\Delta \tilde{x}^k +\gamma^k)'\nabla^2 f(x^k) \Delta \tilde{x}^k|\\ \nonumber
 & \geq (\Delta \tilde{x}^k) '\nabla^2 f(x^k) \Delta \tilde{x}^k -|(\gamma^k)'\nabla^2 f(x^k) \Delta \tilde{x}^k|\\ \nonumber
&\geq \tilde{\lambda}(x^k)^2 - p\tilde{\lambda}(x^k)^2 -  \tilde{\lambda}(x^k)\sqrt{\epsilon},\nonumber
\end{align}
where the first inequality follows from a variation of triangle inequality, and  the last inequality follows from Lemma \ref{lemma:tildeF}.
Combining the two inequalities (\ref{ineq:firstabs}) and (\ref{ineq:secondabs}), we obtain
\be
 \lambda(x^k)\tilde{\lambda}(x^k) \geq \tilde{\lambda}(x^k)^2 - p\tilde{\lambda}(x^k)^2 - \sqrt{\epsilon} \tilde{\lambda}(x^k),\nonumber
\ee By canceling the nonnegative term $\tilde{\lambda}(x^k)$ on both
sides, we have \be \lambda(x^k)\geq \tilde{\lambda}(x^k) -
p\tilde{\lambda}(x^k) - \sqrt{\epsilon}. \nonumber\ee This shows the
first half of the relation (\ref{ineq:lemmaLambdaEps}). For the
second half, using the definition of $\lambda(x^k)$, we have
\begin{align*}
\lambda(x^k)^2 &= (\Delta x^k) '\nabla^2 f(x^k) \Delta x^k \\ \nonumber
& = (\Delta \tilde{x}^k + \gamma^k)'\nabla^2 f(x^k) (\Delta \tilde{x}^k + \gamma^k)\\ \nonumber
&  = (\Delta \tilde{x}^k)'\nabla^2 f(x^k) \Delta \tilde{x}^k + (\gamma^k)'\nabla^2 f(x^k)\gamma^k + 2(\Delta \tilde{x}^k)'\nabla^2 f(x^k)\gamma^k,\nonumber
\end{align*}
where the second equality follows from the definition of $\gamma^k$ [cf.\ Eq.\ (\ref{eq:defGamma})].
By using the definition of $\tilde{\lambda}(x^k)$, Assumption \ref{ass:errorBoundEps} and Lemma \ref{lemma:gammaLambda}, the preceding relation implies,
\begin{align*}
\lambda(x^k)^2 & \leq \tilde{\lambda}(x^k)^2 + p^2  \tilde{\lambda}(x^k)^2 +\epsilon + 2p \tilde{\lambda}(x^k)^2+2\sqrt{\epsilon}\tilde\lambda(x^k)\\
& \leq \tilde{\lambda}(x^k)^2 + p^2  \tilde{\lambda}(x^k)^2 + 2p \tilde{\lambda}(x^k)^2+2\sqrt{\epsilon}(1+p)\tilde\lambda(x^k)  +\epsilon\\
& =  ((1+p)  \tilde{\lambda}(x^k)+\sqrt{\epsilon})^2,
\end{align*}
where the second inequality follows by adding a nonnegative term of $2\sqrt{\epsilon}p\tilde\lambda(x^k) $ to the right hand side. By nonnegativity of $p$, $\epsilon$, $\lambda$ and  $\tilde \lambda (x^k)$, we can take the square root of both sides and this completes the proof for relation (\ref{ineq:lemmaLambdaEps}).
\end{proof}

Before proceeding to establish quadratic convergence in terms of the primal iterations to an error
neighborhood of the optimal solution, we need to impose the
following bound on the errors in our algorithm in this phase. Recall
that $\bar k$ is an index such that $\theta^{\bar k}<V$ and $d^k=1$ for all
$k\ge \bar k$.

\begin{assumption}\label{ass:01}
Let $\{x^k\}$ be the primal sequence generated by the inexact Newton
method (\ref{eq:inexaAlgo}). Let $\phi$ be a positive scalar with
$\phi\leq 0.267$. Let $\xi$ and $v$ be nonnegative scalars defined
in terms of $\phi$ as
\[\xi = \frac{\phi
p+\sqrt\epsilon}{1-p-\phi-\sqrt\epsilon}+\frac{2\phi\sqrt\epsilon+\epsilon}
{(1-p-\phi-\sqrt\epsilon)^2},\quad  \quad
v=\frac{1}{(1-p-\phi-\sqrt\epsilon)^2},\] where $p$ and $\epsilon$
are the scalars defined in Assumption \ref{ass:errorBoundEps}. The
following relations hold
\be\label{ineq:lambdaBound}(1+p)(\theta^{\bar k}+\tau^{\bar
k})+\sqrt\epsilon\leq \phi,\ee \be\label{ineq:next68} v(0.68)^2+\xi
\leq 0.68,\ee \be\label{ineq:lambdatilde}
\frac{0.68+\sqrt{\epsilon}}{1-p}\leq 1,\ee
\be\label{ineq:pepBound}p+\sqrt\epsilon\leq
1-(4\phi^2)^{\frac{1}{4}}-\phi,\ee where $\tau^{\bar k}>0$ is a
bound on the error in the Newton decrement calculation at step $\bar
k$ [cf.\ Assumption \ref{ass:stepsize}].
\end{assumption}

The upper bound of $0.267$ on $\phi$ is necessary here to guarantee
relation (\ref{ineq:pepBound}) can be satisfied by some nonnegative
scalars $p$ and $\epsilon$.  Relation (\ref{ineq:lambdaBound}) can
be satisfied by some nonnegative scalars $p$, $\epsilon$ and
$\tau^{\bar k}$, because we have $\theta^{\bar k}< V <0.267$.
Relation (\ref{ineq:lambdaBound}) and (\ref{ineq:next68}) will be
used to guarantee the condition $\lambda (x^k)\leq 0.68$ is
satisfied throughout this phase, so that we can use Lemma
\ref{lemma:quadConverg} to relate the suboptimality bound with the
Newton decrement, and relation (\ref{ineq:lambdatilde}) and
(\ref{ineq:pepBound}) will be used for establishing the quadratic
rate of convergence of the objective function value, as we will show
in the Theorem \ref{thm:quad}. This assumption can be satisfied by first
choosing proper values for the scalars $p$, $\epsilon$ and $\tau$
such that all the relations are satisfied, and then adapt both the
consensus algorithm for $\theta^{\bar k}$ and the dual iterations
for $w^k$ according to the desired precision (see the discussions
following Assumption \ref{ass:errorBoundEps} and \ref{ass:stepsize}
for how these precision levels can be achieved).

To show the quadratic rate of convergence for the primal iterations, we need the following lemma, which relates the exact Newton decrement at the current and the next step.

\begin{lemma}\label{lemma:quadLambda}
Let $\{x^k\}$ be the primal sequence generated by the inexact Newton
method (\ref{eq:inexaAlgo}) and $\lambda(x^k)$, $\tilde
\lambda(x^k)$ be the exact and inexact Newton decrements at $x^k$
[cf.\ Eqs.\ (\ref{eq:lambdaExactAl}) and (\ref{eq:lambdaInexact})].
Let $\theta^k$ be the computed inexact value of $\tilde
\lambda(x^k)$ and let Assumption \ref{ass:01} hold. Then for all $k$
with $\tilde \lambda(x^k)< 1$, we have
\be\label{ineq:inducStep}\lambda(x^{k+1})\leq
v\lambda(x^k)^2+\xi,\ee where $\xi$ and $v$ are the scalars defined
in Assumption \ref{ass:01} and $p$ and $\epsilon$ are defined as in
Assumption \ref{ass:errorBoundEps}.
\end{lemma}
\begin{proof}
Given $\tilde \lambda(x^k)< 1$, we can apply Lemma \ref{lemma:nexStep} by letting $z=\Delta x^{k+1}$, we have
\begin{align*}
    \lambda(x^{k+1})^2 & = (\Delta x^{k+1})' \nabla f^2(x+\Delta \tilde x) \Delta x^{k+1}\\
        &\leq \frac{\tilde\lambda(x^k)^2}{1-\tilde\lambda(x^k)} \sqrt{(\Delta x^{k+1})' \nabla^2 f(x) \Delta x^{k+1}}+\left|(\gamma^k)'\nabla ^2f(x)'\Delta x^{k+1}\right|\\
        &\leq \frac{\tilde\lambda(x^k)^2}{1-\tilde\lambda(x^k)} \sqrt{(\Delta x^{k+1})' \nabla^2 f(x) \Delta x^{k+1}}+\sqrt{(\gamma^k)'\nabla ^2f(x)\gamma^k}\sqrt{(\Delta x^{k+1})'\nabla ^2f(x)\Delta x^{k+1}},
\end{align*}
where the last inequality follows from the generalized
Cauchy-Schwarz inequality. Using Assumption \ref{ass:errorBoundEps},
the above relation implies \alignShort{\lambda(x^{k+1})^2
        &\leq \left(\frac{\tilde\lambda(x^k)^2}{1-\tilde\lambda(x^k)}+\sqrt{p^2 \tilde\lambda(x^k)^2+\epsilon}\right) \sqrt{(\Delta x^{k+1})' \nabla^2 f(x) \Delta x^{k+1}}.}
By the fact that $\tilde \lambda(x^k)\leq\theta^k+\tau\leq\phi<1$, we can apply Lemma \ref{lemmaMatrix} and obtain,
\alignShort {\lambda(x^{k+1})^2&\leq \frac{1}{1-\tilde\lambda(x^k)}\left(\frac{\tilde\lambda(x^k)^2}{1-\tilde\lambda(x^k)}+\sqrt{p^2 \tilde\lambda(x^k)^2+\epsilon}\right)\sqrt{(\Delta x^{k+1})' \nabla^2 f(x+\Delta \tilde x) \Delta x^{k+1}}\\
        &= \left(\frac{\tilde\lambda(x^k)^2}{(1-\tilde\lambda(x^k))^2}+\frac{\sqrt{p^2 \tilde\lambda(x^k)^2+\epsilon}}{1-\tilde\lambda(x^k)}\right)\lambda(x^{k+1}).
} By dividing the last line by $\lambda(x^{k+1})$, this
yields\alignShort{\lambda(x^{k+1})\leq\frac{\tilde\lambda(x^k)^2}{(1-\tilde\lambda(x^k))^2}+
\frac{\sqrt{p^2\tilde\lambda(x^k)^2+\epsilon}}{1-\tilde\lambda(x^k)}\leq\frac
{\tilde\lambda(x^k)^2}{(1-\tilde\lambda(x^k))^2}+\frac{p
\tilde\lambda(x^k)+\sqrt\epsilon}{1-\tilde\lambda(x^k)}.} From Eq.\
(\ref{ineq:lemmaLambdaEps}), we have $\tilde \lambda (x^k)\leq
\frac{\lambda(x^k)+\sqrt\epsilon}{1-p}$. Therefore the above
relation implies
\alignShort{\lambda(x^{k+1})\leq\left(\frac{\lambda(x^k)+\sqrt\epsilon}{1-p-\lambda(x^k)-\sqrt\epsilon}
\right)^2
+\frac{p\lambda(x^k)+\sqrt\epsilon}{1-p-\lambda(x^k)-\sqrt\epsilon}.}
By Eq.\ (\ref{ineq:267}), we have $\lambda(x^k)\leq \phi$,  and
therefore the above relation can be relaxed to
\alignShort{\lambda(x^{k+1})\leq\left(\frac{\lambda(x^k)}{1-p-\phi-\sqrt\epsilon}\right)^2
+\frac{\phi
p+\sqrt\epsilon}{1-p-\phi-\sqrt\epsilon}+\frac{2\phi\sqrt\epsilon+\epsilon}
{(1-p-\phi-\sqrt\epsilon)^2}.} Hence, by definition of $\xi$ and
$v$, we have \alignShort {\lambda(x^{k+1})\leq v\lambda(x^k)^2+\xi.}
\end{proof}

In the next theorem, building upon the preceding lemma, we apply relation (\ref{eq:selfConOptimal}) to
bound the suboptimality in our algorithm, i.e.,  $f(x^k)-f^*$, using
the exact Newton decrement. We show that under the above assumption,
the objective function value $f(x^k)$ generated by our algorithm
converges quadratically in terms of the primal iterations to an explicitly characterized error
neighborhood of the optimal value $f^*$.
\begin{theorem}\label{thm:quad}
Let $\{x^k\}$ be the primal sequence generated by the inexact Newton
method (\ref{eq:inexaAlgo}) and $\lambda(x^k)$, $\tilde
\lambda(x^k)$ be the exact and inexact Newton decrements at $x^k$
[cf.\ Eqs.\ (\ref{eq:lambdaExactAl}) and (\ref{eq:lambdaInexact})].
Let $f(x^k)$ be the corresponding objective function value at
$k^{th}$ iteration and $f^*$ denote the optimal objective function
value for problem (\ref{eqFormulation}). Let Assumption \ref{ass:01}
hold, and $\xi$ and $v$ be the scalars defined in Assumption
\ref{ass:01}. Assume that for some $\delta \in [0,1/2)$,
\alignShort{\xi+v\xi\leq \frac{\delta}{4v}.} Then for all $m\ge 1$,
we have \be\label{ineq:quadConv}\lambda(x^{\bar k+m})\leq
\frac{1}{2^{2^m}v}+\xi+\frac{\delta}{v}\frac{2^{2^m-1}-1}{2^{2^m}},\ee
and \alignShort{\mbox{limsup}_{m\to\infty}f(x^{\bar k+m})-f^*\leq
\xi + \frac{\delta}{2v},} where $\bar k$ is the iteration index with $\theta^{\bar k}<V$.
\end{theorem}
\begin{proof}
We prove Eq.\ (\ref{ineq:quadConv}) by induction. First for $m=1$,
from Assumption \ref{ass:stepsize}, we have $\tilde\lambda (x^{\bar
k})\leq \theta^{\bar k}+ \tau^{\bar k}$. Relation
(\ref{ineq:lambdaBound}) implies $\theta^{\bar k}+ \tau^{\bar k}\leq
\phi<1$, hence we have $\tilde \lambda(x^{\bar k})<1$ and we can
apply Lemma \ref{lemma:quadLambda} and obtain
\[\lambda(x^{\bar k+1})\leq v\lambda(x^{\bar k})^2+\xi.\]
By Assumption \ref{ass:01} and Eq.\ (\ref{ineq:lemmaLambdaEps}), we
have
\begin{align}\label{ineq:267}
\lambda(x^{\bar k}) \leq (1+p)(\theta^{\bar k}+ \tau^{\bar k})+\sqrt\epsilon\leq \phi.
\end{align}

The above two relations imply
\[\lambda(x^{\bar k+1}) \leq v\phi^2+\xi.\]
The right hand side is monotonically increasing in $\phi$. Since
$\phi\leq 0.68$, we have by Eq.\ (\ref{ineq:next68}),
$\lambda(x^{\bar k+1})\leq 0.68$. By relation (\ref{ineq:pepBound}),
we obtain $(1-p-\phi-\sqrt\epsilon)^4\geq 4\phi^2$. Using the
definition of $v$, i.e., $v=\frac{1}{(1-p-\phi-\sqrt\epsilon)^2}$,
the above relation implies $v\phi^2\leq \frac{1}{4v}$. Hence we have
\[\lambda(x^{\bar k+1})\leq \frac{1}{4v}+\xi.\]
This establishes relation (\ref{ineq:quadConv}) for $m=1$.

We next assume that Eq.\ (\ref{ineq:quadConv}) holds and
$\lambda(x^{\bar k+m})\leq 0.68$ for some $m>0$, and show that these
also hold for $m+1$. From Eqs.\ (\ref{ineq:lemmaLambdaEps}) and
(\ref{ineq:lambdatilde}), we have
\[\tilde \lambda(x^{\bar k+m})\leq \frac{\lambda(x^{\bar k+m})+\sqrt{\epsilon}}{1-p}\leq \frac{0.68+\sqrt{\epsilon}}{1-p}\leq 1,\footnote{Note that we do not need monotonicity in $\tilde \lambda (x^k)$, instead the error level assumption from relation (\ref{ineq:lambdatilde}) enables us to use Lemma \ref{lemma:quadLambda} to establish quadratic rate of convergence.}\]
where in the second inequality we used the inductive hypothesis that
$\lambda(x^{\bar k+m})\leq 0.68$. Hence we can apply Eq.\
(\ref{ineq:inducStep}) and obtain
\[\lambda(x^{{\bar k}+m+1})\leq v\lambda(x^{{\bar k}+m})^2+\xi,\]
using Eq.\ (\ref{ineq:next68}) and $\lambda(x^{\bar { k}+m})\leq
0.68$ once more, we have $\lambda(x^{{\bar k}+m+1})\leq 0.68$. From
our inductive hypothesis that (\ref{ineq:quadConv}) holds for $m$,
the above relation also implies
\begin{align*}
\lambda(x^{{\bar k}+m+1}) &\leq v\left(\frac{1}{2^{2^m}v}+\xi+\frac{\delta}{v}\frac{2^{2^m-1}-1}{2^{2^m}}\right)^2+\xi\\
&= \frac{1}{2^{2^{m+1}}v}+\frac{\xi}{2^{2^m-1}}+\frac{\delta}{v}\frac{2^{2^m-1}-1}{2^{2^{m+1}-1}}+v\left(\xi+\frac{\delta}{v}\frac{2^{2^m-1}-1}{2^{2^m}}\right)^2+\xi,
\end{align*}

Using algebraic manipulations and the assumption that $\xi+v\xi\leq \frac{\delta}{4v}$, this yields
\be\lambda(x^{{\bar k}+m+1})\leq \frac{1}{2^{2^{m+1}}v}+\xi+\frac{\delta}{v}\frac{2^{2^{m+1}-1}-1}{2^{2^{m+1}}},\nonumber\ee
completing the induction and therefore the proof of relation (\ref{ineq:quadConv}).

The induction proof above suggests that the condition $\lambda
(x^{\bar k+m})\leq 0.68$ holds for all $m>0$, we can therefore apply
Lemma \ref{lemma:quadConverg}, and obtain an upper bound on
suboptimality as follows, \be f(x^{{\bar k}+m})-f^* \leq
\left(\lambda(x^{{\bar k}+m})\right)^2\leq\lambda(x^{{\bar
k}+m}).\nonumber\ee Combining this with Eq.\ (\ref{ineq:quadConv}),
we obtain \alignShort{f(x^{{\bar
k}+m})-f^*\leq\frac{1}{2^{2^m}v}+\xi+\frac{\delta}{v}\frac{2^{2^m-1}-1}{2^{2^m}}.}
Taking limit superior on both sides of the preceding relation
establishes the final result.
\end{proof}
The above theorem shows that the objective function value $f(x^k)$
generated by our algorithm converges in terms of the primal iterations quadratically to a neighborhood
of the optimal value $f^*$, with the neighborhood of size
$\xi+\frac{\delta}{2v}$, where \[\xi = \frac{\phi
p+\sqrt\epsilon}{1-p-\phi-\sqrt\epsilon}+\frac{2\phi\sqrt\epsilon+\epsilon}
{(1-p-\phi-\sqrt\epsilon)^2},\quad  \quad
v=\frac{1}{(1-p-\phi-\sqrt\epsilon)^2},\] and the condition
$\xi+v\xi\leq \frac{\delta}{4v}$ is satisfied. Note that with the
exact Newton algorithm, we have $p=\epsilon=0$, which implies
$\xi=0$ and we can choose $\delta =0$, which in turn leads to the
size of the error neighborhood being $0$. This confirms the fact
that the exact Newton algorithm converges quadratically to the
optimal objective function value.

\subsection{Convergence with respect to Design Parameter $\mu$}\label{subsec:mu}
In the preceding development, we have restricted our attention to
develop an algorithm for a given logarithmic barrier coefficient
$\mu$. We next study the convergence property of the optimal object
function value as a function of $\mu$, in order to develop a method
to bound the error introduced by the logarithmic barrier functions
to be arbitrarily small. We utilize the following result from
\cite{InteriorBook}.

\begin{lemma}\label{lemma:dfSC}
Let $G$ be a closed convex domain, and function $g$ be a self-concordant barrier function for $G$, then for any $x$, $y$ in interior of $G$, we have $(y-x)'\nabla g(x)\leq 1$.
\end{lemma}

Using this lemma and an argument similar to that in
\cite{InteriorBook}, we can establish the following result, which
bounds the sub-optimality as a function of $\mu$.

\begin{theorem}\label{thm:mu}
Given $\mu\ge 0$, let $x(\mu)$ denote the optimal solution of
problem (\ref{eqFormulation}) and $h(x(\mu)) = \sum_{i=1}^S
-U_i(x_i(\mu))$ . Similarly, let $x^*$ denote the optimal solution
of problem (\ref{ineqFormulation}) together with corresponding slack
variables (defined in Eq.\ (\ref{eq:slack})), and $h^* =
\sum_{i=1}^S -U_i(x_i^*)$. Then, the following relation holds,
\als{h(x(\mu))-h^*\leq \mu.}
\end{theorem}
\begin{proof}
For notational simplicity, we write $g(x) = -\sum_{i =1}^{S+L}
\log{({x_i})}$. Therefore the objective function for problem
(\ref{eqFormulation}) can be written as $h(x)+\mu g(x)$. By
Assumption \ref{asmp:utility}, we have that the utility functions
are concave, therefore the negative objective functions in the
minimization problems are convex. From convexity, we obtain
\be\label{ineq:h*}h(x^*) \geq h(x(\mu))+(x^*-x(\mu))'\nabla
h(x(\mu)).\ee By optimality condition for $x(\mu)$ for problem
(\ref{eqFormulation}) for a given $\mu$, we have, \als{(\nabla
h(x(\mu))+\mu\nabla g(x(\mu)))'(x-x(\mu))\geq 0,} for any feasible
$x$. Since $x^*$ is feasible, we have \als{(\nabla h(x(\mu))+\mu\nabla
g(x(\mu)))'(x^*-x(\mu))\geq 0,} which implies \als{\nabla
h(x(\mu))'(x^*-x(\mu))\geq -\mu\nabla g(x(\mu))'(x^*-x(\mu)).} For
any $\mu$, we have $x(\mu)$ belong to the interior of the feasible
set, and by Lemma \ref{lemma:dfSC}, we have for all $\tilde\mu$,
$\nabla g(x(\mu))'(x(\tilde\mu)-x(\mu))\leq 1$. By continuity of
$x(\mu)$ and the fact that the convex set $Ax\leq c$ is closed, for
$A$ and $c$ defined in problem (\ref{eqFormulation}), we have $x^* =
\lim_{\mu \to 0} x(\mu)$, and hence \als{\nabla
g(x(\mu))'(x^*-x(\mu))= \lim_{\tilde\mu \to 0}\nabla
g(x(\mu))'(x(\tilde\mu)-x(\mu))\leq 1.} The preceding two relations
imply \als{\nabla h(x(\mu))'(x^*-x(\mu))\geq -\mu.} In view of
relation (\ref{ineq:h*}), this establishes the desired result, i.e.,
\als{h(x(\mu))-h^*\leq \mu.}
\end{proof}

By using the above theorem, we can develop a method to bound the
sub-optimality between the objective function value our algorithm
provides for problem (\ref{eqFormulation}) and the exact optimal
objective function value for problem (\ref{ineqFormulation}), i.e,
the sub-optimality introduced by the barrier functions in the
objective function, such that  for any positive scalar $a$, the
following relation holds,
\be\label{ineq:subOpReq}\frac{h(x(\mu))-h^*}{h^*}\leq a,\ee where
the value $h(x(\mu))$ is the value obtained from our algorithm for
problem (\ref{eqFormulation}), and $h^*$ is the optimal objective
function value for problem (\ref{ineqFormulation}). We achieve the
above bound by implementing our algorithm twice. The first time
involves running the algorithm for problem (\ref{eqFormulation})
with some arbitrary $\mu$. This leads to a sequence of $x^k$
converging to some $x(\mu)$. Let $h(x(\mu)) = \sum_{i=1}^S
-U_i(x_i(\mu))$. By Theorem \ref{thm:mu}, we have
\be\label{ineq:h1}h(x(\mu))-\mu\leq h^*.\ee Let scalar $M$ be such
that $M=(a[h(x(\mu))-\mu])^{-1}$ and implement the algorithm one
more time for problem (\ref{eqFormulation}), with $\mu=1$ and the
objective function multiplied by $M$, i.e., the new objective is to
minimize $-M\sum_{i=1}^S U_i(x_i)-\sum_{i =1}^{S+L} \log{({x_i})}$,
subject to link capacity constraints.\footnote{When $M<0$, we can
simply add a constant to the original objective function to shift it
upward. Therefore the scalar $M$ can be assumed to be positive
without loss of generality. If no estimate on $M$ is available apriori, we can implement the distributed algorithm one more time in the beginning to obtain an estimate to generate the constant accordingly.} We obtain a sequence of $\tilde x^k$
converges to some $\tilde x(1)$. Denote the objective function value
as $h(\tilde x(1))$, then by applying the preceding theorem one more
time we have
\[Mh(\tilde x(1))-Mh^*\leq \mu=1,\]
which implies
\[h(\tilde x(1))-h^*\leq a[h\left(x(\mu)\right)-\mu]\leq a h^*\]
where the first inequality follows by definition of the positive
scalar $M$ and the second inequality follows from relation
(\ref{ineq:h1}). Hence we have the desired bound
(\ref{ineq:subOpReq}).

Therefore even with the introduction of the logarithmic barrier
function, the relative error in the objective function value can be
bounded by an arbitrarily small positive scalar at the cost of
performing the fast Newton-type algorithm twice.

\section{Simulation Results}\label{sec:sim}
Our simulation results demonstrate that the decentralized Newton
method significantly outperforms the existing methods in terms of
number of iterations. For our distributed Newton method, we used the following error tolerance levels: $p=10^{-3}$, $\epsilon = 10^{-4}$ [cf.\ Assumption \ref{ass:errorBoundEps}], $\tau = 10^{-2}$ [cf.\ Assumption \ref{ass:stepsize}] and when $\theta^{\bar k}>V =0.12$ we switch stepsize choice to be $d^k=1$ for all $k\geq \bar k$. With these error tolerance levels, both Assumptions \ref{ass:errorBoundEps} and \ref{ass:01} can be satisfied. We executed distributed Newton method twice with different scaling and barrier coefficients according to Section \ref{subsec:mu} with $B=10^{-2}$ to confine the error in the objective function value to be within $1\%$ of the optimal value. For a comprehensive comparison, we count both
the primal and dual iterations implemented through distributed error checking method described in Appendix \ref{app:errorbounds}.\footnote{In these simulations we
did not include the number of steps required to
compute the stepsize (distributed summation with finite termination) and to implement distributed error checking (maximum consensus) to allow the possibilities that other methods can be used to compute these.  Note that the number of iterations required by both of these computation is upper bounded by the number of sources, which is a small constant (8 for example) in our simulations.} In particular, in what
follows, the number of iterations of our method refers to the sum of dual iterations at each of the generated
primal iterate. In the simulation results, we compare our
distributed Newton method performance against both the
subgradient method used in \cite{LowLapsley} and the Newton-type
diagonal scaling dual method developed in \cite{Low}. Both of these methods were implemented using a constant stepsize that can guarantee convergence as shown in \cite{LowLapsley} and \cite{Low}.

\begin{figure}
\begin{center}
\includegraphics[scale = 0.6]{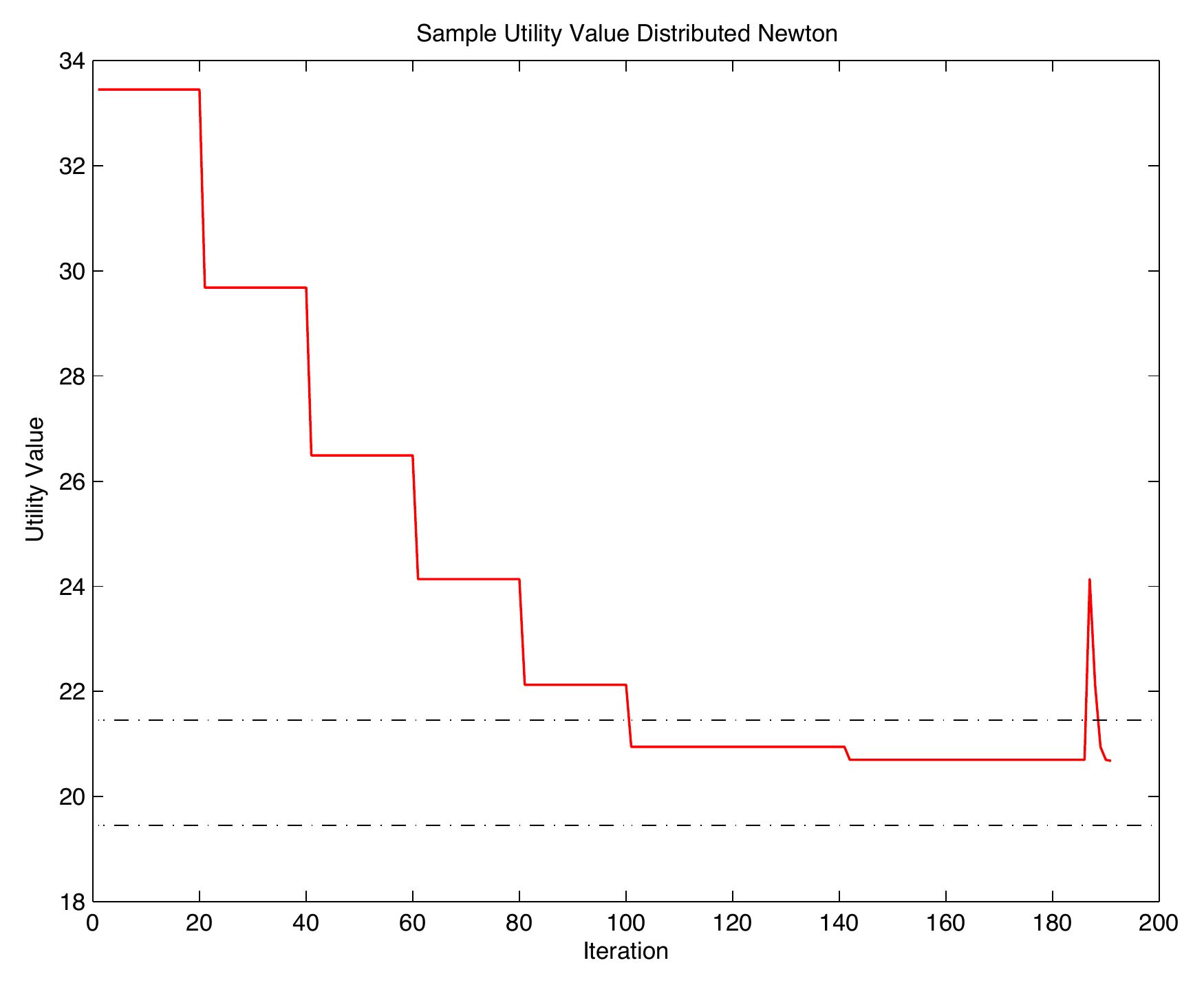}
\caption{One sample objective function value of distributed Newton method against number of iterations. The dotted black lines denote $\pm5\%$ interval of the optimal objective function value.}\label{fig:newton}
\end{center}
\end{figure}
\begin{figure}
\begin{center}
\vspace{-6cm}
\includegraphics[scale = 0.7]{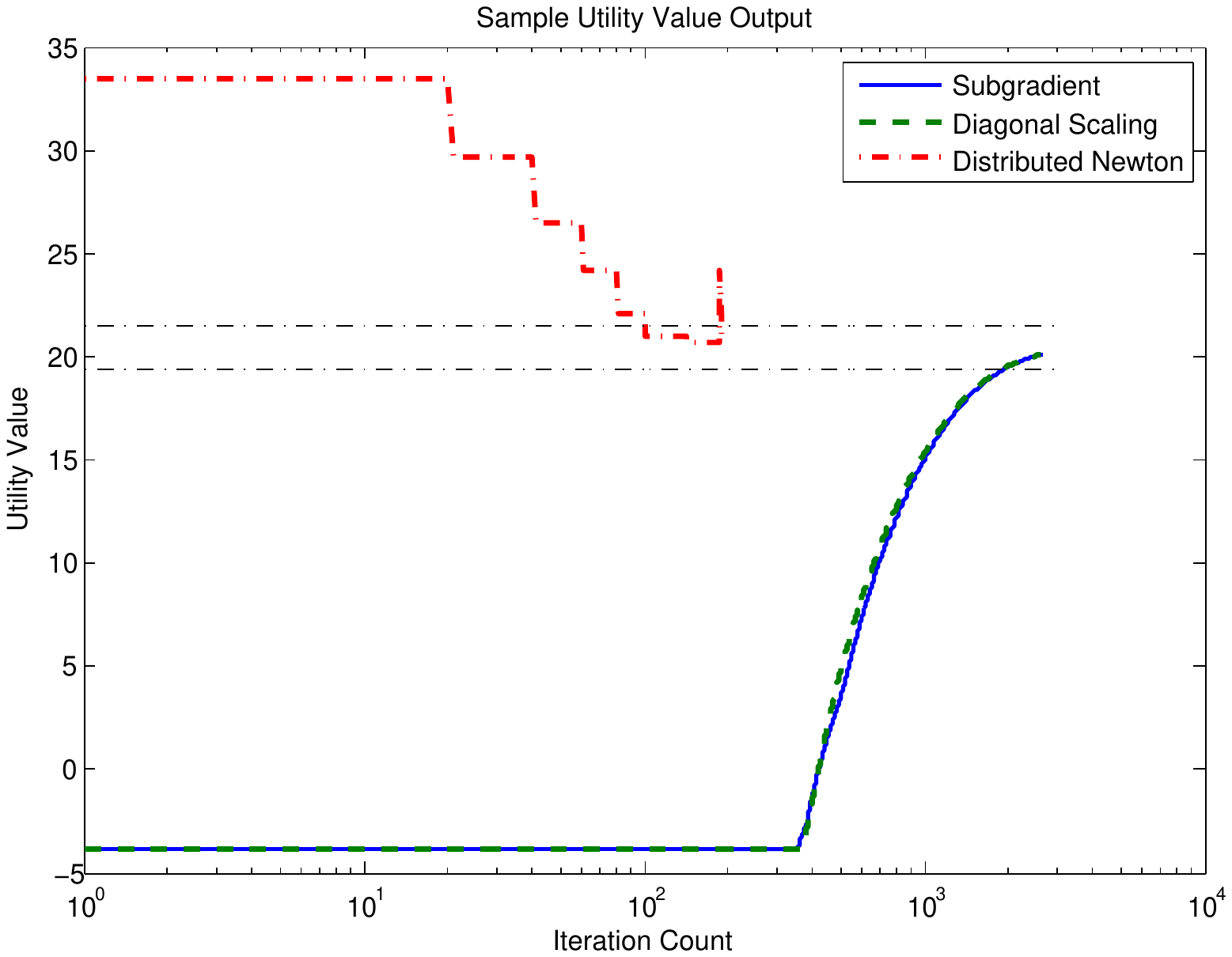}
\vspace{-6cm}
\caption{One sample objective function value of all three methods against log scaled iteration count. The dotted black lines denote $\pm5\%$ interval of the optimal objective function value. }\label{fig:sampleOutput}
\end{center}
\end{figure}

\begin{figure}
\begin{center}
\includegraphics[scale = 0.6]{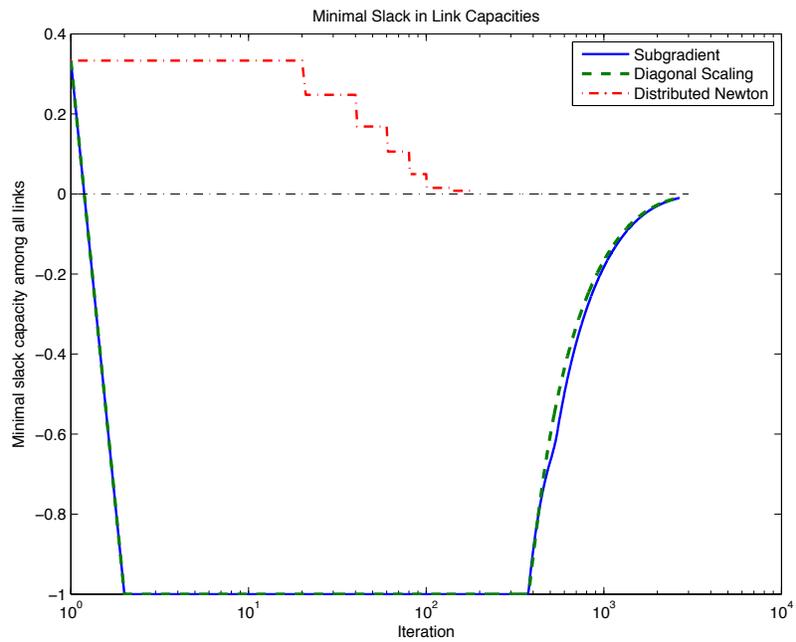}
\caption{Sample minimal slack in link capacity of all three methods against log scaled iteration count. Negative slack means violating capacity constraint.  The dotted black line denotes $0$.}\label{fig:slack}
\end{center}
\end{figure}

\begin{figure}
\begin{center}
\includegraphics[scale = 0.6]{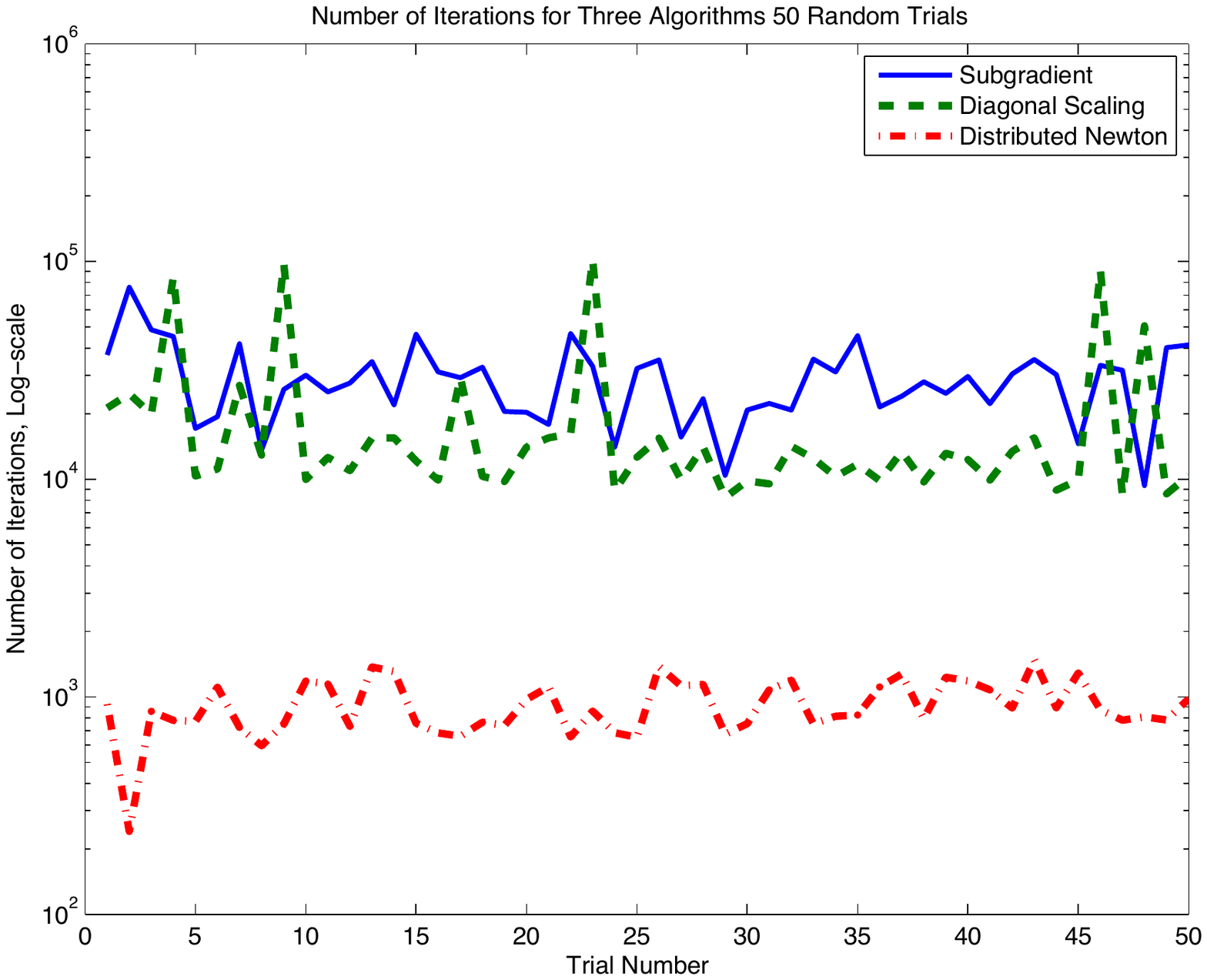}
\caption{Log scaled iteration count for the 3 methods implemented over 50 randomly generated networks.}\label{fig:50trials}
\end{center}
\end{figure}

A sample evolution of the objective function value of the distributed Newton method is presented in Figure \ref{fig:newton}. This is generated for the network in Figure $\ref{fig:withFlow}$. The horizontal line segments correspond to the dual iterations, where the primal vector stays constant, and each jump in the figure is a primal Newton update. The spike close to the end is a result of rescaling and using a new barrier coefficient in the second round of the distributed Newton algorithm [cf.\ Section \ref{subsec:mu}]. The black dotted lines indicate $\pm5\%$ interval around the optimal objective function value.

The other two algorithms were implemented for the same problem, and the objective function values are plotted in Figure \ref{fig:sampleOutput}, with logarithmic scaled iteration count on the $x$-axis. We use black dotted lines to indicate $\pm5\%$ interval around the optimal objective function value. While the subgradient and diagonal scaling methods have similar convergence behavior, the distributed Newton method significantly outperforms the two.

One of the important features of the distributed Newton method is that, unlike the other two algorithms, the generated primal iterates satisfy the link capacity constraint throughout the algorithm. This observation is confirmed by Figure \ref{fig:slack}, where the minimal slacks in links are shown for all three algorithms. The black dotted line is the zero line and a negative slack means violating the capacity constraint. The slacks that our distributed Newton method yields always stays above the zero line, while the other two only becomes feasible in the end.

To test the performances of the methods over general networks, we generated 50 random networks, with number of links $L=15$ and  number of sources $S=8$. Each routing matrix consists of $L\times R$ Bernoulli random variables.\footnote{When there exists a source that does not use any links or a link that is not used by any sources, we discard the routing matrix and generate another one.} All three methods are implemented over the 50 networks. We record the number of iterations upon termination for all 3 methods, and results are shown in Figure \ref{fig:50trials} on a log scale. The mean number of iterations to convergence from the 50 trials is $924$ for distributed Newton method, $20286$ for Newton-type diagonal scaling and $29315$ for subgradient method.

\section{Conclusions}\label{sec:conclu}
This paper develops a distributed Newton-type second order algorithm
for network utility maximization problems, which can achieve
superlinear convergence rate in primal iterates within some error
neighborhood. We show that the computation of the dual Newton step
can be implemented in a decentralized manner using a matrix
splitting scheme. The key feature of this scheme is that its
implementation uses an information exchange mechanism similar to
that involved in first order methods applied to this problem. We
show that even when the Newton direction and stepsize are computed
with some error, the method achieves superlinear convergence rate in
terms of primal iterations to an error neighborhood. Simulation
results also indicate significant improvement over traditional
distributed algorithms for network utility maximization problems.
Possible future directions include a more detailed analysis of the
relationship between the rate of convergence of the dual iterations
and the underlying topology of the network and investigating
convergence properties for a fixed finite truncation of dual
iterations.

\appendix
\section{Distributed Stepsize Computation}\label{app:stepsizeComp}
In this section, we describe a distributed procedure with finite
termination to compute stepsize $d^k$ according to Eq.\
(\ref{eq:stepsize}). We first note that in Eq.\ (\ref{eq:stepsize}),
the scalar $b\in (0,1)$ is predetermined and the only unknown term
is the inexact Newton decrement $\tilde\lambda (x^k)$. In order to
compute the value of $\tilde\lambda (x^k)$, we rewrite the inexact
Newton decrement based on definition (\ref{eq:lambdaInexact}) as
$\tilde{\lambda} (x^k) = \sqrt{\sum_{i\in \mathcal{S}} (\Delta
\tilde x_i^k)^2(H_k)_{ii} + \sum_{l\in \mathcal{L}} (\Delta \tilde
x_{l+S}^k)^2(H_k)_{(l+S)(l+S)}} $, or equivalently,
\be\label{eq:sum}\left(\tilde{\lambda} (x^k)\right)^2=\sum_{i\in
\mathcal{S}} (\Delta \tilde x_i^k)^2(H_k)_{ii} + \sum_{l\in
\mathcal{L}} (\Delta \tilde x_{l+S}^k)^2(H_k)_{(l+S)(l+S)}.\ee

In the sequel, we develop a distributed summation procedure to
compute this quantity by aggregating the local information available
on sources and links. A key feature of this procedure is that it
respects the simple information exchange mechanism used by first
order methods applied to the NUM problem: information about the links along the routes is aggregated
 and sent back to the sources using a
feedback mechanism. Over-counting is avoided using a novel off-line
construction, which forms an \textit{(undirected) auxiliary graph}
that contains information on sources sharing common links.

Given a network with source set $\mathcal{S} = \{1, 2, \ldots, S\}$
(each associated with a predetermined route) and link set
$\mathcal{L}=\{1, 2, \ldots, L\}$, we define the set of nodes in the
auxiliary graph as the set $\mathcal{S}$, i.e., each node
corresponds to a source (or equivalently, a flow) in the original
network. The edges are formed between sources that share common
links according to the following iterative construction. In this
construction, each source is equipped with a state (or color) and
each link is equipped with a set (a subset of sources), which are
updated using signals sent by the sources along their routes.

\noindent\textbf{Auxiliary Graph Construction:}
\begin{itemize}
\item Initialization:  Each link $l$ is associated with a set $\Theta_l=\emptyset$. One arbitrarily chosen source is marked as grey, and the rest are marked as white. The grey source sends a signal $\{$label, $i\}$ to its route. Each link $l$ receiving the signal, i.e., $l\in L(i)$, adds $i$ to $\Theta_l$.
\item Iteration:  In each iteration, first the sources update their states and
send out signals according to step (A). Each link $l$ then receives
signals sent in step (A) from the sources $i \in S(l)$ and updates
the set $\Theta_l$ according to step (B).
\begin{itemize}
\item[(A)] Each source $i$:
\begin{itemize}
\item[(A.a)] If it is white, it sums up $|\Theta_l|$ along its route, using the value $|\Theta_l|$ from the previous time.
\begin{itemize}
\item[(A.a.1)] If $\sum_{l\in L(i)}|\Theta_l|>0$, then the source $i$ is marked grey and it sends two signals $\{$neighbor, $i\}$ and $\{$label, $i\}$ to its route.
\item[(A.a.2)] Else, i.e., $\sum_{l\in L(i)}|\Theta_l|=0$, source $i$ does nothing for this iteration.
\end{itemize}
\item[(A.a)] Otherwise, i.e., it is grey, source $i$ does nothing.
\end{itemize}
\item[(B)] Each link $l$:
\begin{itemize}
\item[(B.a)]If $\Theta_l = \emptyset$:
\begin{itemize}
\item[(B.a.1)] If it experiences signal $\{$label, $i\}$ passing through it, it adds $i$ to $\Theta_l$. When there are more than one such signals during the same iteration, only the smallest $i$ is added. The signal keeps traversing the rest of its route.
\item[(B.a.2)] Otherwise link $l$ simply carries the signal(s) passing through it, if any, to the next link or node.
\end{itemize}
\item[(B.b)] Else, i.e., $\Theta_l\neq \emptyset$:
\begin{itemize}
\item[(B.b.1)] If it experiences signal $\{$neighbor, $i\}$ passing through it, an
edge $(i,j)$  with label $L_l$ is added to the auxiliary graph for
all $j\in\Theta_l$, and then $i$ is added to the set $\Theta_l$. If
there are more than one such signals during the same iteration, the
sources are added sequentially, and the resulting nodes in the set
$\Theta_l$ form a clique in the auxiliary graph. Link $l$ then stops
the signal, i.e., it does not pass the signals to the next link or
node.
\item[(B.b.2)] Otherwise link $l$ simply carries the signal(s) passing through it, if any, to the
next link or node.
\end{itemize}
\end{itemize}
\end{itemize}
\item Termination: Terminate after $S-1$ number of iterations.
\end{itemize}

The auxiliary graph construction process for the sample network in
Figure \ref{fig:bigPrimal} is illustrated in Figure \ref{fig:agc1},
where the left column reflects the color of the nodes in the
original network and the elements of the set $\Theta_l$ (labeled on
each link $l$), while the right column corresponds to the auxiliary
graph constructed after each iteration.\footnote{Note that depending
on construction, a network may have different auxiliary graphs
associated with it. Any of these graphs can be used in the
distributed summation procedure.}

We next investigate some properties of the auxiliary graph, which
will be used in proving that our distributed summation procedure
yields the corrects values.

\begin{figure}
\centering
\includegraphics[scale = 0.7]{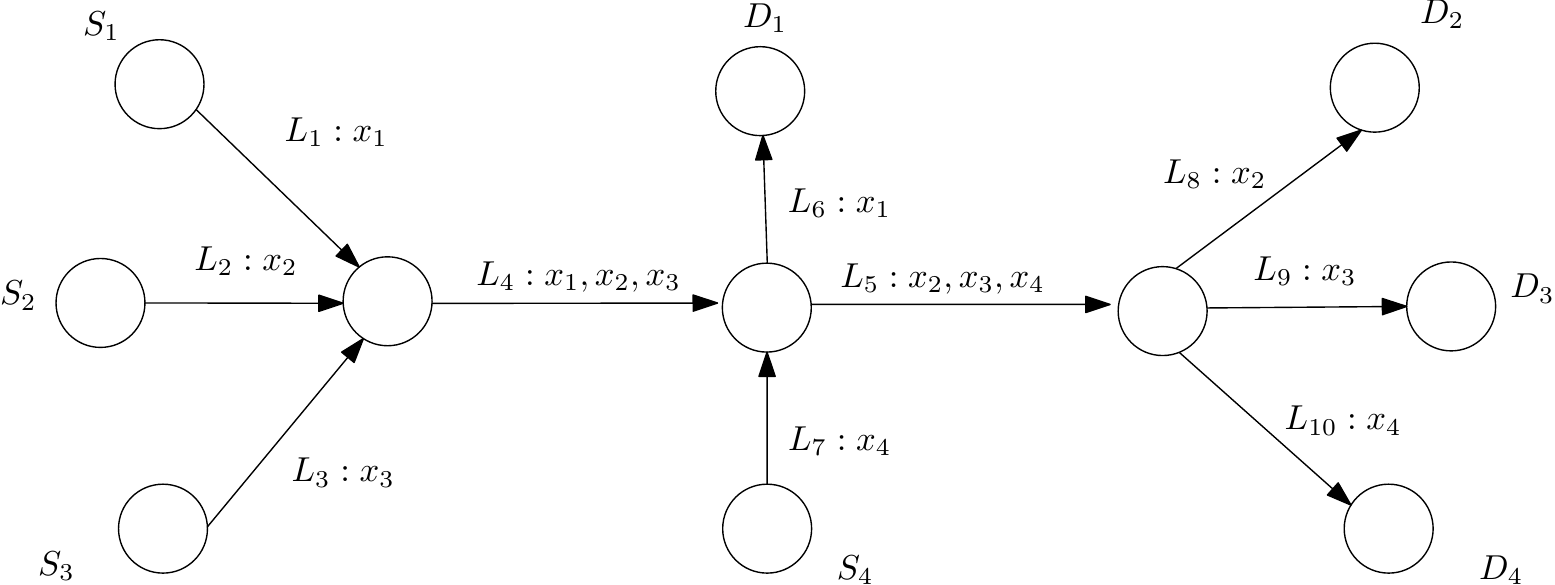}
\caption{A sample network with four sources and ten links. Each link
shows the flows (or sources) using that link. This example will be
used to illustrate different parts of the distributed stepsize
computation in this section.}\label{fig:bigPrimal}
\end{figure}

\begin{figure}
\subfloat[State of the network $t=0$]{
\includegraphics[scale=0.6]{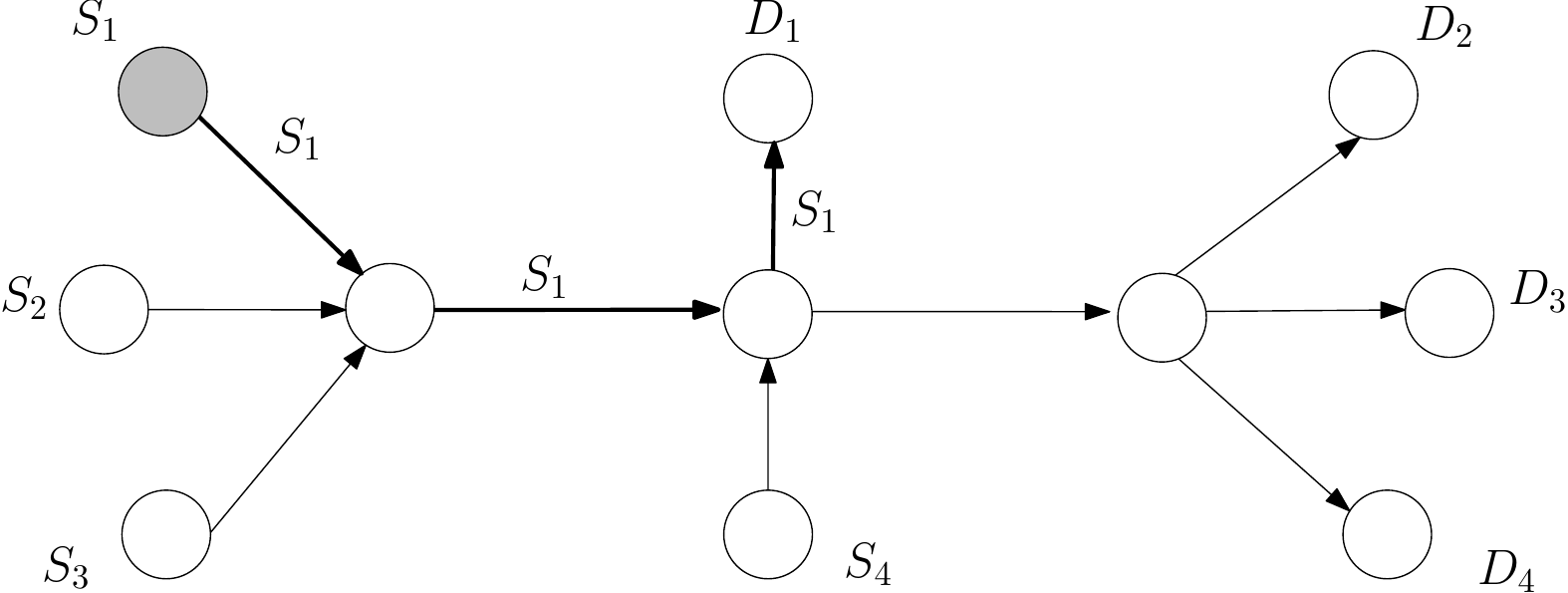}
}\qquad\qquad
\subfloat[State of the auxiliary graph $t=0$]{
\includegraphics[scale=0.5]{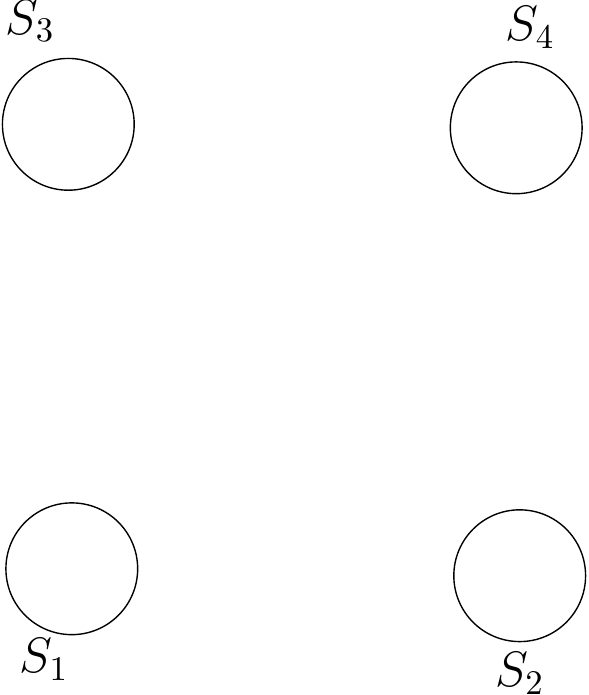}}\\
\subfloat[State of the network $t=1$]{
\includegraphics[scale=0.6]{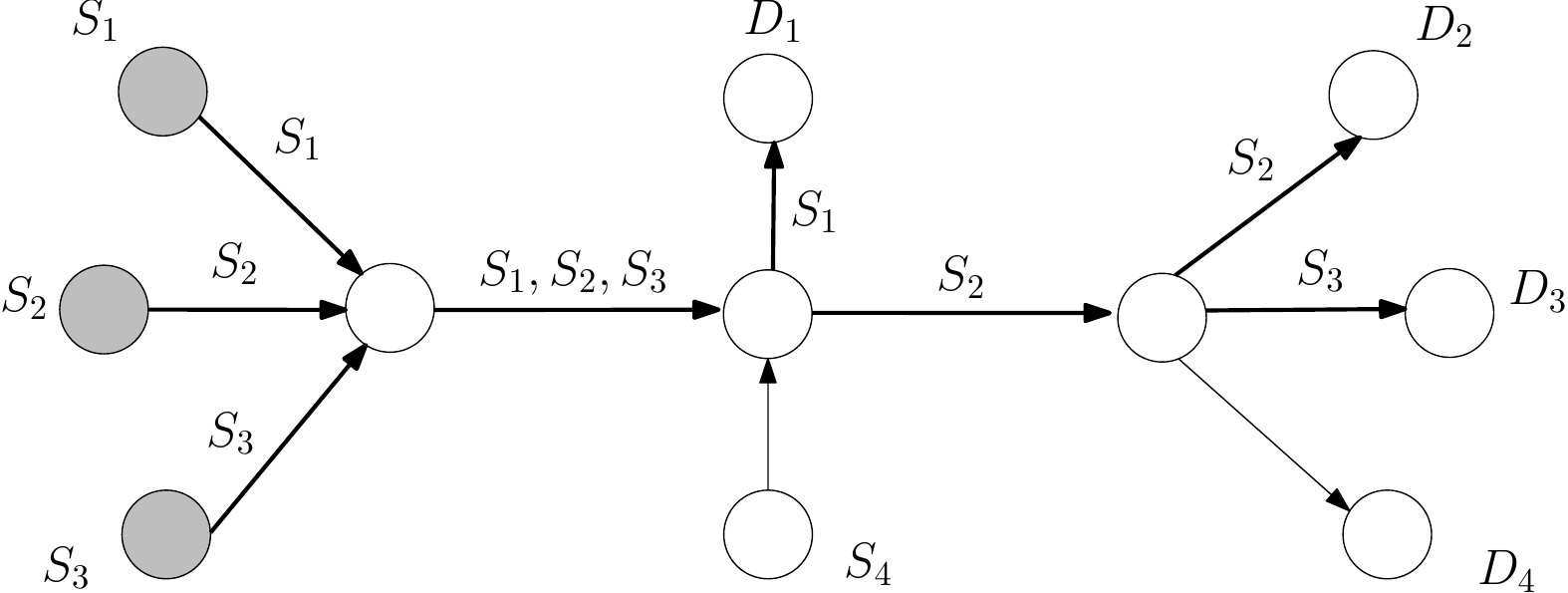}
}\qquad\qquad
\subfloat[State of the auxiliary graph $t=1$]{
\includegraphics[scale=0.5]{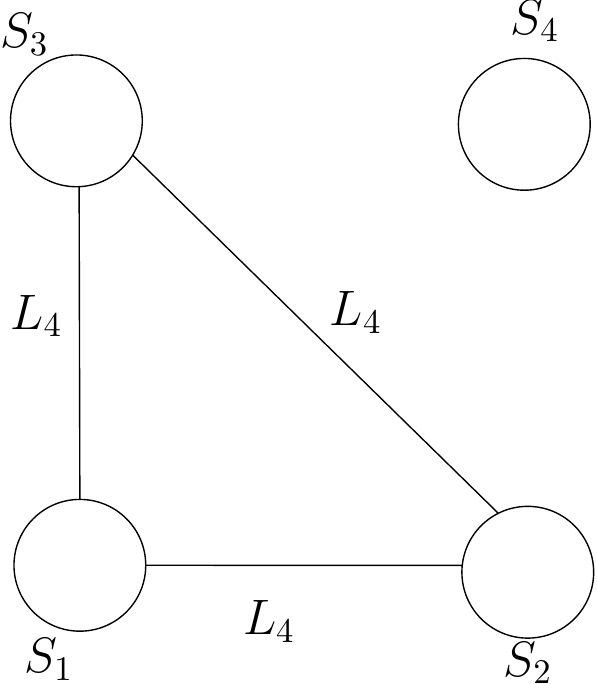}}
\\
\subfloat[State of the network $t=2$]{
\includegraphics[scale=0.6]{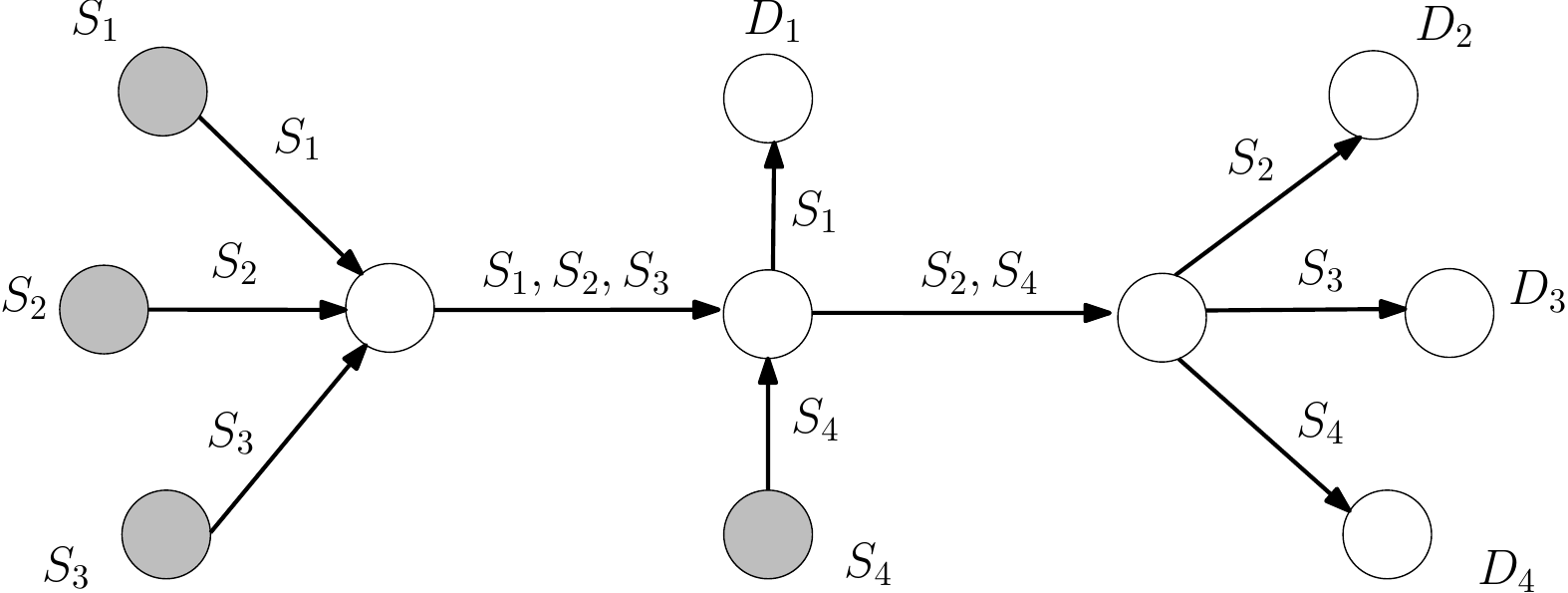}
}\qquad\qquad
\subfloat[State of the auxiliary graph $t=2$]{
\includegraphics[scale=0.5]{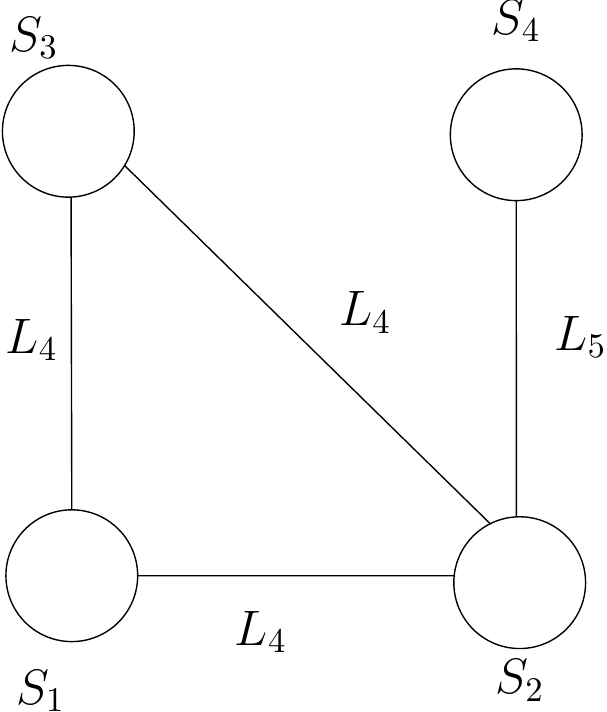}}\\
\subfloat[State of the network $t=3$]{
\includegraphics[scale=0.6]{primalagct5.pdf}
}\qquad\qquad
\subfloat[State of the auxiliary graph $t=3$]{
\includegraphics[scale=0.5]{auxGrapht4.pdf}}
\caption{Steps of the construction of the auxiliary graph
corresponding to the network in Figure \ref{fig:bigPrimal}. The
elements of $\Theta_l$ are labeled on link $l$. A link is drawn bold
in the original graph if $\Theta_l\neq \emptyset$.}\label{fig:agc1}
\end{figure}

\begin{lemma}\label{lemma:nocycle}
Consider a network and its auxiliary graph with sets
$\{\Theta_l\}_{l\in \mathcal{L}}$. The following statements hold:
\begin{itemize}
\item[(1)] For each link $l$, $\Theta_l\subset S(l)$.
\item[(2)] Source nodes $i, j$ are connected in the auxiliary graph if and only if there exists a link $l$, such that $\{i,j\}\subset \Theta_l$.
\item[(3)] The auxiliary graph does not contain multiple edges, i.e., there exists at most one edge between any pair of nodes.
\item[(4)] The auxiliary graph is connected.
\item[(5)] For each link $l$, $\Theta_l\neq \emptyset$.
\item[(6)] There is no simple cycle in the auxiliary graph other than that formed by only the edges with the same label.
\end{itemize}
\end{lemma}
\begin{proof} We prove the above statements in the order they are stated.
\begin{itemize}
\item[(1)] Part (1) follows immediately from our auxiliary graph construction, because each source only sends signals to links on its own route and the links only update their set $\Theta_l$ when they experience some signals passing through them.
\item[(2)] In the auxiliary graph construction, a link is added to the auxiliary graph only in
step (B.b.1), where part (2) clearly holds.
\item[(3)] From the first two parts, there is an edge between source nodes $i,j$,
i.e., $\{i,j\}\subset \Theta_l$ for some $l$, only if $i$ and $j$
share link $l$ in the original network. From the auxiliary graph
construction, if sources $i$ and $j$ share link $l$ then an edge
with label $L_l$ between $i$ and $j$ is formed at some iteration if
and only if one of the following three cases holds:
 \begin{itemize}
\item[I] In the beginning of the previous iteration $\Theta_l =\emptyset$ and sources $i, j$ are both white. During the previous iteration, source $i$ becomes grey and sends out the signal $\{$label,$i\}$ to link $l$, hence $\Theta_l=\{i\}$. In the current iteration, source $j$ with $\sum_{m\in L(j)}|\Theta_m|\geq |\Theta_l|>0$ becomes grey and sends out signal $\{$neighbor, $j\}$ to link $l$;
\item[II] The symmetric case of I, where first source $j$ becomes grey and one iteration later source $i$ becomes grey.
\item[III] In the beginning of the previous iteration $\Theta_l =\emptyset$ and sources $i, j$ are both white. During the previous iteration, some other source $t$ with $l\in L(t)$ becomes grey and sends out the signal $\{$label,$t\}$ to link $l$, hence $\Theta_l=\{t\}$. In the current iteration, both source $i$ and $j$ with $\sum_{m\in L(i)}|\Theta_m|\geq |\Theta_l|>0$ and $\sum_{m\in L(j)}|\Theta_m|\geq |\Theta_l|>0$ become grey and send out signals $\{$neighbor, $i\}$ and $\{$neighbor, $j\}$ to link $l$.
\end{itemize}
Hence if an edge connecting nodes $i$ and $j$ exists in the auxiliary graph, then in the beginning of the iteration when the edge is formed at least one of the nodes is white, and by the end of the iteration both nodes are colored grey and stay grey. Therefore the edges between $i$ and $j$ in the auxiliary graph can only be formed during exactly one iteration.

We next show that only one such edge can be formed in one iteration.
The first two cases are symmetric, and without loss of generality we
only consider cases I and III. In both of these cases, an edge
between $i$ and $j$ is formed with label $L_l$ only if link $l$
receives the signal $\{$neighbor, $j\}$ and $\Theta_l\neq
\emptyset$. In step (B.b.1) of the auxiliary graph construction, the
first link with $\Theta_l\neq \emptyset$ stops the signal from
passing to the rest of its route, hence at most one edge between $i$
and $j$ can be generated. Hence part (3) holds.
\item[(4)] By using a similar analysis as above, it is straightforward to see that if at one iteration source $i$ from the original network becomes grey, then in the next iteration all the sources which share link with $i$ become grey and are connected to $i$ in the auxiliary graph. By induction, we conclude that all the nodes in the auxiliary graph corresponding to sources colored grey in the original network are connected to the source node marked grey in the initialization step, and hence these nodes form a  connected component.

We next show that all nodes are colored grey when the auxiliary
graph construction procedure terminates. We first argue that at
least one node is marked grey from white at each iteration before
all nodes are marked grey. Assume the contrary is true, that is at
some iteration no more nodes are marked grey and there exists a set
of white nodes $S^*$. This implies that the nodes in $S^*$ do not
share any links with the nodes in $\mathcal{S}\backslash S^*$ and
thus there is no path from any source in the set
$\mathcal{S}\backslash S^*$ to any source in $S^*$ using the links
(including the feedback mechanisms) in the original network.
However, this contradicts the fact that all links form a strongly
connected graph. Therefore after $S-1$ iterations all nodes in the
original graph are colored grey and therefore we have the desired
statement hold.
\item[(5)] Analysis for part (3) suggests that all the connected nodes in the
auxiliary graph are colored grey. In view of the part (4), all
the sources are colored grey when the auxiliary graph construction
procedure terminates. Step (B.a.1) implies that a link has
$\Theta_l=\emptyset$ if all sources $i\in S(l)$ are white. Since
each link is used by at least one source, and all sources are grey,
part (5) holds.

\item[(6)] We prove part (6) by showing the auxiliary graph, when the cycles formed by the
edges of the same label are removed, is acyclic. For each link $l$,
let $i_l^*$ denote the first element added to the set $\Theta_l$ in
the auxiliary graph construction process, which is uniquely defined
for each link $l$ by Step (B.a.1). In the set $\mathcal{S}$ for each
link $l$, we define an equivalence class by $i\sim j$ if
$\{i,j\}\subset\Theta_l\backslash\{i_l^*\}$, which implies if and
only if $i$ and $j$ are connected in the auxiliary graph and $i\sim
j$, this link is formed by scenario III as defined above in the
proof of part (3).

The nodes in each equivalence class are connected by edges with the
same label, which form the undesired cycles. We remove these cycles
by merging each equivalence class into one representative node,
which inherits all the edges going between the nodes in the
equivalence class and $\mathcal{S}\backslash \Theta_l$ in the
auxiliary graph, and is connected to $i_l^*$ via one edge. Note the
resulting graph is connected, since the auxiliary graph is by
part (4) and all the remaining edges are generated under
scenarios I and II as defined in the proof of part (3).

We now show that the resulting graph contains no cycle. From cases I
and II, it follows immediately that an edge is generated when one
more source becomes grey. Therefore if number of noes is $N$, we
have $N-1$ edges. In a connected graph, this implies we have a tree,
i.e. acyclic, and hence part (6) holds.
\end{itemize}

\end{proof}

We denote the set of links inducing edges in the auxiliary graph as
$L^*=\{l\mid |\Theta_l|>1\}$ and for each source $i$ the set of
links which induce edges in the auxiliary graph as $L^*(i)=\{l \mid
i\in\Theta_l, l\in L^*\}$ for notational convenience. Each link can
identify if it is in $L^*$ by the cardinality of the set $\Theta_l$.
Each source $i$ can obtain $|L^*(i)|$ along the links on its route.
The auxiliary graph remains the same throughout the
distributed Newton algorithm and only depends on the structure of
the network (independent of the utility functions and link
capacities), therefore given a network, the above construction only
needs to be preformed once prior to execution of the distributed
Newton algorithm.

We next present a distributed procedure to compute the sum in Eq.\
(\ref{eq:sum}) and show that the sets $\Theta_l$ constructed using
the above procedure avoids over-counting and enables computation of
the correct values.\footnote{Note that the execution of the
procedure only uses the sets $\Theta_l$, $L^*$, and $L^*(i)$.We will
use the structure of the auxiliary graph in proving the correctness
of the procedure.}

\noindent\textbf{Distributed Summation Procedure:}
\begin{itemize}
\item Initialization: Each link $l$ initializes to $z_l(0)=0$. Each source $i$ computes $y_i^* = (\Delta \tilde x_{i}^k)^2(H_k)_{ii}$
and each link $l$ computes $z_l^* =\frac{1}{|S(l)|}(\Delta \tilde
x_{l+S}^k)^2(H_k)_{(l+S)(l+S)}$. Each source $i$ aggregates the sum
\be\label{eq:y0}y_i(0) = y_i^*+ \sum_{l\in L(i)} z_l^*\ee along its
route.

\item Iteration for $t=1,2,\ldots, S$. The following 3 steps are completed in the order they are presented.
\begin{itemize}
    \item[a.] Each source $i$ sends its current value $y_i(t)$ to its route.
    \item[b.] Each link $l$ uses the $y_i(t)$ received and computes \be\label{eq:zupdate}z_l(t) =
    \sum_{i\in \Theta_l} y_i(t-1) - \left(|\Theta_l|-1\right)z_l(t-1).\ee
    \item[c.] Each source $i$ aggregates information along its route from the links $l\in L^*(i)$ and computes
    \be\label{eq:yupdate} y_i(t) = \sum_{l\in L^*(i)} z_l(t) -\left(|L^*(i)|-1\right) y_i(t-1).\ee
\end{itemize}
\item Termination: Terminate after $S$ number of iterations.
\end{itemize}
\newpage
\begin{figure}
\subfloat[State of the network $t=0$]{
\includegraphics[scale=0.6]{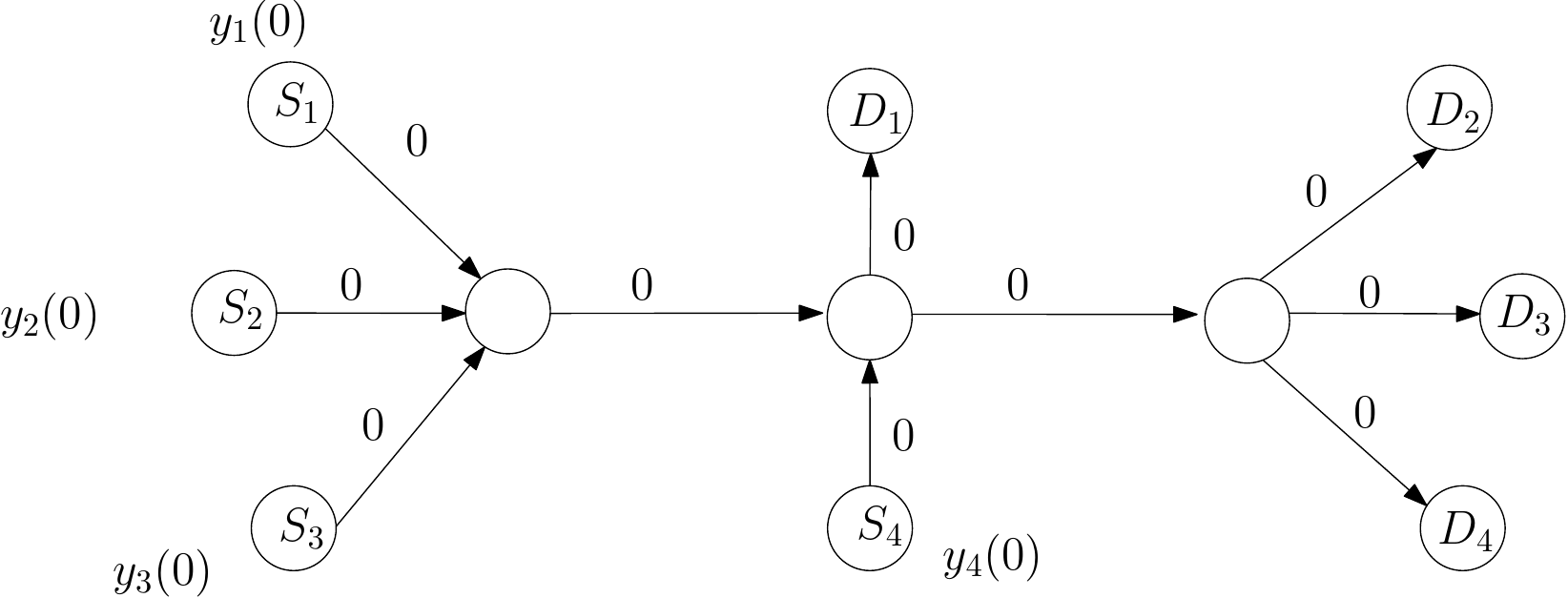}
}\hspace{0.6cm}
\subfloat[State of auxiliary graph $t=0$]{
\includegraphics[scale=0.5]{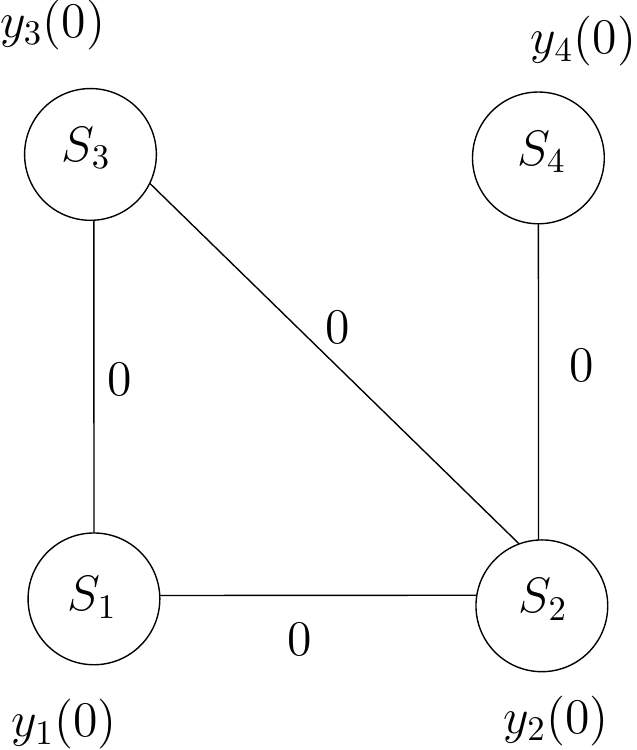}}
\\
\hspace{-0.5cm}\subfloat[State of the network $t=1$]{
\includegraphics[scale=0.6]{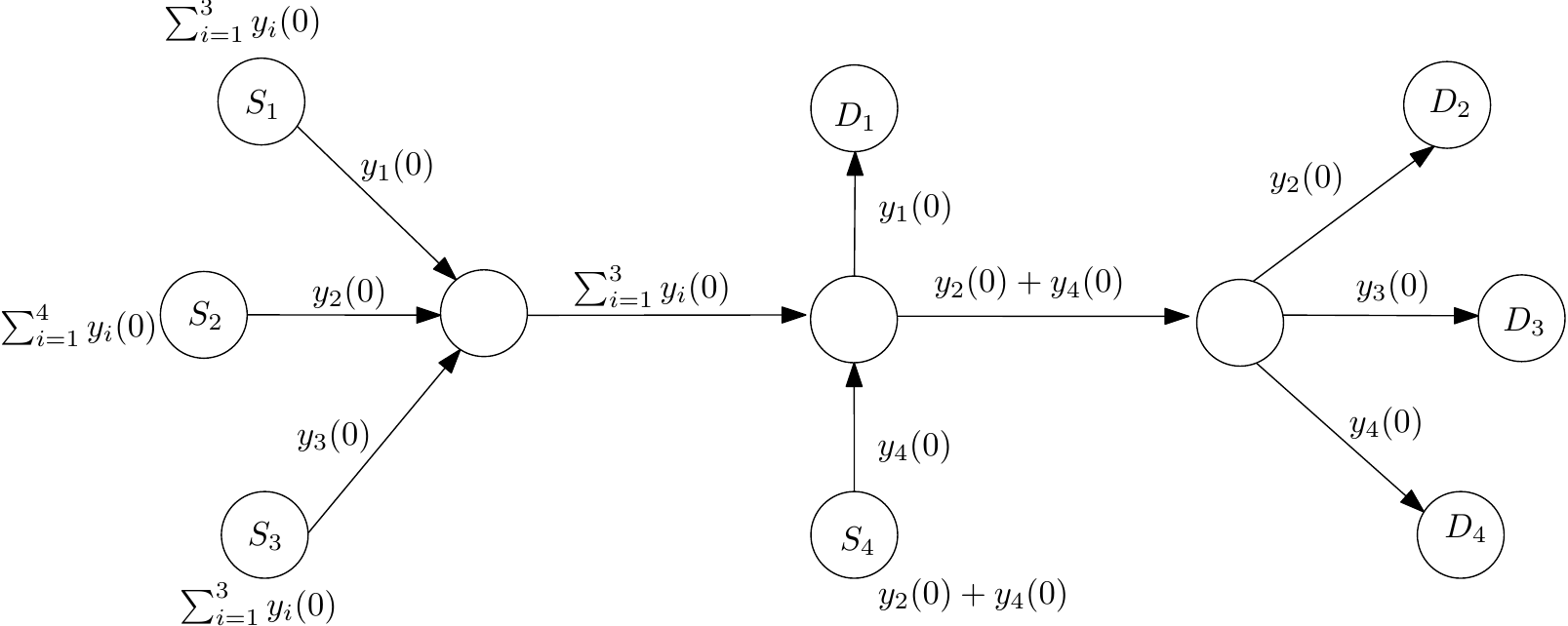}
}\hspace{0.5cm}
\subfloat[State of auxiliary graph $t=1$]{
\includegraphics[scale=0.5]{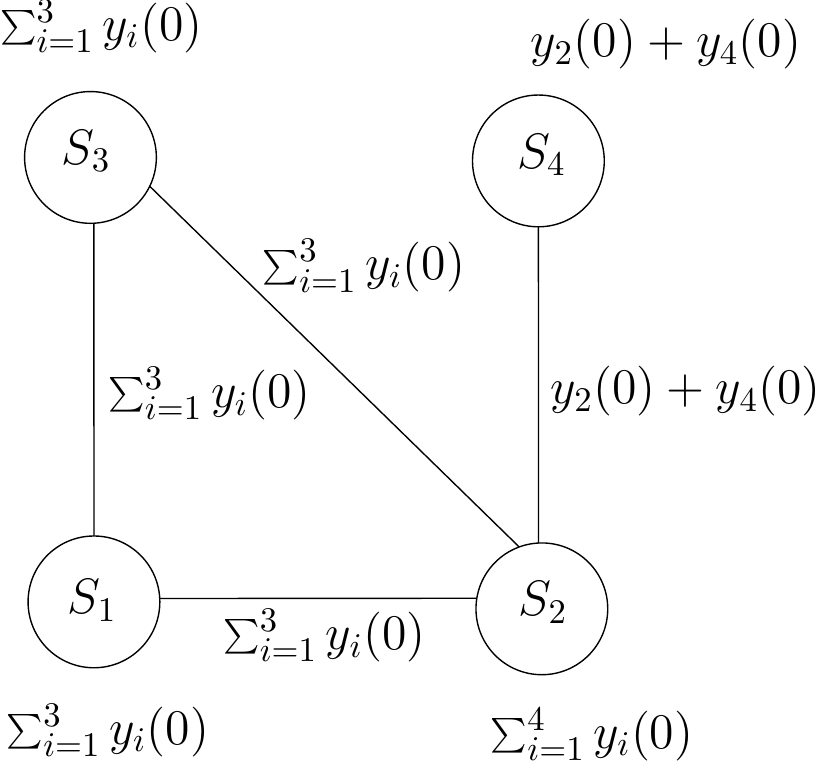}}\hspace{-6cm}\\
\subfloat[State of the network $t=2$]{
\includegraphics[scale=0.6]{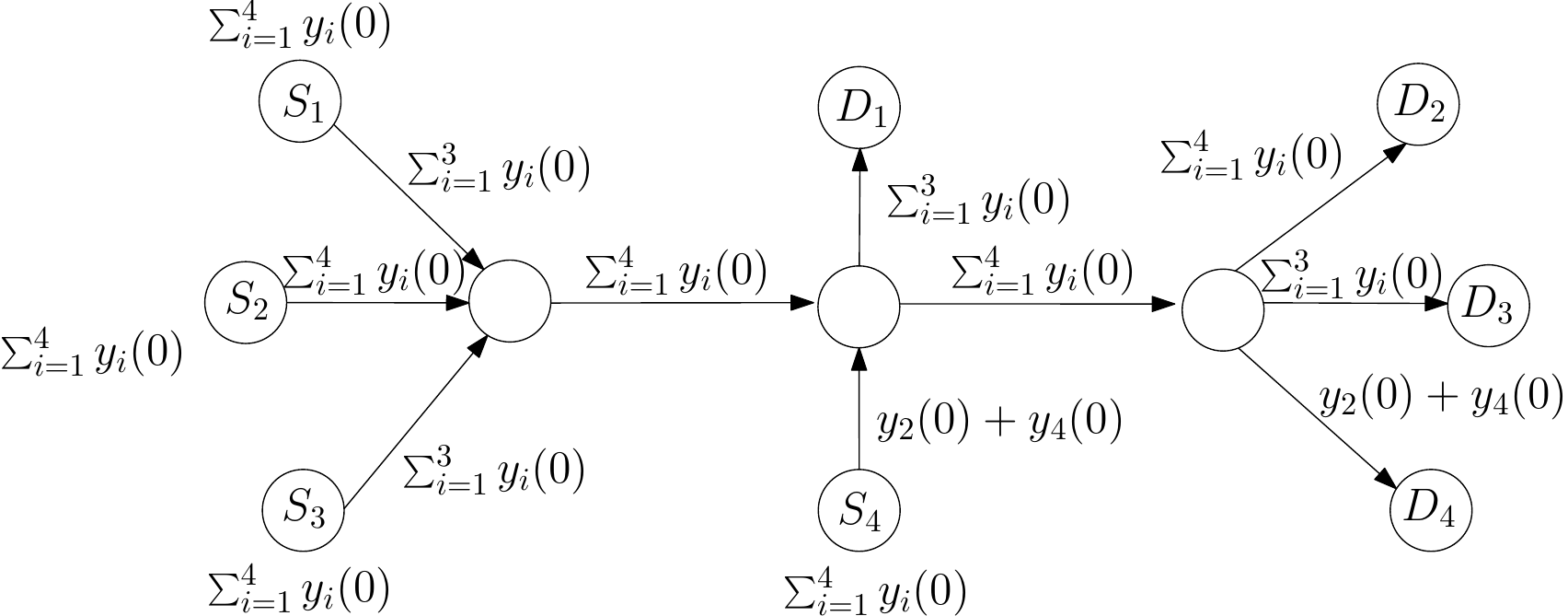}
}\hspace{0.3cm}
\subfloat[State of auxiliary graph $t=2$]{
\includegraphics[scale=0.5]{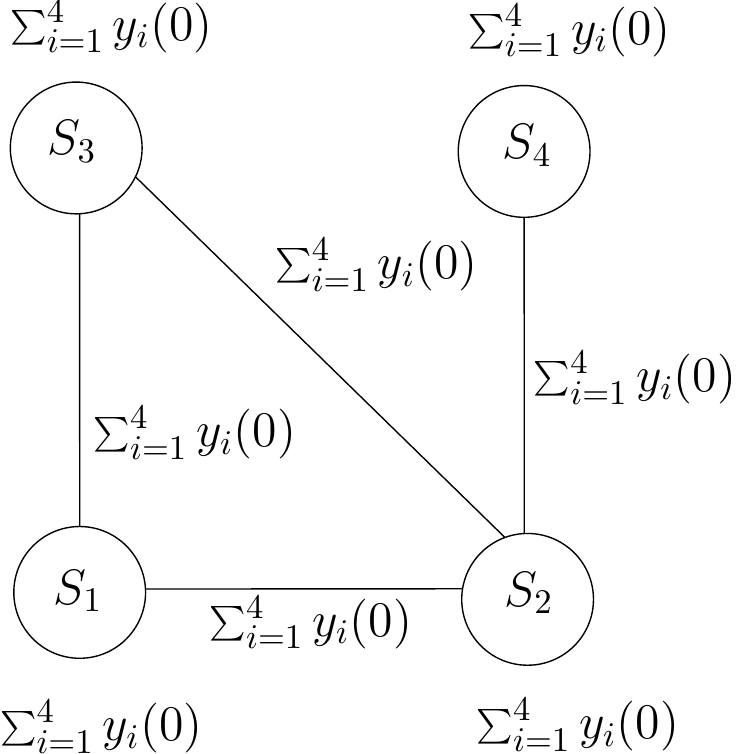}}\hspace{-6cm}\\
\subfloat[State of the network $t=3$]{
\includegraphics[scale=0.6]{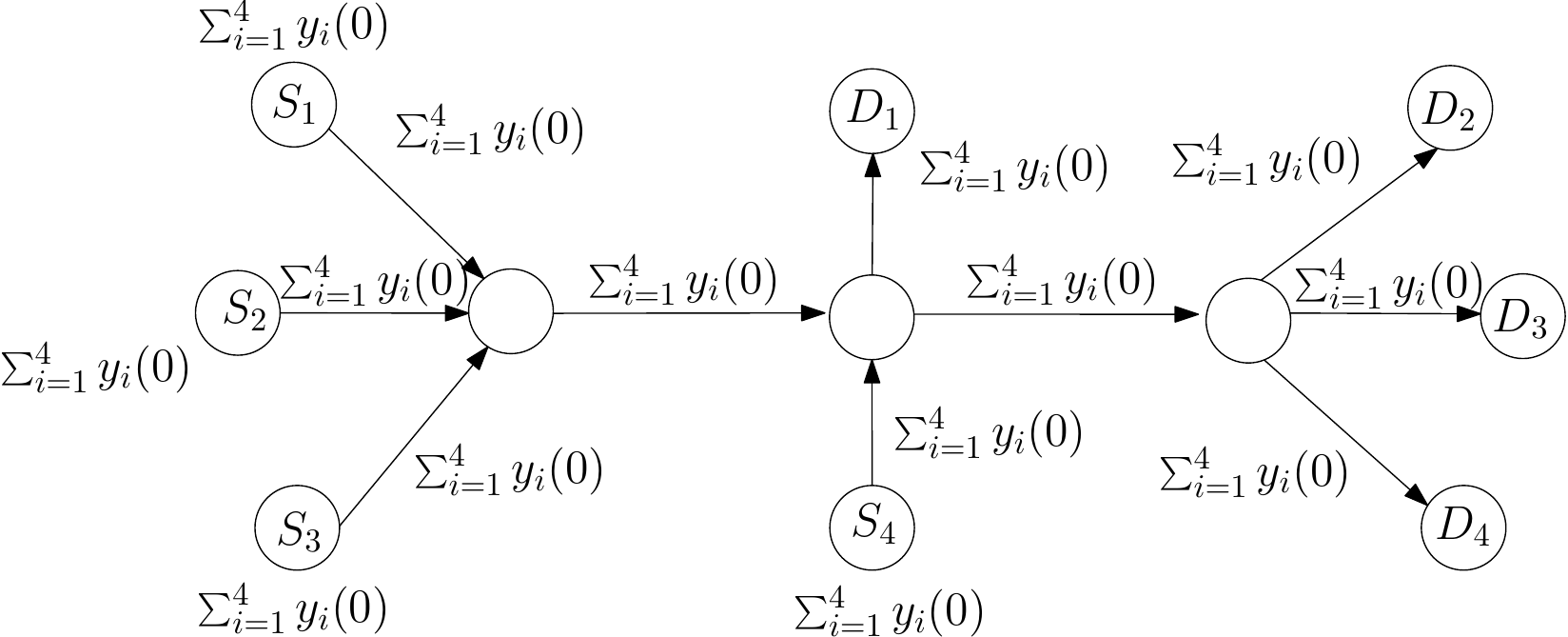}
}\hspace{0.5cm}
\subfloat[State of auxiliary graph $t=3$]{
\includegraphics[scale=0.5]{auxSum2.pdf}}\hspace{-6cm}\caption{Evolution of distributed summation process, where $\rho_i = y_i(0)$ and destination node is indicated using a dot with the same color as its corresponding source.}\label{fig:sumEvolution}
\end{figure}
\begin{figure}
\subfloat[State of the network $t=4$]{
\includegraphics[scale=0.6]{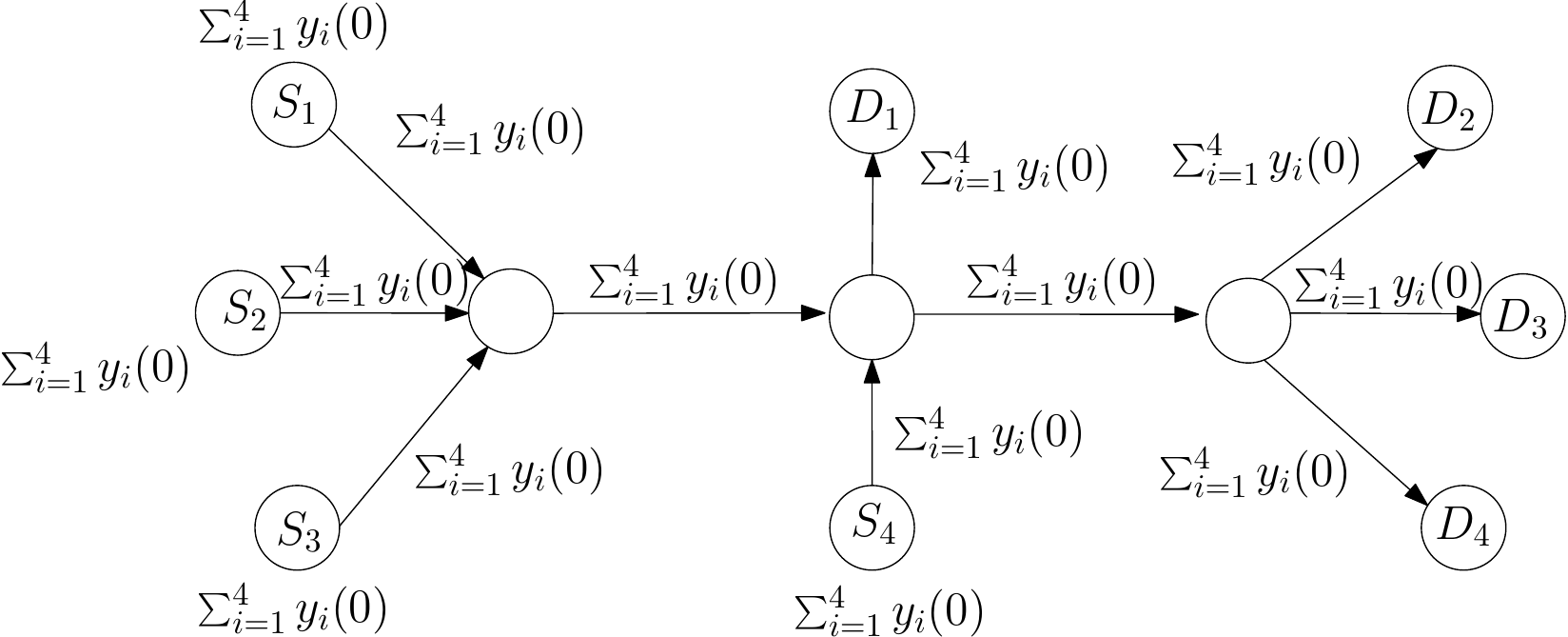}
}\hspace{0.5cm}
\subfloat[State of auxiliary graph $t=4$]{
\includegraphics[scale=0.5]{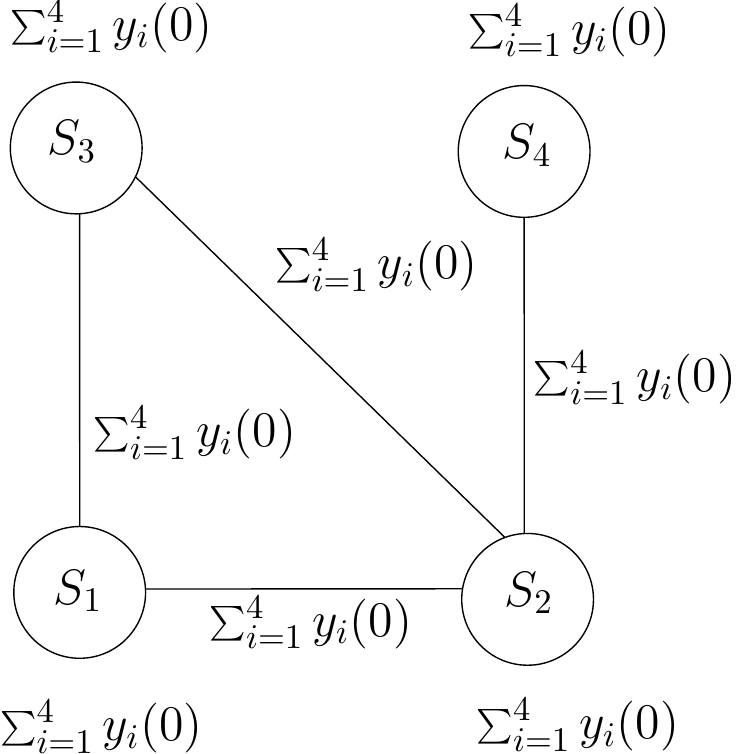}}\hspace{-6cm}\caption{Evolution of distributed summation process continued, where $\rho_i = y_i(0)$ and destination node is indicated using a dot with the same color as its corresponding source.}\label{fig:sumEvolution1}
\end{figure}

By the diagonal structure of the Hessian matrix $H_k$, the scalars $
(\Delta \tilde x_i^k)^2(H_k)_{ii}$ and $(\Delta \tilde
x_{l+S}^k)^2(H_k)_{(l+S)(l+S)}$ are available to the corresponding
source $i$ and link $l$ respectively, hence $z_l^*$  and $y_i^*$ can
be computed using local information. In the above process, each
source only uses aggregate information along its route and each link
$l$ only uses information from sources $i\in S(l)$. The evolution of
the distributed summation procedure for the sample network in Figure
\ref{fig:bigPrimal} is shown in Figures \ref{fig:sumEvolution} and
\ref{fig:sumEvolution1}.

We next establish two lemmas, which quantifies the expansion of the
$t$-hop neighborhood in the auxiliary graph for the links and
sources. This will be key in showing that the aforementioned
summation procedure yields the correct values at the sources and the
links. For each source $i$, we use the notation $ \mathcal{N}_i(t)$
to denote the set of nodes that are connected to node $i$ by a path
of length at most $t$ in the auxiliary graph. Note that
$\mathcal{N}_i(0)=\{i\}$. We say that node $i$ is {\it $t$-hops
away} from node $j$ is the length of the shortest path between nodes
$i$ and $j$ is $t$.

\begin{lemma}\label{lemma:nbd}
Consider a network and its auxiliary graph with sets
$\{\Theta_l\}_{l\in \mathcal{L}}$. For any
 link $l$ and all $t\geq 1$, we have, \be\label{eq:nbdSets} \mathcal{N}_i(t)\cap \mathcal{N}_j(t) =
\cup_{m\in \Theta_l}\mathcal{N}_m(t-1)\qquad \hbox{for }i,j\in
\Theta_l \hbox{ with } i\neq j. \ee
\end{lemma}

\begin{proof}
Since the source nodes $i,j\in \Theta_l$, by part (2) of Lemma
\ref{lemma:nocycle}, they are $1$-hop away from all other nodes in
$\Theta_l$. Hence if a source node $n$ is in $\mathcal{N}_m(t-1)$
for $m\in \Theta_l$, then $n$ is at most $t$-hops away from $i$ or
$j$. This yields \be\label{eq:nbd1}\cup_{m\in
\Theta_l}\mathcal{N}_m(t-1)\subset \mathcal{N}_i(t)\cap
\mathcal{N}_j(t).\ee

On the other hand, if $n\in \mathcal{N}_i(t)\cap \mathcal{N}_j(t)$,
then we have either $n\in \mathcal{N}_i(t-1)$ and hence $n \in
\cup_{m\in \Theta_l}\mathcal{N}_m(t-1)$ or \[n\in
\left(\mathcal{N}_i(t)\backslash \mathcal{N}_i(t-1)\right)\cap
\mathcal{N}_j(t).\] Let $P(a,b)$ denote an ordered set of nodes on
the path between nodes $a$ and $b$ including $b$ but not $a$ for
notational convenience. Then the above relation implies there exists
a path with $|P(i,n)|=t$ and $|P(j,n)|\leq t$. Let $n^* \in
P(i,n)\cap P(j,n)$ and $P(j,n^*)\cap P(j, n^*)=\emptyset$. The node $n^*$ exists, because the two paths both end
at $n$. If $n^*\not\in \Theta_l$, then we have a cycle of
$\{P(i,n^*),P(n^*,j), P(j,i)\}$, which includes an edge with label
$L_l$ between $i$ and $j$ and other edges. In view of part (6) of
Lemma \ref{lemma:nocycle}, this leads to a contradiction. Therefore
we obtain $n^*\in \Theta_l$, implying $P(i,n)=\{P(i,n^*),
P(n^*,n)\}$. Since $i$ is connected to all nodes in $\Theta_l$,
$|P(i,n^*)|=1$ and hence $|P(n^*,n)|=t-1$, which implies $n\in
\mathcal{N}_{n^*}(t-1)\subset\cup_{m\in
\Theta_l}\mathcal{N}_m(t-1)$. Therefore the above analysis yields
\[\mathcal{N}_i(t)\cap \mathcal{N}_j(t)\subset \cup_{m\in \Theta_l}\mathcal{N}_m(t-1).\]
With relation (\ref{eq:nbd1}), this establishes the desired equality.
\end{proof}

\begin{lemma}\label{lemma:nbdS}
Consider a network and its auxiliary graph with sets
$\{\Theta_l\}_{l\in \mathcal{L}}$.  For any source $i$, and all $t\geq 1$, we have,
\be\label{eq:nbdSets}
\left(\cup_{j\in\Theta_l}\mathcal{N}_j(t)\right)\cap\left(\cup_{j\in\Theta_m}\mathcal{N}_j(t)\right)
= \mathcal{N}_i(t) \qquad \hbox{for } l,m \in L^*(i) \hbox{ with } l\neq m.\ee
\end{lemma}
\begin{proof}
Since $l,m \in L^*(i)$, we have $i\in \Theta_l$ and $i\in\Theta_m$, this yields,
\[\mathcal{N}_i(t)\subset\left(\cup_{j\in\Theta_l}\mathcal{N}_j(t)\right)\cap\left(\cup_{j\in\Theta_m}\mathcal{N}_j(t)\right) .\]

On the other hand, assume there exists a node $n$ with $n\in \left(\cup_{j\in\Theta_l}\mathcal{N}_j(t)\right)\cap\left(\cup_{j\in\Theta_m}\mathcal{N}_j(t)\right)$, and $n\not \in \mathcal{N}_i(t)$. Then there exists a node $p\in \Theta_l$ with $p\neq i$ and $n\in\mathcal{N}_p(t)$. Similarly there exists a node $q\in \Theta_m$ with $q\neq i$ and $n\in\mathcal{N}_q(t)$. Let $P(a,b)$ denote an ordered set nodes on the path between nodes $a$ and $b$ including $b$ but not $a$ for notational convenience. Let $n^* \in P(p,n)\cap P(q,n)$ and $P(p,n^*)\cap P(q, n^*)=\emptyset$. The node $n^*$ exists, because the two paths both end at $n$. Since nodes $i, p$ are connected via an edge with label $L_l$ and $i, q$ are connected via an edge with label $L_m$, we have a cycle of $\{P(i,p), P(p,n), P(n,q), P(q,i)\}$, which contradicts part (6) in Lemma \ref{lemma:nocycle} and we have
\[\left(\cup_{j\in\Theta_l}\mathcal{N}_j(t)\right)\cap\left(\cup_{j\in\Theta_m}\mathcal{N}_j(t)\right) \subset \mathcal{N}_i(t).
\]
The preceding two relations establish the desired equivalence.
\end{proof}

Equipped with the preceding lemma, we can now  show that upon
termination of the summation procedure, each source $i$ and link $l$
have $y_i(S) = z_l(S-1) = (\tilde\lambda(x^k))^2$ [cf.\ Eq.\
(\ref{eq:sum})].

\begin{theorem}\label{thm:sum}
Consider a network and its auxiliary graph with sets
$\{\Theta_l\}_{l\in \mathcal{L}}$. Let $\Omega$ denote the set of
all subsets of $S$ and define the function $\sigma: \Omega\to
\mathbb{R}$ as \[\sigma(K) = \sum_{l\in \cup_{i\in K} L(i)}
z_l^*\sum_{i\in K}I_{\{l\in L(i)\}} + \sum_{i\in K} y_i^*, \] where
$y_i^* = (\Delta \tilde x_{i}^k)^2(H_k)_{ii}$, $z_l^*
=\frac{1}{|S(l)|}(\Delta \tilde x_{l+S}^k)^2(H_k)_{(l+S)(l+S)}$ and
$I_{\{l\in L(i)\}}$ is the indicator function for the event $\{l\in
L(i)\}$. Let $y_i(t)$ and $z_l(t)$ be the iterates generated by the
distributed summation procedure described above. Then for all $t \in
\{1, \ldots, S\}$, the value $z_l(t)$  at each link satisfies
\be\label{eq:tHopL} z_l(t) =
\sigma(\cup_{i\in\Theta_l}\mathcal{N}_i(t-1)), \ee and the value
$y_i(t)$ at each source node satisfies \be\label{eq:tHopS} y_i(t) =
\sigma(\mathcal{N}_i(t)).\ee
\end{theorem}

\begin{proof}
We use induction to prove the theorem.
\\ \textbf{Base case: $t=1$.\\}
Since $z_l(0)=0$ for all links, Eq.\ (\ref{eq:zupdate}) for $t=1$ is
\[z_l(1) = \sum_{i\in \Theta_l} y_i(0) = \sum_{i\in \Theta_l}( y_i^*+\sum_{l\in L(i)}z^*_l) = \sigma(\Theta_l),\]
where we use the definition of $y(0)$ [cf.\ Eq.\ (\ref{eq:y0})] and the
 function $\sigma(\cdot)$. Since $\mathcal{N}_i(0)=i$, the above relation implies Eq.\ (\ref{eq:tHopL}) holds.

For source $i$, from update relation (\ref{eq:yupdate}), we have
\[y_i(1) = \sum_{l\in L^*(i)} \sigma(\Theta_l) -\left(|L^*(i)|-1\right) y_i(0).\] Lemma \ref{lemma:nbdS} and inclusion-exclusion principle imply \[\sum_{l\in L^*(i)} \sigma(\Theta_l) = \sigma(\cup_{l\in L^*(i)} \Theta_l) + \left(|L^*(i)|-1\right) \sigma(i).\]
Since $y_i(0) = \sigma(i)$ based on the definition of $y_i(0)$ [cf.\ Eq.\ (\ref{eq:y0})], by rearranging the preceding two relations, we obtain
\[y_i(1) =  \sigma(\cup_{l\in L^*(i)} \Theta_l) = \sigma(\mathcal{N}_i(1)),\]
which shows Eq.\ (\ref{eq:tHopS}) holds for $t=1$.

\noindent\textbf{Inductive step for $t=T\geq 2$.\\} Assume for
$t=T-1$, Eqs.\ (\ref{eq:tHopS}) and (\ref{eq:tHopL}) hold, we first
show that Eq.\ (\ref{eq:tHopL}) hold. When $t=T$, by update equation
(\ref{eq:zupdate}), we obtain for link $l$
\begin{align*}z_l(T) =& \sum_{i\in \Theta_l} y_i(T-1) - \left(|\Theta_l|-1\right)z_l(T-1)\\
 =& \sum_{i\in \Theta_l}\sigma(\mathcal{N}_i(T-1)) -\left(|\Theta_l|-1\right) z_l(T-1), \end{align*}
where the second equality follows from Eq.\ (\ref{eq:tHopS}) for $t=T-1$.

If $|\Theta_l|=1$, then we have $z_l(T) =\sigma(\mathcal{N}_i(T-1))$, for $i\in\Theta_l$, therefore Eq.\ (\ref{eq:tHopL}) is satisfied.

For $|\Theta_l|>1$, using Lemma \ref{lemma:nbd} for $t=T$ and by inclusion-exclusion principle, we obtain
\[ \sum_{i\in \Theta_l}\sigma(\mathcal{N}_i(T-1)) = \sigma\left(\cup_{i\in\Theta_l}\mathcal{N}_i(T-1)\right)+ \left(|\Theta_l|-1\right) \sigma(\cup_{m\in \Theta_l}\mathcal{N}_m(T-2)).
\]
Eq.\ (\ref{eq:tHopL}) for $t=T-1$ yields $z_l(T-1)=\sigma(\cup_{m\in
\Theta_l}\mathcal{N}_m(T-2))$. By using this fact and rearranging
the preceding two relations, we have Eq.\ (\ref{eq:tHopL}) holds for
$t=T$, i.e., \[z_l(T) =
\sigma\left(\cup_{i\in\Theta_l}\mathcal{N}_i(T-1)\right).\]

We next establish Eq.\ (\ref{eq:tHopS}). From update equation
(\ref{eq:yupdate}), using the preceding relation, we have
\begin{align*}
y_i(T) =& \sum_{l\in L^*(i)} z_l(T) -\left(|L^*(i)|-1\right) y_i(T-1)\\
=& \sum_{l\in L^*(i)} \sigma\left(\cup_{i\in\Theta_l}\mathcal{N}_i(T-1)\right) -\left(|L^*(i)|-1\right) y_i(T-1).
\end{align*}

Lemma \ref{lemma:nbdS} and inclusion-exclusion principle imply
\[\sum_{l\in L^*(i)} \sigma\left(\cup_{i\in\Theta_l}\mathcal{N}_i(T-1)\right) = \sigma(\cup_{l\in L^*(i)} \cup_{i\in\Theta_l}\mathcal{N}_i(T-1)) + \left(|L^*(i)|-1\right) \sigma(\mathcal{N}_i(T-1)).\]
By definition of $\mathcal{N}_i(\cdot)$, we have $\cup_{l\in L^*(i)}
\cup_{i\in\Theta_l}\mathcal{N}_i(T-1) = \mathcal{N}_i(T)$. By using
Eq.\ (\ref{eq:tHopS}) for $t=T-1$, i.e.,
$y_i(T-1)=\sigma(\mathcal{N}_i(T-1))$ and rearranging the above two
equations, we obtain
\[ y(T) = \sigma(\mathcal{N}_i(T)),
\]
which completes the inductive step.\end{proof}

Using definition of the function $\sigma(\cdot)$, we have
$(\tilde\lambda(x^k))^2 = \sigma(\mathcal{S})$. By the above
theorem, we conclude that after $S$ iterations, \[y_i(S) =
\sigma(\mathcal{N}_i(S))=\sigma(\mathcal{S})=(\tilde\lambda(x^k))^2.\]
By observing that $\cup_{i\in\Theta_l}\mathcal{N}_i(S-1) =
\mathcal{S}$, we also have \[z_i(S-1) =
\sigma\left(\cup_{i\in\Theta_l}\mathcal{N}_i(S-1)\right) =
\sigma(\mathcal{S}))=(\tilde\lambda(x^k))^2,\] where we used part
(5) of Lemma \ref{lemma:nocycle}. This shows that the value
$\tilde\lambda (x^k)^2$ is available to all sources and links after
$S-1$ iterations.

Note that the number $S$ is an upper bound on the number of
iterations required in the distributed summation process to obtain
the correct value at the links and sources in the original graph. If
the value of the diameter of the auxiliary graph (or an upper bound
on it) is known, then the process would terminate in number of
steps equal to this value plus 1. For instance, when all the sources share
one common link, then the auxiliary graph is a complete graph, and only
$2$ iterations is required. On the other hand, when the auxiliary
graph is a line, the summation procedure would take $S$ iterations.

We finally contrast our distributed summation procedure with
spanning tree computations, which were used widely in 1970s and
1980s for performing information exchange among different processors
in network flow problems. In spanning tree based approaches,
information from all processors is passed along the edges of a
spanning tree, and stored at and broadcast by a designated central
root node (see \cite{Klincewicz83} and \cite{spanningTree}). In
contrast, our summation procedure involves (scalar) information
aggregated along the routes and fed back independently to different
sources, which is a more natural exchange mechanism in an
environment with decentralized sources. Moreover, processors in the
system (i.e., sources and links) do not need to maintain
predecessor/successor information (as required by spanning tree
methods). The only network-related information is the sets
$\theta_l$ for $l\in \mathcal{L}$ kept at the individual links and
obtained from the auxiliary graph, which is itself constructed using
the feedback mechanism described above.

\section{Distributed Error Checking}\label{app:errorbounds}

In this section, we present a distributed error checking method to
determine when to terminate the dual computation procedure to meet
the error tolerance level in Assumption \ref{ass:errorBoundEps} at a
primal iteration $k$. The method involves two stages: in the first
stage, the links and sources execute a predetermined number of dual
iterations. In the second stage, if  the error tolerance level is
not satisfied in the previous stage, the links and sources implement
dual iterations until some distributed termination criteria is met.
For the rest of this section we suppress the dependence of the dual
vector on the primal iteration index $k$ for notational convenience
and we adopt the following assumption on the information available
to each node and link.

\begin{assumption}\label{ass:spectralGap}
There exists a positive scalar $F<1$ such that the spectral radius
of the matrix $M=(D_k+\bar{B}_k)^{-1}(\bar{B}_k-B_k)$ satisfies
$\rho(M)\leq F$. Each source and link knows the scalar $F$ and the
total number of sources and links in the graph, denoted by $S$ and
$L$ respectively.
\end{assumption}

As noted in Section \ref{sec:dual}, the matrix
$M=(D_k+\bar{B}_k)^{-1}(\bar{B}_k-B_k)$ is the weighted Laplacian
matrix of the dual graph, therefore the bound $F$ can be obtained
once the structure of the dual graph is known \cite{ChungGraphBook},
\cite{graphTheory}. In this assumption, we only require availability
of some aggregate information, and hence the distributed nature of
the algorithm is preserved. Before we introduce the details of the
algorithm, we establish a relation between $\norm{w^*-w(t)}_\infty$
and $\norm{w(t+1)-w(t)}_\infty$, which is a key relation in
developing the distributed error checking method.

\begin{lemma}
Let the matrix $M$ be $M=(D_k+\bar{B}_k)^{-1}(\bar{B}_k-B_k)$. Let
$w(t)$ denote the dual variable generated by iteration
(\ref{eq:dualInnerIteration}), and $w^*$ be the fixed point of the
iteration. Let $F$ and $L$ be the positive scalar defined in
Assumption \ref{ass:spectralGap}. Then the following relation holds,
\be\label{ineq:wNorms}\norm{w^*-w(t)}_\infty\leq\frac{\sqrt
{L}}{1-F}\norm{w(t+1)-w(t)}_\infty.\ee
\end{lemma}
\begin{proof}

Iteration (\ref{eq:dualInnerIteration}) implies that the fixed point $w^*$ satisfies the following relation,
\[ w^*=Mw^*+(D_k+\bar{B}_k)^{-1}(-AH_k^{-1}\nabla f(x^k)),\]
and the iterates $w(t)$ satisfy,
\[ w(t+1)=Mw(t)+(D_k+\bar{B}_k)^{-1}(-AH_k^{-1}\nabla f(x^k)).\]
By combining the above two relations, we obtain,
\[w(t+1)-w(t) = (I-M)(w^*-w(t)).\]
Hence, by the definition of matrix infinity norm, we have
\[\norm{w^*-w(t)}_\infty\leq\norm{(1-M)^{-1}}_\infty\norm{w(t+1)-w(t)}_\infty.\]
Using norm equivalence for finite dimensional Euclidean space and theories of linear algebra, we obtain $\norm{(I-M)^{-1}}_\infty\leq \sqrt{L}\norm{(I-M)^{-1}}_2\leq \frac{\sqrt L}{1-\norm{M}_2}$ \cite{Horn}, \cite{Berman}. For the symmetric real matrix $M$ we have $\rho(M)=\norm{M}_2$, and hence we obtain the desired relation.
\end{proof}

We next use the above lemma to develop two theorems each of which
serves as a starting point for one of the stages of the distributed
error checking method.

\begin{theorem}\label{thm:1ndCheck}
Let $\{x^k\}$ be the primal sequence generated by the inexact Newton
method (\ref{eq:inexaAlgo}) and $H_k$ be the corresponding Hessian
matrix at the $k^{th}$ iteration. Let $w(t)$ be the inexact dual
variable obtained after $t$ dual iterations
(\ref{eq:dualInnerIteration}) and $w^*$ be the exact solution to
(\ref{eq:dualUpdateBackground}), i.e. the limit of the sequence
$\{w(t)\}$ as $t\to\infty$. Let vectors $\Delta x^k$ and
$\Delta\tilde x^k$ be the exact and inexact Newton directions
obtained using $w^*$ and $w(t)$ [cf.\ Eqs.\
(\ref{eq:ds})-(\ref{eq:dxPrimal})], and vector $\gamma^k$ be the
error in the Newton direction computation at $x^k$ , defined by
$\gamma^k = \Delta x^k - \Delta \tilde x^k$. For some positive
scalar $p$, let \be\label{eq:1stErrorS}\rho_i =
\left|\frac{\sqrt{L}(H_k^{-1})_{ii}|L(i)|\norm{w(t+1)-w(t)}_\infty}{(1-F)(H_k^{-1})_{ii}[R']^iw(t)}
\right|\ee for each source $i$, and
\be\label{eq:1stErrorL}\rho_l=\left|\frac{\sqrt{L}\sum_{i\in S(l)}
(H_k^{-1})_{ii}|L(i)|\norm{w(t+1)-w(t)}_\infty}{(1-F)\sum_{i\in
S(l)}(H_k^{-1})_{ii}[R']^iw(t)}\right|\ee for each link $l$. Define
a nonnegative scalar $\beta^k$ as \be\label{eq:1stb}\beta^k =
\left(\max\left\{\max_{i\in \mathcal{S}}\frac{\rho_i}{p} ,
\max_{l\in \mathcal{L}} \frac{\rho_l}{p} \right\}\right)^{-2}.\ee
Then we have \be\label{ineq:pbeta} \beta^k(\gamma^k)'H_k\gamma^k\leq
p^2 (\Delta \tilde x^k)'H_k(\Delta \tilde x^k). \ee
\end{theorem}
\begin{proof}
For notational convenience, we let matrix $P_k$ denote the $S\times S$ principal submatrix of $H_k$, i.e. $(P_k)_{ii} = (H_k)_{ii}$ for $i\leq S$, vector $\Delta \tilde s$ in $\mathbb{R}^S$ denote the first $S$ components of the vector $\Delta \tilde x^k$, vector $\Delta \tilde y_l$ in $\mathbb{R}^L$ denote the last $L$ components of the vector $\Delta \tilde x^k$. Similarly, we denote by $\Delta s$ and $\Delta y$ the first $S$ and last $L$ components of the exact Newton direction $\Delta x$ respectively. From Eq.\ (\ref{eq:ds}), we have for each $i\in S$,
\begin{align*}
\left|\frac{\Delta s_i-\Delta\tilde s_i}{\Delta \tilde s_i} \right| &=\left| \frac{(H_k^{-1})_{ii}[R']^i(w^*-w(t))}{(H_k^{-1})_{ii}[R']^i w(t)} \right|\\ \nonumber
&\leq \left|\frac{(H_k^{-1})_{ii}[R']^ie\norm{w^*-w(t)}_\infty}{(H_k^{-1})_{ii}[R']^i w(t)}  \right|\\ \nonumber
& \leq \left|\frac{\sqrt{L}(H_k^{-1})_{ii}|L(i)|\norm{w(t+1)-w(t)}_\infty}{(1-F)(H_k^{-1})_{ii}[R']^iw(t)}  \right| = \rho_i,
\end{align*}
where the first inequality follows from the element-wise nonnegativity of matrices $H_k$ and $R$, and the second inequality follows from relation (\ref{ineq:wNorms}).

Similarly for each link $l\in L$, by relations (\ref{eq:ds}) and (\ref{eq:dxPrimal}) we obtain
\begin{align*}
\left|\frac{\Delta y_l-\Delta\tilde y_l}{\Delta \tilde y_l}  \right|& =
\left| \frac{[R]^{l}P_k^{-1}R'(w^*-w(t))}{[R]^{l}P_k^{-1}R' w(t)}  \right|\\ \nonumber
&\leq \left|\frac{\sum_{i\in S(l)}(P_k^{-1})_{ii}R' e \norm{w^*-w(t)}_\infty}{\sum_{i\in S(l)}(P_k^{-1})_{ii}[R']^i w(t)}\right|\\ \nonumber
&\leq \left|\frac{\sqrt{L}\sum_{i\in S(l)} (H_k^{-1})_{ii}|L(i)|\norm{w(t+1)-w(t)}_\infty}{(1-F)\sum_{i\in S(l)}(H_k^{-1})_{ii}[R']^iw(t)}\right|  =\rho_l,
\end{align*}
where the first inequality follows from the structure of the matrix $R$ and the element-wise nonnegativity of matrices $H_k$ and $R$, and the second inequality follows from relation (\ref{ineq:wNorms}) and the definition for matrix $P_k$.

The definition for $\beta^k$ [cf.\ Eq.\ (\ref{eq:1stb})] implies that
\[\frac{p}{\sqrt{\beta^k}} = \max\left\{\max_{i\in \mathcal{S}}\rho_i , \max_{l\in \mathcal{L}} \rho_l \right\}.\]
Therefore the preceding relations imply that $\left|\frac{\Delta s_i-\Delta\tilde s_i}{\Delta \tilde s_i} \right|\leq \frac{p}{\sqrt {\beta^k}}$ and $\left|\frac{\Delta y_l-\Delta\tilde y_l}{\Delta \tilde y_l}  \right|\leq \frac{p}{\sqrt {\beta^k}}$, i.e., \[ \sqrt {\beta^k}\left|\gamma^k_i \right|\leq p |{\Delta \tilde x_i}|,\]
which implies the desired relation.

\end{proof}

\begin{theorem}\label{thm:2ndCheck}
Let $\{x^k\}$ be the primal sequence generated by the inexact Newton method (\ref{eq:inexaAlgo}) and $H_k$ be the corresponding Hessian matrix at $k^{th}$ iteration. Let $w(t)$ be the inexact dual variable obtained after $t$ dual iterations (\ref{eq:dualInnerIteration}) and $w^*$ be the exact solution to (\ref{eq:dualUpdateBackground}), i.e. the limit of the sequence $\{w(t)\}$ as $t\to\infty$. Let vectors $\Delta x^k$ and $\Delta\tilde x^k$ be the exact and inexact Newton directions obtained using $w^*$ and $w(t)$ [cf.\ Eqs.\ (\ref{eq:ds})-(\ref{eq:dxPrimal})] respectively, and vector $\gamma^k$ be the error in the Newton direction computation at $x^k$ , defined by $\gamma^k = \Delta x^k - \Delta \tilde x^k$. For some scalar $\beta$ and $\epsilon$ where $0<\beta^k<1$ and $\epsilon>0$, let
\be\label{eq:2ndErrorS}h_i = \sqrt{\frac{\epsilon}{(1-\beta^k)(L+S)L}}\frac{1-F}{|L(i)|(H_k^{-\frac{1}{2}})_{ii}}\ee
for each source $i$, and  \be\label{eq:2ndErrorL}
h_l=\sqrt{\frac{\epsilon}{(1-\beta^k)(L+S)L}}\frac{1-F}{(H_k^{\frac{1}{2}})_{(S+l)(S+l)} \sum_{i\in S(L)} |L(i)|(H_k)_{ii}^{-1}}\ee
for each link $l$. Define a nonnegative scalar $h$ as
\be\label{eq:2ndh}h = \left(\min\left\{\min_{i\in \mathcal{S}} h_i, \min_{l\in \mathcal{L}} h_l\right\}\right).\ee
Then the condition \be\label{ineq:terminate}\norm{w(t+1)-w(t)}_\infty\leq h\ee implies
\be\label{ineq:epsi}
(\gamma^k)'H_k\gamma^k \leq \frac{\epsilon}{1-\beta^k}.
\ee
\end{theorem}
\begin{proof}
We let matrix $P_k$ denote the $S\times S$ principal submatrix of $H_k$, i.e. $(P_k)_{ii} = (H_k)_{ii}$ for $i\leq S$, for notational convenience.
The definition of $h$ [cf.\ Eq.\ (\ref{eq:2ndh})] and relation (\ref{ineq:terminate}) implies
\[
\norm{w(t+1)-w(t)}_\infty\leq h_i, \in \mathcal{S},
\]
and
\[
\norm{w(t+1)-w(t)}_\infty\leq h_l, \mbox{\quad for l }\in \mathcal{L},
\]
Using relation (\ref{ineq:wNorms}) and the definition of $h_i$ and
$h_l$ [cf.\ Eqs.\ (\ref{eq:2ndErrorS}) and (\ref{eq:2ndErrorL})],
the above two relations implies respectively that
\be\label{ineq:sourceError} \norm{w^*-w(t)}_\infty\leq
\sqrt{\frac{\epsilon}{(1-\beta^k)(L+S)}}\frac{1}{|L(i)|(H_k^{-\frac{1}{2}})_{ii}},
\mbox{\quad for i } \in \mathcal{S}, \ee and
\be\label{ineq:linkError} \norm{w^*-w(t)}_\infty\leq
\sqrt{\frac{\epsilon}{(1-\beta^k)(L+S)}}\frac{1}{(H_k^{\frac{1}{2}})_{(S+l)(S+l)}\sum_{i\in
S(L)} |L(i)|(H_k)_{ii}^{-1}}, \mbox{\quad for l }\in \mathcal{L}.
\ee By using the element-wise nonnegativity of matrices $H$ and $A$,
we have for each source $i$,
\[ \left|(H_k^{-\frac{1}{2}})_{ii}[R']^i(w^*-w(t))\right|\leq (H_k^{-\frac{1}{2}})_{ii}[R']^i e
\norm{w^*-w(t)}_\infty = |L(i)|(H_k^{-\frac{1}{2}})_{ii}\norm{w^*-w(t)}_\infty,\]
where the last equality follows from the fact that $[R']^i e = |L(i)|$ for each source $i$.

The above inequality and relation (\ref{ineq:sourceError}) imply
\be\label{ineq:boundS}
\left|(H_k^{-\frac{1}{2}})_{ii}[R']^i(w^*-w(t))\right|\leq
\sqrt{\frac{\epsilon}{(1-\beta^k)(L+S)}}.\ee By the definition of
matrices $P_k$ and $R$, we have for each link $l$,
\begin{align*}
\left|\right. (H_k^{\frac{1}{2}})_{(S+l)(S+l)}\left(\right.RP_k^{-1}R'(w^*-&w(t))\left.\right)^{l}\left.\right| \leq (H_k^{\frac{1}{2}})_{(S+l)(S+l)}[R]^{l}P_k^{-1}R' e \norm{w^*-w(t)}_\infty\\
&=(H_k^{\frac{1}{2}})_{(S+l)(S+l)}\sum_{i\in S(l)} |L(i)|(H_k^{-1})_{ii} \norm{w^*-w(t)}_\infty.
\end{align*}
When combined with relation (\ref{ineq:linkError}), the preceding
relation yields \be\label{ineq:boundL} \left|
(H_k^{\frac{1}{2}})_{(S+l)(S+l)}\left(R
P_k^{-1}R'(w^*-w(t))\right)^{l}\right|\leq
\sqrt{\frac{\epsilon}{(1-\beta^k)(L+S)}}. \ee

From Eqs.\ (\ref{eq:ds})-(\ref{eq:dxPrimal}) and the definition of $\gamma$, we have
\[\gamma^k_i = -\left(\begin{array}{c}P_k^{-1}R'(w^*-w(t))\\RP_k^{-1}R'(w^*-w(t))\end{array}\right),\]
which implies that
\als{(\gamma^k)'H_k\gamma^k&= \sum_{i\in \mathcal{S}}\left((H_k^{-\frac{1}{2}})_{ii}[R']^i(w^*-w(t))\right)^2 + \sum_{l\in\mathcal{L}}\left( (H_k^{\frac{1}{2}})_{(S+l)(S+l)}\left(R P_k^{-1}R'(w^*-w(t))\right)^{l}\right)^2\\&\leq \frac{\epsilon}{1-\beta^k},}
where the inequality follows from (\ref{ineq:boundS}), (\ref{ineq:boundL}), which establishes the desired relation.

\end{proof}

We develop the distributed error checking method based on the
preceding two theorems:
\begin{itemize}
\item{Stage 1:} The links and sources implement $T$ iterations of (\ref{eq:dualInnerIteration}),
where $T$ is a predetermined globally known constant. The links and
sources then use Theorem \ref{thm:1ndCheck} with $t=T-1$ and $p$ as
the desired relative error tolerance level defined in Assumption
\ref{ass:errorBoundEps} to obtain a value $\beta^k$. If $\beta^k\geq
1$, then the dual iteration terminates.

\item{Stage 2:} The links and sources use Theorem \ref{thm:2ndCheck} with
$\beta^k$ obtained in the first stage and $\epsilon$ defined in
Assumption \ref{ass:errorBoundEps} to obtain value $h$. Then they
perform more iterations of the form (\ref{eq:dualInnerIteration})
until the criterion (\ref{ineq:terminate}) is
satisfied.\footnote{The error tolerance level will terminate after
finite number of iterations for any $h>0$, due to the convergence of
the sequence $w(t)$ established in Section \ref{sec:ConvDual1}.}
\end{itemize}

Stage 1 corresponds to checking the term $ p^2 (\Delta \tilde x^k)'H_k(\Delta \tilde x^k)$, while Stage 2 corresponds to the term $\epsilon$ in the error tolerance level. If the method terminates the dual iterations in Stage
1, then Assumption \ref{ass:errorBoundEps} is satisfied for any
$\epsilon>0$; otherwise, by combining relations (\ref{ineq:pbeta})
and (\ref{ineq:epsi}), we have
\[
(\gamma^k)'H_k\gamma^k= (\beta^k+(1-\beta^k))(\gamma^k)'H_k\gamma^k\leq p^2 (\Delta \tilde x^k)'H_k(\Delta \tilde x^k)+\epsilon,
\]
which shows that the error tolerance level in Assumption
\ref{ass:errorBoundEps} is satisfied.

To show that the above method can be implemented in a distributed way, we first rewrite the terms $\rho_i$, $\rho_l$, $h_i$ and $h_l$ and analyze the information required to compute them in a decentralized way. We use the definition of the weighted price of the route $\Pi_i(t)$ and obtain $\Pi_i(t) = (H_k^{-1})_{ii}\sum_{l\in L(i)}w(t)= (H_k^{-1})_{ii}[R']^iw(t)$ and $\Pi_i(0) = (H_k^{-1})_{ii}|L(i)|$, where $w_l(0)=1$ for all links $l$. Therefore relations (\ref{eq:1stErrorS}) and (\ref{eq:1stErrorL}) can be rewritten as
\[\rho_i =  \left|\frac{\sqrt{L}\Pi_i(0)\norm{w(t+1)-w(t)}_\infty}{(1-F)\Pi_i(t)}  \right|,\]
\[\rho_l=\left|\frac{\sqrt{L}\sum_{i\in S(l)} \Pi_i(0)\norm{w(t+1)-w(t)}_\infty}{(1-F)\sum_{i\in S(l)}\Pi_i(t)}\right|.\]

Similarly, relations (\ref{eq:2ndErrorS}) and (\ref{eq:2ndErrorL}) can be transformed into
\[h_i = \sqrt{\frac{\epsilon}{(1-\beta^k)(L+S)L}}\frac{1-F}{\pi_i(0)(H_k^{-\frac{1}{2}})_{ii}},\]
\[h_l=\sqrt{\frac{\epsilon}{(1-\beta^k)(L+S)L}}\frac{1-F}{(H_k^{\frac{1}{2}})_{(S+l)(S+l)} \sum_{i\in S(L)} \Pi_i(0)}.\]
In our dual variable computation procedure, the values $\pi_i(0)$,
$\Pi_i(0)$ and $\Pi_i(t)$ are made available to all the links source
$i$ traverses through the feedback mechanism described in Section
\ref{sec:dual}. Each
source and node knows its local Hessian, i.e., $(H_k)_{ii}$ for
source $i$ and $(H_k)_{(S+l)(S+l)}$ for link $l$. The value
$\beta^k$ is available from the previous stage. Therefore in the
above four expressions, the only not immediately available
information is  $\norm{w(t+1)-w(t)}_\infty$, which can be obtained
using a maximum consensus algorithm.\footnote{In a maximum consensus
algorithm, each node starts with some state and updates its current
state with the maximum state value in its neighborhood (including
itself). Therefore after one round of algorithm, the neighborhood of
the node with maximal value has now the maximum value, after the
diameter of the graph rounds of algorithm, the entire graph reaches
a consensus on the maximum state value and the algorithm
terminates.} Based on these four terms, the values of $\beta^k$ and
$h$ can be obtained using once again maximum consensus and hence all
the components necessary for the error checking method can be
computed in a distributed way.

We observe that in the first $T$ iterations, i.e., Stage 1, only two executions of maximum consensus
algorithms is required, where one is used to compute
$\norm{w(t+1)-w(t)}_\infty$ and the other for $\beta^k$. On the other hand, even
though the computation of the value $h$ in Stage 2 needs only one
execution of the maximum consensus algorithm, the term
$\norm{w(t+1)-w(t)}_\infty$ needs to be computed at each dual
iteration $t$. Therefore the error checking in Stage 1 can be completed much more efficiently than in Stage 2. Hence, when we design values $p$ and $\epsilon$ in
Assumption \ref{ass:errorBoundEps}, we should choose $p$ to be relatively large, which results in an error checking method that does not enter Stage 2 frequently, and is hence faster.

\bibliographystyle{plain}
\bibliography{citationJournal}

\newcommand{\noopsort}[1]{} \newcommand{\printfirst}[2]{#1}
  \newcommand{\singleletter}[1]{#1} \newcommand{\switchargs}[2]{#2#1}
\begin{thebibliography}{10}

\bibitem{Low}
S.~Athuraliya and S.~Low.
\newblock {Optimization flow control with Newton-like algorithm}.
\newblock {\em Journal of Telecommunication Systems}, 15:345--358, 2000.

\bibitem{Berman}
A.~Berman and R.~J. Plemmons.
\newblock {\em {Nonnegative Matrices in the Mathematical Sciences}}.
\newblock Academic Press, New York, 1979.

\bibitem{ConvexBertsekas}
D.~P. Bertsekas.
\newblock {\em {Convex Optimization Theory}}.
\newblock Athena Scientific, 2009.

\bibitem{BetGaf83}
D.~P. Bertsekas and E.~M. Gafni.
\newblock {Projected Newton Methods and Optimization of Multicommodity Flows}.
\newblock {\em IEEE Transactions on Automatic Control}, 28(12), 1983.

\bibitem{Asubook}
D.~P. Bertsekas, A.~Nedic, and A.~E. Ozdaglar.
\newblock {\em {Convex Analysis and Optimization}}.
\newblock Athena Scientific, Cambridge, MA, 2003.

\bibitem{BoydBP}
D.~Bickson, Y.~Tock, A.~Zymnis, S.~Boyd, and D.~Dolev.
\newblock Distributed large scale network utility maximization.
\newblock {\em Proceedings of the 2009 IEEE International Conference on
  Symposium on Information Theory}, 2, 2009.

\bibitem{graphTheory}
N.~Biggs.
\newblock {\em {Algebraic Graph Theory}}.
\newblock Cambridge University Press, second edition, 1993.

\bibitem{Blondel}
V.~D. Blondel, J.~M. Hendrickx, A.~Olshevsky, and J.~N. Tsitsiklis.
\newblock Convergence in multiagent coordination, consensus, and flocking.
\newblock {\em Proceedings of IEEE CDC}, 2005.

\bibitem{Boyd}
S.~Boyd and L.~Vandenberghe.
\newblock {\em {Convex Optimization}}.
\newblock Cambridge University Press, 2004.

\bibitem{spanningTree}
G.~H. Bradley, G.~G. Brown, and G.~W. Graves.
\newblock {Design and implementation of large scale primal transshipment
  algorithms}.
\newblock {\em Management Science}, 24(1):1--34, 1977.

\bibitem{mung}
M.~Chiang, S.~H. Low, A.~R. Calderbank, and J.C. Doyle.
\newblock Layering as optimization decomposition: a mathematical theory of
  network architectures.
\newblock {\em Proceedings of the IEEE}, 95(1):255--312, 2007.

\bibitem{ChungGraphBook}
F.~R.~K. Chung.
\newblock {\em {Spectral Graph Theory (CBMS Regional Conference Series in
  Mathematics)}}.
\newblock No. 92, American Mathematical Society, 1997.

\bibitem{Cottle}
R.~Cottle, J.~Pang, and R.~Stone.
\newblock {\em {The Linear Complementarity Problem}}.
\newblock Academic Press, 1992.

\bibitem{dembo}
R.~Dembo, S.~Eisenstat, and T.~Steihaug.
\newblock {Inexact Newton Methods}.
\newblock {\em SIAM Journal on Numerical Analysis}, 19, 1982.

\bibitem{oldInt}
A.~V. Fiacco and G.~P. McCormick.
\newblock {\em {Nonlinear Programming: Sequential Unconstrained Minimization
  Techniques}}.
\newblock SIAM, 1990.

\bibitem{Horn}
R.~Horn and C.~R. Johnson.
\newblock {\em {Matrix Analysis}}.
\newblock Cambridge University Press, New York, 1985.

\bibitem{ali}
A.~Jadbabaie, J.~Lin, and S.~Morse.
\newblock Coordination of groups of mobile autonomous agents using nearest
  neighbor rules.
\newblock {\em IEEE Transactions on Automatic Control}, 48(6):988--1001, 2003.

\bibitem{AsuNewton}
A.~Jadbabaie, A.~Ozdaglar, and M.~Zargham.
\newblock {A Distributed Newton method for network optimization}.
\newblock {\em Proc. of CDC}, 2009.

\bibitem{Jarre}
F.~Jarre.
\newblock {Interior-point methods for convex programming}.
\newblock {\em Applied Mathematics and Optimization}, 26:287--311, 1992.

\bibitem{inexactKelley}
C.~T. Kelley.
\newblock {\em {Iterative Methods for Linear and Nonlinear Equations}}.
\newblock SIAM, Philadelphia, PA, 1995.

\bibitem{Kelly}
F.~Kelly.
\newblock {Charging and rate control for elastic traffic}.
\newblock {\em European Transactions on Telecommunications}, 8:33--37, 1997.

\bibitem{KellyMaTan}
F.~P. Kelly, A.~K. Maulloo, and D.~K. Tan.
\newblock {Rate control for communication networks: shadow prices, proportional
  fairness, and stability}.
\newblock {\em Journal of the Operational Research Society}, 49:237--252, 1998.

\bibitem{Klincewicz83}
J.~G. Klincewicz.
\newblock {A Newton Method for Convex Separable Network Flow Problems}.
\newblock {\em Networks}, 13:427--442, 1983.

\bibitem{Lay}
D.~C. Lay.
\newblock {\em {Linear Algebra and Its Applications}}.
\newblock Person Education, third edition, 2006.

\bibitem{LowLapsley}
S.~H. Low and D.~E. Lapsley.
\newblock {Optimization flow control, I: basic algorithm and convergence}.
\newblock {\em IEEE/ACM Transaction on Networking}, 7(6):861--874, 1999.

\bibitem{graphLaplacian}
B.~Mohar.
\newblock {Some Applications of Laplace Eigenvalues of Graphs}.
\newblock {\em In: Hahn, G. and Sabidussi, G. (Eds.) Graph Symmetry: Algebraic
  Methods and Applications}, NATO ASI Series C 497:227--275, 1997.

\bibitem{AsuChapter}
A.~Nedic and A.~Ozdaglar.
\newblock {\em {Convex Optimization in Signal Processing and Communications}},
  chapter Cooperative distributed multi-agent optimization.
\newblock Eds., Eldar, Y. and Palomar, D., Cambridge University Press, 2008.

\bibitem{NedicSubgradientConsensus}
A.~Nedic and A.~Ozdaglar.
\newblock Distributed subgradient methods for multi-agent optimization.
\newblock {\em Automatic Control, IEEE Transactions on}, 54(1):48 --61, Jan
  2009.

\bibitem{ConsensusDelay}
A.~Nedic and A.~Ozdaglar.
\newblock {Convergence rate for consensus with delays}.
\newblock {\em Journal of Global Optimization}, 47(3):437--456, 2010.

\bibitem{InteriorBook}
Y.~Nesterov and A.~Nemirovskii.
\newblock {\em {Interior-Point Polynomial Algorithms in Convex Programming}}.
\newblock SIAM, 2001.

\bibitem{murray}
R.~Olfati-Saber and R.~M. Murray.
\newblock Consensus problems in networks of agents with switching topology and
  time-delays.
\newblock {\em IEEE Transactions on Automatic Control}, 49(9):1520--1533, 2004.

\bibitem{Consensus2}
A.~Olshevsky and J.~Tsitsiklis.
\newblock {Convergence speed in distributed consensus and averaging}.
\newblock {\em SIAM Journal on Control and Optimization}, 48(1):33--35, 2009.

\bibitem{Srikant}
R.~Srikant.
\newblock {\em {The Mathematics of Internet Congestion Control (Systems and
  Control: Foundations and Applications)}}.
\newblock Birkh\"{a}user Boston, 2004.

\bibitem{johnthes}
J.~N. Tsitsiklis.
\newblock {\em {Problems in Decentralized Decision Making and Computation}}.
\newblock PhD thesis, Dept. of {E}lectrical {E}ngineering and {C}omputer
  {S}cience, {M}assachusetts {I}nstitute of {T}echnology, 1984.

\bibitem{distasyn}
J.~N. Tsitsiklis, D.~P. Bertsekas, and M.~Athans.
\newblock Distributed asynchronous deterministic and stochastic gradient
  optimization algorithms.
\newblock {\em IEEE Transactions on Automatic Control}, 31(9):803--812, 1986.

\bibitem{VargaMatrixIte}
R.~Varga.
\newblock {\em {Matrix Iterative Analysis}}.
\newblock Prentice-Hall, Inc, Englewood Cliffs, NJ, 1965.

\bibitem{Gershgorin}
R.~Varga.
\newblock {\em {Gershgorin and His Circles}}.
\newblock Springer Series in Computational Mathematics, 2004.

\end{thebibliography}

\end{document}